\documentclass[11pt,reqno]{amsart}  
\usepackage{amsmath,amssymb,amsthm}
\theoremstyle{plain}
\newtheorem{theorem}{Theorem}[section]
\newtheorem{lemma}[theorem]{Lemma}

\newtheorem{corollary}[theorem]{Corollary}

\numberwithin{equation}{section}
\allowdisplaybreaks[1]

\theoremstyle{definition}

\newtheorem{remark}[theorem]{Remark}
\newtheorem{example}[theorem]{Example}

\theoremstyle{remark}

\newcommand{\cC}{{\mathcal C}}
\newcommand{\cD}{{\mathcal D}}

\newcommand{\cG}{{\mathcal G}}
\newcommand{\cH}{{\mathcal H}}

\newcommand{\cK}{{\mathcal K}}
\newcommand{\cL}{{\mathcal L}}
\newcommand{\cM}{{\mathcal M}}

\newcommand{\cO}{{\mathcal O}}

\newcommand{\cS}{{\mathcal S}}

\newcommand{\cU}{{\mathcal U}}

\newcommand{\cX}{{\mathcal X}}
\newcommand{\cY}{{\mathcal Y}}

\newcommand{\boldf}{{\mathbf f}}
\newcommand{\bg}{{\mathbf g}}

\newcommand{\by}{{\mathbf y}}
\newcommand{\bz}{{\mathbf z}}

\newcommand{\bA}{{\mathbf A}}

\newcommand{\bC}{{\mathbf C}}

\newcommand{\lam}{\lambda}

\newcommand{\bGamma}{\boldsymbol{\Gamma}}
\newcommand{\sbm}[1]{\left[\begin{smallmatrix} #1
		\end{smallmatrix}\right]}

\begin{document}

\title[Bitangential interpolation]
{The bitangential matrix 
Nevanlinna-Pick interpolation problem revisited}
\author[J.A.\ Ball]{Joseph A. Ball}
\address{Department of Mathematics,
Virginia Tech,
Blacksburg, VA 24061-0123, USA}
\email{joball@math.vt.edu}
\author[V. Bolotnikov]{Vladimir Bolotnikov}
\address{Department of Mathematics,
The College of William and Mary,
Williamsburg VA 23187-8795, USA}
\email{vladi@math.wm.edu}

\dedicatory{Dedicated to Heinz Langer, with respect and admiration}

\begin{abstract}  We revisit four approaches to the BiTangential 
    Operator Argument Nevanlinna-Pick (BTOA-NP) interpolation theorem on the 
    right half plane: (1) the state-space approach of 
    Ball-Gohberg-Rodman, (2) the Fundamental Matrix Inequality 
    approach of the Potapov school, (3) a reproducing kernel space 
    interpretation for the solution criterion, and (4) the 
    Grassmannian/Kre\u{\i}n-space geometry approach of Ball-Helton.  
    These four approaches lead to three distinct solution criteria 
    which therefore must be equivalent to each other.  We give  
    alternative concrete direct proofs of each of these latter 
    equivalences.
    In the final section we show how all the results extend to the 
    case where one seeks to characterize interpolants in the 
    Kre\u{\i}n-Langer generalized Schur class $\cS_{\kappa}$ of 
    meromorphic matrix functions on the right half plane, with the integer 
    $\kappa$ as small as possible.
 \end{abstract}


\maketitle

\section{Introduction}  \label{S:Intro}

The simple-multiplicity case of the BiTangential Nevanlinna-Pick 
(BTNP) Interpolation Problem  over the right half plane $\Pi_{+} = \{ 
z \in {\mathbb C} \colon \operatorname{Re} z > 0\}$ can be formulated 
as follows. Let $\cS^{p \times m}(\Pi_{+})$ denote the Schur class of
$\mathbb C^{p\times m}$-valued functions that are analytic and contractive-valued
on $\Pi_+$:
$$
 \cS^{p \times m}(\Pi_{+}) : = \{ S \colon \Pi_{+} \to {\mathbb C}^{p
 \times m} \colon \| S(\lam) \| \le 1 \text{ for all } \lam \in
 \Pi_{+} \}.
$$
The data set ${\mathfrak D}_{\rm simple}$ for the problem consists of
a collection of the form
\begin{align}
{\mathfrak D}_{\rm simple} =  \{ & z_{i} \in \Pi_{+}, \;  x_{i} \in 
{\mathbb C}^{1 \times p},\; y_{i} \in {\mathbb C}^{1 \times m}\;  
\text{ for }\; i=1,\ldots,N, \notag \\
&w_j\in \Pi_{+}, \;  u_{j} \in {\mathbb C}^{m \times 1},\, v_{j} \in {\mathbb C}^{p
  \times 1}\;  \text{ for } \; j = 1, \dots, N',\notag \\
  & \rho_{ij} \in {\mathbb C} \text{ for } (i,j) \text{ such that }
  z_{i} = w_{j} =: \xi_{ij}\}. \label{babydata}
  \end{align}
  The problem then is to find a function $S\in \cS^{p \times m}(\Pi_{+})$ 
that satisfies the collection of interpolation conditions
\begin{align}
  &  x_{i} S(z_{i}) = y_{i} \;  \text{ for } \; i = 1, \dots, N, 
    \label{inter1'} \\
 &  S(w_{j}) u_{j} = v_{j}\; \text{ for } \; j = 1, \dots, N', 
 \label{inter2'} \\
 & x_{i} S'(\xi_{ij}) y_{j} = \rho_{ij} \; \text{ for } (i,j) \; \text{ 
 such that } \; z_{i} = w_{j} =: \xi_{ij}.  \label{inter3'}
\end{align}
We note that the existence of a solution $S$ 
to interpolation conditions \eqref{inter1'}, \eqref{inter2'}, 
\eqref{inter3'} forces the data set \eqref{babydata} to satisfy 
additional compatibility equations; indeed, if $S$ solves 
\eqref{inter1'}--\eqref{inter3'}, and if $(i,j)$ 
is a pair of indices where $z_{i} = w_{j} =: \xi_{ij}$, then the 
quantity $x_{i} S(\xi_{ij}) u_{j}$ can be computed in two ways:
\begin{align*}
    & x_{i} S(\xi_{ij}) u_{j} = (x_{i} S(\xi_{ij})) u_{j} = y_{i} 
    u_{j}, \\
& x_{i} S(\xi_{ij}) u_{j} = x_{i} (S(\xi_{ij}) u_{j}) = x_{i} v_{j}
\end{align*}
forcing the compatibility condition
\begin{equation}   \label{babycomp}
    x_{i} v_{j} = y_{i} u_{j} \; \text{ if } \; z_{i} = w_{j}.
\end{equation}
Moreover, there is no loss of generality in assuming that each row 
vector $x_{i}$ and each column vector $u_{j}$ in \eqref{babydata} is nonzero;
if $x_{i}= 0$ for some $i$, existence of a solution $S$ then forces also that 
$y_{i}=0$ and then the interpolation condition $x_{i} S_{i}(z_{i}) = 
y_{i}$ collapses to $0 = 0$ and can be discarded, with a similar 
analysis in case some $u_{j} = 0$. 

\smallskip

The following result gives the precise solution criterion.  The 
result actually holds even without the normalization conditions on the data 
set discussed in the previous paragraph.

\begin{theorem}  \label{T:BT-NP}
{\rm (See \cite[Section 4]{LA} for the case where $z_{i} \ne w_{j}$ for all $i,j$).} 
Given a data set ${\mathfrak  D}_{\rm simple}$ as in \eqref{babydata}, there 
exists a solution $S$ of the associated problem {\rm BTNP} if and only if the 
associated Pick matrix
   \begin{equation}  \label{Pick-simple}
       P_{{\mathfrak D}_{\rm simple}} : = \begin{bmatrix} 
    P_{11} & P_{12} \\ P_{12}^{*} & P_{22} \end{bmatrix}
    \end{equation}
    with entries given by
    \begin{align*}
&	[P_{11}]_{ij} = \frac{x_{i} x_{j}^{*} - y_{i} 
y_{j}^{*}}{z_{i} + \overline{z}_{j}} \; \text{ for } \; 1 \le i,j \le N, \\
&  [P_{12}]_{ij} = \left\{\begin{array}{cl}
 {\displaystyle\frac{x_{i}v_{j} - y_{i}u_{j}}{ w_{j}-z_{i}}} & \text{ if } \; 
 z_{i} \ne w_{j},  \\ \rho_{i j} & \text{ if } z_{i} = w_{j}, \end{array}\right. 
 \text{ for } 1  \le i \le N, 1 \le j \le N', \\
& [ P_{22}]_{ij} = \frac{ u_{i}^{*} u_{j} - 
v_{i}^{*} v_{j}}{  \overline{w}_{i} +  w_{j} } \; \text{ for } \; 1 \le 
i, j \le N',
\end{align*}
is positive semidefinite.
\end{theorem}

Given a data set ${\mathfrak D}_{\rm simple}$ as above, 
it is convenient to repackage it in a more aggregate form 
as follows (see \cite{BGR}). With data as in  \eqref{babydata}, 
form the septet of matrices
$(Z,X,Y, W, U, V, \Gamma)$ where:
\begin{align}
 & Z = \begin{bmatrix} z_{1} & & 0 \\ & \ddots & \\ 0 & & z_{N} 
\end{bmatrix}, \; \; 
 X = \begin{bmatrix} x_{1} \\ \vdots \\ x_{N} \end{bmatrix}, \; \;  
 Y = \begin{bmatrix} y_{1} \\ \vdots \\  y_{N} \end{bmatrix}, \notag \\
& W = \begin{bmatrix} w_{1} & & 0 \\ & \ddots & \\0 & & w_{N'} 
\end{bmatrix}, \; \; 
U = \begin{bmatrix} u_{1} & \cdots & u_{N'} \end{bmatrix}, \; \;
V = \begin{bmatrix} v_{1} & \cdots & v_{N'} \end{bmatrix}, \notag \\
& \Gamma = \left[ \gamma_{ij} \right]_{i=1,\ldots N}^
{j=1,\ldots,N^\prime} \; \text{ where } \; 
\gamma_{ij} =  \left\{\begin{array}{cl}
 {\displaystyle\frac{x_{i}v_{j} - y_{i}u_{j}}{ w_{j}-z_{i}}} & \text{ if } \;
 z_{i} \ne w_{j},  \\ \rho_{i j} & \text{ if } z_{i} = w_{j}. 
\end{array}\right.
\label{babydata'}
\end{align}
Note that the compatibility condition \eqref{babycomp} translates to 
the fact that $\Gamma$ satisfies the Sylvester equation
$$
\Gamma W-Z\Gamma   = \begin{bmatrix} X & -Y \end{bmatrix} 
    \begin{bmatrix} V \\ U \end{bmatrix}.
$$
The normalization requirements ($x_{i} \ne 0$ for all $i$ and $u_{j} 
\ne 0$ for all $j$ together with $z_{1}, \dots, z_{N}$ all distinct 
and $w_{1}, \dots, w_{N'}$ all distinct)  translate to the conditions
$$
 (Z,X) \text{ is controllable}, \quad (U, W) \text{ is observable.}
$$
Then it is not hard to see that the interpolation conditions 
\eqref{inter1'}, \eqref{inter2'}, \eqref{inter3'} can be written in 
the more aggregate form 
 \begin{align} 
	& \sum_{z_{0} \in \sigma(Z)} {\rm Res}_{\lambda = z_{0}} (\lam I 
	- Z)^{-1} X  S(\lam) = Y, \label{resint1} \\
& \sum_{z_{0} \in \sigma(W)} {\rm Res}_{\lambda = z_{0}} S(\lam) U (\lam 
I - W)^{-1} = V, \label{resint2} \\
& \sum_{z_{0} \in \sigma(Z)\cup \sigma(W)} {\rm Res}_{\lam = z_{0}} 
(\lam I - Z)^{-1} 
X S(\lam) U (\lam I - W)^{-1} = \Gamma. \label{resint3}
\end{align}

Suppose that $(Z,X)$ is any controllable input pair and that $(U,W)$ 
is an observable output pair.  Assume in addition that $\sigma(Z) 
\cup \sigma(W) \subset \Pi_{+}$ and that $S$ is an analytic matrix 
function (of appropriate size) on $\Pi_{+}$.  We define the 
\textbf{Left-Tangential Operator Argument (LTOA) point evaluation} 
$(XS)^{\wedge L}(Z)$ of $S$ at $Z$ in left direction $X$ by
$$
  (X S)^{\wedge L}(Z) = \sum_{z_{0}\in\sigma(Z)} {\rm Res}_{\lam = z_{0}}
  (\lam I - Z)^{-1} X  S(\lam).
$$
Similarly we define the \textbf{Right-Tangential Operator Argument 
(RTOA) point evaluation} $(S U)^{\wedge R}(W)$ of $S$ at $W$ in right 
direction $U$ by
$$
  (S U)^{\wedge R}(W) = \sum_{z_{0} \in \sigma(W)} {\rm Res}_{\lam = z_{0}}
  S(\lam) U (\lam I - W)^{-1}.
$$
Finally the \textbf{BiTangential Operator Argument (BTOA) point 
evaluation} $(X S U)^{\wedge L,R}(Z, W)$ of $S$ at left argument $Z$ 
and right argument $W$ in left direction $X$ and right direction $U$ 
is given by
$$
(X S U)^{\wedge L,R}(Z, W) =
\sum_{z_{0} \in \sigma(Z)\cup \sigma(W)} {\rm Res}_{\lam = z_{0}}
(\lam I - Z)^{-1} X S(\lam) U (\lam I - W)^{-1}.
$$
With this condensed notation, we write the interpolation conditions 
\eqref{resint1}, \eqref{resint2}, \eqref{resint3} simply as
\begin{align} 
    & (X S)^{\wedge L}(Z) = Y, \label{BTOAint1} \\
    & (S U)^{\wedge R} (W) = V, \label{BTOAint2} \\
    & (X S U)^{\wedge L,R}(Z,W) = \Gamma. \label{BTOAint3}
\end{align}
Let us say that the data set 
\begin{equation}   \label{data}
    {\mathfrak D} = (Z,X,Y; U,V,W; \Gamma)
\end{equation}
is a {\em $\Pi_{+}$-admissible BiTangential Operator Argument (BTOA)
interpolation data set} 
if the following conditions hold:
\begin{enumerate}
    \item Both $Z$ and  $W$ have spectrum inside $\Pi_{+}$:  
    $\sigma(Z) \cup \sigma(W) \subset \Pi_{+}$.
    \item $(Z,X)$ is controllable and $(U,W)$ is observable.
    \item $\Gamma$ satisfies the Sylvester equation
 \begin{equation}   \label{Sylvester'}
    \Gamma W - Z \Gamma = X V - Y U.
 \end{equation}
  \end{enumerate}
 Then it makes sense to consider the collection of interpolation 
 conditions \eqref{BTOAint1}, \eqref{BTOAint2}, \eqref{BTOAint3} for 
 {\em any} $\Pi_{+}$-admissible BTOA interpolation data set $(Z,X,Y; 
 U,V,W; \Gamma)$.  It can be shown that these interpolation 
 conditions can be expressed equivalently as a set of 
 higher-order versions of the  interpolation conditions 
 \eqref{inter1'}, \eqref{inter2'}, \eqref{inter3'}
 (see \cite[Theorem 16.8.1]{BGR}), as well as a representation of $S$ in 
 the so-called Model-Matching form (see \cite[Theorem 16.9.3]{BGR}, 
 \cite{Francis})
 $$
   S(\lam) = T_{1}(\lam) + T_{2}(\lam) Q(\lam) T_{3}(\lam),
 $$
 where $T_{1}$, $T_{2}$, $T_{3}$ are rational matrix functions 
 analytic on $\Pi_{+}$ with $T_{2}$ and $T_{3}$ square and analytic 
 and invertible along the imaginary line, and where $Q$ is a 
 free-parameter matrix function analytic on all of $\Pi_{+}$. 
 
\smallskip

 It is interesting to note that the Sylvester equation 
 \eqref{Sylvester'} is still necessary for the existence of a 
 $p \times m$-matrix function $S$ analytic on $\Pi_{+}$ satisfying the 
 BTOA interpolation conditions \eqref{BTOAint1}, \eqref{BTOAint2}, 
 \eqref{BTOAint3}.  Indeed, note that
  \begin{align*}
 &   \left( (\lam I - Z)^{-1} X S(\lam) 
  U (\lam I - W)^{-1}\right) W -
  Z \left( (\lam I - Z)^{-1} X S(\lam) U (\lam I - W)^{-1} \right) \\
  & = (\lam I - Z)^{-1} X S(\lam) 
  U (\lam I - W)^{-1} (W - \lambda I + \lam I) \\
  &  \quad \quad + (\lam I - Z - \lam I)
  (\lam I - Z)^{-1} X S(\lam) U (\lam I - W)^{-1}  \\
  & = -(\lam I - Z)^{-1} X S(\lam)  U  +
  \lam \cdot  (\lam I - Z)^{-1} X S(\lam)  U (\lam I - W)^{-1} \\
  & \quad \quad + X S(\lam) U (\lam I - W)^{-1} - 
  \lam \cdot  (\lam I - Z)^{-1} X S(\lam)  U (\lam I - W)^{-1} \\
  & =  -(\lam I - Z)^{-1} X S(\lam)  U  + X S(\lam) U (\lam I - 
  W)^{-1}.
 \end{align*}
 If we now take the sum of the residues of the first and last 
 expression in this chain of equalities over points $z_{0} \in 
 \Pi_{+}$ and use the interpolation conditions 
 \eqref{resint1}--\eqref{resint3}, we arrive at
 $$
  \Gamma W - Z \Gamma = -Y U + X V
 $$ 
 and the Sylvester equation \eqref{Sylvester'} follows.
 
\smallskip

 We now pose the \textbf{BiTangential Operator Argument Nevanlinna-Pick 
 (BTOA-NP) Interpolation Problem}:  Given a $\Pi_{+}$-admissible BTOA 
 interpolation data set \eqref{data}, find $S$ in the 
 matrix Schur class over the right half plane $\cS^{p \times 
 m}(\Pi_{+})$ which satisfies the BTOA interpolation conditions 
 \eqref{BTOAint1}, \eqref{BTOAint2}, \eqref{BTOAint3}.  
 
\smallskip

 Before formulating the solution, we need some additional 
 notation.  Given a $\Pi_{+}$-admissible BTOA interpolation data set 
 \eqref{data},  introduce two additional matrices $\Gamma_{L}$ and 
 $\Gamma_{R}$ as the unique solutions of the respective Lyapunov 
 equations
 \begin{align}
& \Gamma_{L} Z^{*} + Z \Gamma_{L} = X X^{*} - Y Y^{*}, \label{GammaL}\\
& \Gamma_{R} W + W^{*} \Gamma_{R} = U^{*} U - V^{*}V.      \label{GammaR} 
 \end{align}
 We define the  BTOA-Pick matrix $\bGamma_{\mathfrak D}$ associated 
 with the data set \eqref{data} by
 \begin{equation}   \label{GammaData}
     \bGamma_{\mathfrak D} = \begin{bmatrix} \Gamma_{L} & \Gamma
     \\ \Gamma^{*} & \Gamma_{R} \end{bmatrix}.
 \end{equation}
 The following is the canonical generalization of Theorem \ref{T:BT-NP} 
 to this more 
 general situation.
 
 \begin{theorem}  \label{T:BTOA-NP}  Suppose that 
 $$
 {\mathfrak D}  = ( X,Y,Z; U,V, W; \Gamma)
 $$
 is a $\Pi_{+}$-admissible BTOA interpolation data set.   Then there 
 exists a solution $S \in \cS^{p \times m}(\Pi_{+})$ of the BTOA-NP 
 interpolation problem associated with data set ${\mathfrak D}$ if 
 and only if the associated  BTOA-Pick matrix $\bGamma_{\mathfrak D}$
 defined by \eqref{GammaData} is positive semidefinite.
 
 In case $\Gamma_{\mathfrak D}$ is strictly positive definite 
 ($\Gamma_{\mathfrak D} \succ 0$), the set of all solutions is 
 parametrized as follows.  Define a $(p+m) \times (p+m)$-matrix 
 function 
 $$
 \Theta(\lam) =  \begin{bmatrix} \Theta_{11}(\lam) & 
 \Theta_{12}(\lam) \\ \Theta_{21}(\lam) & \Theta_{22}(\lam) 
\end{bmatrix}
$$ via
\begin{equation}  \label{Theta}
\Theta(\lam) = \begin{bmatrix} I_{p} & 0 \\ 0 & I_{m} \end{bmatrix} +
\begin{bmatrix}-X^{*} & V  \\ -Y^{*} & U  \end{bmatrix}
    \begin{bmatrix} \lam I + Z^{*} & 0 \\ 0 & \lam I - W 
    \end{bmatrix}^{-1} \bGamma_{\mathfrak D}^{-1}  
\begin{bmatrix} X & -Y \\-V ^{*} & U^{*}  \end{bmatrix}. 
\end{equation}
Then $S$ is a solution of the BTOA-NP interpolation problem if and 
only if $S$ has a representation as
\begin{equation}   \label{LFTparam}
S(\lambda) = (\Theta_{11}(\lam) G(\lam) + \Theta_{12}(\lam))
( \Theta_{21}(\lam) G(\lam) + \Theta_{22}(\lam))^{-1}
\end{equation}
where $G$ is a free-parameter function in the 
Schur class $\cS^{p \times m}(\Pi_{+})$.
\end{theorem}
Note that the first part of Theorem \ref{T:BTOA-NP} for the special 
case where the data set ${\mathfrak D}$ has the  
form \eqref{babydata'} coming from the data set \eqref{babydata} for a 
BT-NP problem amounts to the content of Theorem \ref{T:BT-NP}.

\smallskip

The BTOA-NP interpolation problem and closely related problems have 
been studied and analyzed using a variety of methodologies by  number 
of authors, especially in the 1980s and 1990s, largely inspired by 
connections with the then emerging $H^{\infty}$-control theory (see 
\cite{Francis}).  We mention in 
particular the Schur-algorithm approach in \cite{LA, DD, AD}, the method of 
Fundamental Matrix Inequalities by the 
Potapov school (see e.g., \cite{KP})) and the related formalism of the 
Abstract Interpolation Problem of Katsnelson-Kheifets-Yuditskii 
(see \cite{KKY, Kh}), the Commutant Lifting approach of 
Foias-Frazho-Gohberg-Kaashoek (see \cite{FF, FFGK}, and the 
Reproducing Kernel approach of Dym and collaborators (see \cite{Dym, 
DymOT143}).  Our focus here is to revisit two other approaches:  
(1) the Grassmannian/Kre\u{\i}n-space-geometry 
approach of Ball-Helton \cite{BH}, and (2) the state-space implementation 
of this approach due to Ball-Gohberg-Rodman (\cite{BGR}).  The first 
(Grassmannian) approach  relies on Kre\u{\i}n-space geometry to arrive at the 
existence of a solution; the analysis is constructive only after one 
introduces bases to coordinatize various subspaces and operators. The 
second (state-space) approach has the same starting point as the first 
(encoding the problem in terms of the graph of the sought-after solution rather 
than in terms of the solution itself), but finds state-space 
coordinates in which to coordinatize the $J$-inner function 
parametrizing the set of solutions and then verifies the 
linear-fractional parametrization by making use of intrinsic 
properties of $J$-inner functions together with an explicit 
winding-number argument, thereby bypassing any appeal to general 
results from Kre\u{\i}n-space geometry.  This second approach 
proved to be more accessible to users (e.g., engineers) who were not comfortable with 
the general theory of Kre\u{\i}n spaces.

\smallskip

It turns out that the solution criterion  $\bGamma_{\mathfrak D}\succeq 0$
arises more naturally in the second (state-space) approach.  
Furthermore, when $\bGamma_{\mathfrak D} \succ 0$ 
($\bGamma_{\mathfrak D}$ is strictly positive 
definite), one gets a linear-fractional parametrization for the set 
of all Schur-class solutions of the interpolation conditions. The 
matrix function $\Theta$ generating the linear-fractional map also 
generates a matrix kernel function $K_{\Theta,J}$ which is a positive 
kernel exactly when $\bGamma_{\mathfrak D} \succ 0$.  We can then 
view the fact that the associated reproducing kernel space 
$\cH(K_{\Theta,J})$ is a Hilbert space as also a solution criterion 
for the BTOA-NP interpolation problem in the nondegenerate case.

\smallskip

In the first (Grassmannian/Kre\u{\i}n-space-geometry) 
approach, on the other hand, the immediate solution criterion is in terms of the 
positivity of a certain finite-dimensional subspace  $(\cM_{\mathfrak 
D}^{[\perp \cK]})_{0}$ of a Kre\u{\i}n space constructed 
from the interpolation data ${\mathfrak D}$.  In the Left Tangential 
case, one can identify $\bGamma_{\mathfrak D}$ as the Kre\u{\i}n-space 
gramian matrix with respect to a natural basis for  
$(\cM_{\mathfrak D}^{[\perp \cK]})_{0}$, thereby confirming directly 
the equivalence of the two seemingly distinct solution criteria.  For the general 
BiTangential case, the connection between $\bGamma_{\mathfrak D}$ and  
$(\cM_{\mathfrak D}^{[\perp \cK]})_{0}$ is not so direct, but nevertheless, 
using ideas from \cite{BHMesa}, we present here a direct proof as to why 
$\bGamma_{\mathfrak D} \succeq 0$ is equivalent to Kre\u{\i}n-space positivity 
of $(\cM_{\mathfrak D}^{[\perp \cK]})_{0}$ which is interesting in its own right. 
Along the way, we also show how the Fundamental Matrix Inequality approach 
to interpolation of the Potapov school \cite{KP} can be incorporated into this 
BTOA-interpolation formalism to give an alternative derivation of the linear-fractional 
parametrization which also bypasses the winding-number argument, at 
least for the classical Schur-class setting.  We also sketch how all 
the results extend to the more general problem where one seeks 
solutions of the BTOA interpolation conditions \eqref{BTOAint1}--\eqref{BTOAint3} 
in the Kre\u{\i}n-Langer generalized Schur class $\cS_{\kappa}^{p \times m}(\Pi_{+})$ 
with the integer $\kappa$ as small as possible.

\smallskip

The plan of the paper is as follows.  In Section \ref{S:proof1} we 
sketch the ideas of the second (state-space) approach, with the 
Fundamental Matrix Inequality approach and the reproducing-kernel 
interpretation dealt with in succeeding subsections.  In Section 
\ref{S:proof2} we sketch the somewhat more involved ideas behind the 
first (Grassmannian/Kre\u{\i}n-space-geometry) approach.  In Section 
\ref{S:synthesis} we identify the connections between the two 
approaches and in particular show directly that the two solution 
criteria are indeed equivalent.  In the final Section 
\ref{S:negsquares} we indicate how the setup extends to interpolation 
problems for the generalized Schur class $\cS_{\kappa}^{p \times 
m}(\Pi_{+})$.

\section{The state-space approach to the BTOA-NP interpolation problem}  \label{S:proof1}

In this section we sketch the analytic proof of Theorem  \ref{T:BTOA-NP} from \cite{BGR}.
For $\cU$ and $\cY$ Hilbert spaces, we let $\cL(\cU, \cY)$ denote the
space of bounded linear operators mapping $\cU$ into $\cY$,
abbreviated to $\cL(\cU)$ in case $\cU = \cY$.  We then define the
operator-valued version of the Schur class $\cS_{\Omega}(\cU, \cY)$
to consist of holomorphic functions $S$ on $\Omega$
with values equal to contraction operators between $\cU$ and $\cY$. 

\smallskip

We first recall some standard facts concerning positive kernels and 
reproducing kernel Hilbert spaces (see e.g., \cite{BB-HOT}).
Given a point-set $\Omega$  and coefficient Hilbert space $\cY$ along 
with a function $K \colon \Omega \times \Omega \to \cL(\cY)$, we say 
that $K$ is a {\em positive kernel on $\Omega$} if
\begin{equation}   \label{posker}
    \sum_{i,j=1}^{N}\langle K(\omega_{i}, \omega_{j}) y_{j}, y_{i} 
    \rangle_{\cY} \ge 0
\end{equation}
for any collection of $N$ points $\omega_{1}, \dots, \omega_{N} \in 
\Omega$ and vectors $y_{1}, \dots, y_{N} \in \cY$ with arbitrary 
$N\ge 1$. It is well known  that 
the following are equivalent:
\begin{enumerate}
    \item $K$ is a positive kernel on $\Omega$.
    \item $K$ is the {\em reproducing kernel} for a {\em reproducing kernel 
    Hilbert space} $\cH(K)$ consisting of  functions $f\colon \Omega 
    \to \cY$ such that, for each $\omega \in \Omega$ and $y \in \cY$ the 
	function $k_{\omega,y} \colon \Omega \to \cY$ defined by
\begin{equation}  \label{RKHSa}
k_{\omega,y}(\omega') = K(\omega', \omega)y
\end{equation}
is in $\cH(\Omega)$ and has the reproducing property: for each $f \in 
\cH(K)$, 
\begin{equation}   \label{RKHSb}
\langle f, k_{\omega,y}\rangle_{\cH(K)} = \langle f(\omega), y 
\rangle_{\cY}.
\end{equation}
\item $K$ has a {\em Kolmogorov decomposition}:  there is a Hilbert 
space $\cX$ and a function $H \colon \Omega \to \cL(\cX, \cY)$ so that
\begin{equation}   \label{Kolmogorov}
    K(\omega', \omega) = H(\omega') H(\omega)^{*}.
\end{equation}
\end{enumerate}

\begin{proof}[Proof of Theorem \ref{T:BTOA-NP}]  We first illustrate the 
    proof of necessity for the easier 
     simple-multiplicity case as formulated 
    in Theorem \ref{T:BT-NP}; the idea is essentially the same as the necessity proof 
    in Limebeer-Anderson \cite{LA}.

\smallskip
 
It is well known that a Schur-class function $F \in\cS_{\mathbb D}(\cU,\cY)$ 
on the unit disk can be characterized not only by the positivity of 
the de Branges-Rovnyak kernel 
$$
  {\mathbf K}_{F}(\lam, w) = \frac{ I - F(\lam) F(w)^{*}}{1 - z 
  \overline{\zeta}}
$$
on the unit disk ${\mathbb D}$, but also by positivity of the
block $2 \times 2$-matrix kernel defined on $({\mathbb D} \times 
    {\mathbb D}) \times ({\mathbb D} \times {\mathbb D})$ by
$$
   \widetilde  {\mathbf K}_{F}(\lam, \lam_{*}; w, w_{*}): = \begin{bmatrix} 
   {\displaystyle\frac{ I -  F(\lam) 
    F(w)^{*}}{1 - \lam \overline{w}}} & {\displaystyle\frac{F(\lam) - 
    F(\overline{w}_{*})}{\lam - \overline{w}_{*}}}\vspace{1mm}\\
   {\displaystyle\frac{F(\overline{\lam}_{*})^{*} - 
   F(w)^{*}}{\lam_{*} - \overline{w}}} &
   {\displaystyle\frac{I - F(\overline{\lam}_{*})^{*} F(\overline{w}_{*})}{1 - \lam_{*} 
   \overline{w}_{*}}} 
\end{bmatrix}.
$$
Making use of the linear-fractional change of variable from ${\mathbb D}$ to $\Pi_{+}$
$$
    \lam \in {\mathbb D} \mapsto z = \frac{ 1+ \lam}{1- \lam} \in \Pi_{+}
$$
with inverse given by
$$
  z \in \Pi_{+} \mapsto \lam = \frac{z -1}{z +1} \in {\mathbb D},
$$
it is easily seen that the function $S$ defined on $\Pi_{+}$ is in the 
Schur class $\cS_{\Pi_+}(\cU,\cY)$ over $\Pi_{+}$ if and only if, not 
only the $\Pi_{+}$-de Branges-Rovnyak kernel
\begin{equation}   \label{dBRkerPi+}
K_{S}(z, \zeta) = \frac{I - S(z) S(\zeta)^{*}}{z + \overline{\zeta}}
\end{equation}
is a positive kernel on $\Pi_{+}$, but also the ($2 \times 2$)-block de 
Branges-Rovnyak kernel
\begin{equation}  \label{ker22}
  {\mathbf K}_{S}(z, z_{*}; \zeta, \zeta_{*}): = \begin{bmatrix} 
{\displaystyle\frac{ I - S(z) S(\zeta)^{*}}{ z + \overline{\zeta}}} &
{\displaystyle\frac{S(z) - S(\overline{\zeta}_{*})}{z - \overline{\zeta}_{*}}} 
\vspace{1mm}\\  {\displaystyle\frac{ S(\overline{z}_{*})^{*} - S(\zeta)^{*}}
{ z_{*} - \overline{\zeta}}} &
{\displaystyle\frac{I - S(\overline{z}_{*})^{*} S(\overline{\zeta}_{*})}
{ z_{*} + \overline{\zeta}_{*}}}
\end{bmatrix}
\end{equation}
is a positive kernel on $(\Pi_{+}\times\Pi_+) \times (\Pi_+\times \Pi_{+})$.  Specifying the 
latter kernel at the points $(z, z_{*}), (\zeta, \zeta_{*}) \in \Pi_{+} \times \Pi_{+}$ 
where  $z, \zeta = z_{1}, \dots, z_{N}$ and $z_{*}, \zeta_{*} = \overline{w}_{1}, \dots, 
\overline{w}_{N'}$, leads to the conclusion that the block matrix
\begin{equation}  \label{Pick-pre}
\begin{bmatrix} \left[ {\displaystyle\frac{I - S(z_{i}) 
    S(z_{j})^{*}}{z_{i} + \overline{z}_{j}}} \right]  &
   \left[ {\displaystyle\frac{S(z_{i}) - S(w_{j'})}{ z_{i} - w_{j'}}} \right] 
   \vspace{1mm}\\
   \left[ {\displaystyle\frac{S(w_{i'})^{*} - S(z_{j})^{*}}{\overline{w}_{j'} - 
   \overline{z}_{i}}} \right] & \left[ {\displaystyle\frac{ I - S(w_{i'})^{*} 
   S(w_{j'})}{\overline{w}_{i'} + w_{j'}}} \right] \end{bmatrix},
\end{equation}
where $1 \le i,j \le N$ and $1 \le i', j' \le N'$, is positive 
semidefinite. Note that the entry  $\frac{S(z_{i}) - S(w_{j'})}{ 
z_{i} - w_{j'}}$ in the upper right corner is to be interpreted as 
$S'(\xi_{i j'})$ in case $z_{i} = w_{j'} =: \xi_{i,j'}$ for some pair 
of indices $i,j'$.  

\smallskip

Suppose now that $S \in \cS_{\Pi_{+}}(\cU, \cY)$ is a 
Schur-class solution of the interpolation conditions \eqref{inter1'}, 
\eqref{inter2'}, \eqref{inter3'}.  When we multiply the matrix 
\eqref{Pick-pre} on the left by the block diagonal matrix
$$
  \begin{bmatrix} \operatorname{diag}_{1 \le i \le N}[ x_{i}] & 0 \\ 
      0 & \operatorname{diag}_{1 \le i' \le N'} [ u_{i'}^{*}] 
  \end{bmatrix}
$$
and on the right by its adjoint, we arrive at the matrix 
$P_{{\mathfrak D}_{\rm simple}}$ \label{Psimple}.  This verifies the 
necessity of the condition $P_{{\mathfrak D}_{\rm simple}} \succeq 0$ 
for a solution of the BT-NP interpolation problem to exist.

\smallskip

We now consider the proof of necessity for the general case.
We note that the proof of necessity in 
    \cite{BGR} handles explicitly only the case where the Pick matrix 
    is invertible and relies on use of the matrix-function $\Theta$ 
    generating the linear-fractional parametrization (see 
    \eqref{oct1} below).  We give a proof 
    here which proceeds directly from the BTOA-interpolation formulation; it 
   amounts to a specialization of the proof of necessity for the more 
   complicated multivariable interpolation problems in the Schur-Agler 
   class done in \cite{BB05}.      

\smallskip

The starting point is the observation that the positivity of the 
kernel ${\mathbf K}_{S}$ implies that it has a Kolmogorov 
decomposition \eqref{Kolmogorov};  furthermore the extra structure of the arguments of 
the kernel ${\mathbf K}_{S}$ implies  that the Kolmogorov 
decomposition can be taken to have the form
\begin{equation}   \label{KS2}
 {\mathbf K}_{S}(z, z_{*}; \zeta, \zeta_{*}) =
 \begin{bmatrix} H(z) \\ G(z_{*})^* \end{bmatrix}
     \begin{bmatrix} H(\zeta)^{*} & G(\zeta_{*}) \end{bmatrix}
\end{equation}
for holomorphic operator functions
$$  
H \colon \Pi_{+} \to \cL(\cX, \cY), \quad G \colon \Pi_{+} \to \cL(\cU, \cX).
$$
In the present matricial setting of $\mathbb C^{p\times m}$-valued functions, 
the spaces $\cU$ and $\cY$ are finite dimensional and can be identified
with $\mathbb C^p$ and $\mathbb C^m$, respectively.  
In particular we read off the identity
\begin{equation}\label{Gamma'12}
 H(z)G(\zeta) =  \frac{ S(z) - S(\zeta)}{z - \zeta}
 \end{equation}
with appropriate interpretation in case $z = \zeta$. Observe that 
for a fixed $\zeta\in \Pi_+\backslash\sigma(Z)$, we have from \eqref{Gamma'12}
\begin{align}
(XH)^{\wedge L}(Z)\cdot G(\zeta)
&=\sum_{z_{0} \in \sigma(Z)} {\rm Res}_{\lambda 
= z_{0}} (\lam I- Z)^{-1} X H(\lam)G(\zeta)\notag\\ 
&=\sum_{z_{0} \in \sigma(Z)} {\rm Res}_{\lambda = z_{0}} (\lam I- Z)^{-1} X  
\frac{ S(\lam) - S(\zeta)}{\lam - \zeta}\notag\\
&=\sum_{z_{0} \in \sigma(Z)} {\rm Res}_{\lambda = z_{0}} (\lam I- Z)^{-1}
\frac{Y - XS(\zeta)}{\lam - \zeta}\notag\\
&=(\zeta I-Z)^{-1}(XS(\zeta)-Y)\label{aug1}
\end{align}
where we used the interpolation condition \eqref{resint1} for the third equality.
Since the function $g(\zeta)=(\zeta I-Z)^{-1}YU(\zeta I-W)^{-1}$ 
satisfies an estimate of the form
$ \|g(\zeta \| \le \frac{M}{ |\zeta|^{2}}\; $ as $\; |\zeta| \to \infty$, it follows that
$$
\sum_{z_{0} \in \sigma(Z)\cup\sigma(W)}{\rm Res}_{\zeta= z_{0}}
(\zeta I-Z)^{-1}YU(\zeta I-W)^{-1}=0.
$$
On the other hand, due to condition \eqref{resint1}, the function on the 
right hand side of \eqref{aug1}
is analytic (in $\zeta$) on $\Pi_+$, so that 
\begin{align*}
&\sum_{z_{0} \in \sigma(W)} {\rm Res}_{\zeta=z_0}
(\zeta I-Z)^{-1}(XS(\zeta)-Y)U(\zeta I-W)^{-1}\\
&=\sum_{z_{0} \in \sigma(Z)\cup\sigma(W)} 
{\rm Res}_{\zeta=z_0}(\zeta I-Z)^{-1}(XS(\zeta)-Y)U(\zeta I-W)^{-1}.
\end{align*}
We now apply the {\bf RTOA} point evaluation to both sides in \eqref{aug1} and
make use of the two last equalities and the interpolation condition \eqref{resint3}:
\begin{align}
&(XH)^{\wedge L}(Z)(GU)^{\wedge R}(W)\notag \\
&=\sum_{z_{0} \in \sigma(W)} {\rm Res}_{\zeta=z_0}(\zeta I-Z)^{-1}
(XS(\zeta)-Y)U(\zeta I-W)^{-1}\notag\\
&=\sum_{z_{0} \in \sigma(Z)\cup\sigma(W)} {\rm Res}_{\zeta=z_0}
(\zeta I-Z)^{-1}(XS(\zeta)-Y)U(\zeta I-W)^{-1}\notag\\
&=\sum_{z_{0} \in \sigma(Z)\cup\sigma(W)} 
{\rm Res}_{\zeta=z_0}(\zeta I-Z)^{-1}XS(\zeta)U(\zeta I-W)^{-1}=\Gamma.
\label{aug2}
\end{align}
Let us now introduce the block $2 \times 2$-matrix ${\mathbf \Gamma}'_{\mathfrak 
D}$ by
\begin{equation} 
  \bGamma'_{\mathfrak D} = \begin{bmatrix} (XH)^{\wedge L}(Z) 
  \\ (GU)^{\wedge R}(W)^{*} \end{bmatrix}
  \begin{bmatrix} ((XH)^{\wedge L}(Z) )^{*} & (GU)^{\wedge R}(W) 
  \end{bmatrix}.
\label{aug4}
\end{equation}
We claim that $\bGamma'_{\mathfrak D} = \bGamma_{\mathfrak D}$.  Note 
that equality of the off-diagonal blocks follows from 
\eqref{aug2}. It remains to show the two equalities
\begin{align}
\Gamma'_{L}&:=(XH)^{\wedge L}(Z) ((XH)^{\wedge L}(Z))^{*}= \Gamma_{L}, \label{verifyGamma}\\
\Gamma'_{R}&:=((G^{*}U)^{\wedge R}(W) )^{*} (G^{*}U)^{\wedge R}(W)=
   \Gamma_{R}.\label{verifyGamma1}
\end{align}
To verify \eqref{verifyGamma}, we note 
that $\Gamma_{L}$ is defined as the unique solution of the Lyapunov 
equation \eqref{GammaL}.  Thus it suffices to verify that $\Gamma'_{L}$
also satisfies \eqref{GammaL}.  Toward this end, the two expressions \eqref{ker22} and 
\eqref{KS2} for ${\mathbf K}_{S}$ give us equality of the 
$(1,1)$-block entries:
$$
  H(z) H(\zeta)^{*} = \frac{ I - S(z) S(\zeta)^{*}}{ z + 
  \overline{\zeta}}
$$
which we prefer to rewrite in the form
\begin{equation}  \label{identity}
 z \cdot H(z) H(\zeta)^{*} + H(z) H(\zeta)^{*} \cdot 
 \overline{\zeta} = I - S(z) S(\zeta)^{*}.
\end{equation}
To avoid confusion, let us introduce the notation $\chi$ for the 
identity function $\chi(z) = z$ on $\Pi_{+}$.  Then it is 
easily verified that
\begin{equation}  \label{identity'}
 ( X \chi \cdot H)^{\wedge L}(Z) = Z (X H)^{\wedge L}(Z).
\end{equation}
Multiplication on the left by $X$ and on the right by $X^{*}$ and 
then plugging in the left operator argument $Z$ for $\lam$ in \eqref{identity} then gives
\begin{align*}
& Z (X H)^{\wedge L}(Z) H(\zeta)^{*}X^{*} + (X H)^{\wedge L}(Z) ( \zeta 
 \cdot H(\zeta)^{*}X^{*} \\
 & \quad = XX^{*} - (X S)^{\wedge L}(Z) S(\zeta)^{*}  = XX^{*} - Y S(\zeta)^{*}X^{*}.
\end{align*}
Replacing the variable $\zeta$ by the operator argument $Z$ and 
applying the adjoint of the identity \eqref{identity'} then brings us to
$$
Z (XH)^{\wedge L}(Z)( (XH)^{\wedge L}(Z))^{*} Z^{*} = 
X X^{*} - Y \left( X S^{\wedge L}(Z) \right)^{*} = X X^{*} - Y Y^{*},
$$
i.e., $\Gamma'_{L}$ satisfies \eqref{GammaL} as wanted.  The proof 
that $\Gamma'_{R}$ (see \eqref{verifyGamma1}) 
satisfies \eqref{GammaR} proceeds in a similar way.
 
\smallskip

For the sufficiency direction, for simplicity we shall assume 
    that $\bGamma_{\mathfrak D}$ is strictly positive definite rather 
    than just positive semidefinite. We then must show that 
    solutions $S$ of the BTOA-NP problem exist and in fact the set of 
    all solutions is given by the linear-fractional parametrization
    \eqref{LFTparam}.  The case where the Pick matrix is 
    positive-semidefinite then follows by perturbing the semidefinite 
    Pick matrix to a definite Pick matrix and using an approximation 
    and normal families argument.  The ideas follow \cite{BGR}.
    
\smallskip

Let us therefore assume that $\Gamma_{\mathfrak D}$ is positive 
definite.  Then we can form the rational matrix function $\Theta$ 
given by \eqref{Theta}.  Let us write $\Theta$ in the more condensed 
form 
\begin{equation}
\Theta(\lam) = I_{p+m} - \bC (\lam I - \bA)^{-1} \bGamma_{\mathfrak 
D}^{-1} \bC^{*} J
\label{oct1}
\end{equation}
where we set
\begin{equation}
\bA = \begin{bmatrix} -Z^{*} & 0 \\ 0 & W \end{bmatrix}, \quad
\bC = \begin{bmatrix} -X^{*} & V \\ -Y^{*} & U \end{bmatrix},\quad 
J = \begin{bmatrix} I_{p} & 0 \\ 0 & -I_{m} \end{bmatrix}.
\label{oct2}
\end{equation}
Recall that $\Gamma_{L}$, $\Gamma_{R}$, $\Gamma$ satisfy the 
Lyapunov/Sylvester equations \eqref{GammaL}, \eqref{GammaR}, 
\eqref{Sylvester'}.  Consequently one can check that 
$\bGamma_{\mathfrak D}$ satisfies the $(2 \times 2)$-block 
Lyapunov/Sylvester  equation
\begin{align*}
  &  \begin{bmatrix} \Gamma_{L} & \Gamma \\ \Gamma^{*} & \Gamma_{R} 
	\end{bmatrix} \begin{bmatrix} -Z^{*} & 0 \\ 0 & W \end{bmatrix}
+ \begin{bmatrix} -Z & 0 \\ 0 & W^{*} \end{bmatrix} \begin{bmatrix} 
\Gamma_{L} & \Gamma \\ \Gamma^{*} & \Gamma_{R} \end{bmatrix} \\
& \quad = \begin{bmatrix} Y Y^{*} - X X^{*} & XV - YU \\ V^{*} X^{*} - U^{*} 
Y^{*} & V^{*} U - V^{*} V \end{bmatrix},
\end{align*}
or, in more succinct form,
\begin{equation}   \label{bigLySyl}
  \bGamma_{\mathfrak D} \bA + \bA^{*} \bGamma_{\mathfrak D} = - \bC^{*} J \bC.
\end{equation}
Using this we compute
\begin{align*}
 &   J - \Theta(\lam) J \Theta(\zeta)^{*}  = \\
&  J - \left( I - \bC (\lam I - \bA)^{-1} \bGamma_{\mathfrak D}^{-1} 
 \bC^{*} J \right) J \left( I - J \bC \bGamma_{\mathfrak D}^{-1} 
 (\overline{\zeta} I - \bA^{*})^{-1} \bC^{*} \right) \\
 & =\bC (\lam I - \bA)^{-1} \bGamma_{\mathfrak D}^{-1} \bC^{*}
 + \bC \bGamma_{\mathfrak D}^{-1} (\overline{\zeta} I -\bA^{*})^{-1} 
 \bC^{*} \\
 & \quad \quad \quad \quad - \bC (\lam I - \bA)^{-1} \bGamma_{\mathfrak 
 D}^{-1} \bC^{*} J \bC \Gamma_{\mathfrak D}^{-1} (\overline{\zeta} I 
 - \bA^{*})^{-1} \bC^{*} \\
 & = \bC (\lam I - A)^{-1} \bGamma_{\mathfrak D}^{-1} \Xi(\lam,\zeta) (\overline{\zeta} I - 
 \bA^{*})^{-1} \bGamma_{\mathfrak D}^{-1} \bC^{*}
 \end{align*}
 where
$$
     \Xi(\lam,\zeta) =(\overline{\zeta} I - \bA^{*}) \bGamma_{\mathfrak D} 
     + \bGamma_{\mathfrak D} (\lam I - \bA) - \bC^{*} J \bC= (\lam + \overline{\zeta}) 
     \bGamma_{\mathfrak D},
$$
 where we used \eqref{bigLySyl} in the last step. We conclude that 
 \begin{equation}
 K_{\Theta,J}(\lam.\zeta):= \frac{J - \Theta(\lam) J \Theta(\zeta)^{*}}
 {\lam + \overline{\zeta}}=  
\bC(\lam I - \bA)^{ -1} \bGamma_{\mathfrak D}^{-1} 
   (\overline{\zeta}I - \bA^{*})^{-1} \bC^{*}.
 \label{ad1}
\end{equation}
 By assumption, $\sigma(Z)\cup\sigma(W)\subset \Pi_+$, so   
the matrix $\bA = \sbm{ -Z^{*} & 0 \\ 0 & W }$ has no 
 eigenvalues on the imaginary line, and hence $\Theta$ is analytic and 
 invertible on $i {\mathbb R}$.  As a consequence of \eqref{ad1}, 
 we see that  $\Theta(\lam)$ is $J$-coisometry for $\lam 
 \in i{\mathbb R}$.  As $J$ is a finite matrix we actually have (see \cite{AI}): 
 \begin{itemize}
     \item {\em for $\lam \in i{\mathbb R}$, $\Theta(\lam)$ is 
     $J$-unitary:}
 \begin{equation}  \label{ThetaJunitary}
     J - \Theta(\lam)^{*} J \Theta(\lam) = J - \Theta(\lam) J 
     \Theta(\lam)^{*} = 0 \; \text{ for } \; \lam \in i{\mathbb R}.
 \end{equation}
  \end{itemize}
 The significance of the assumption that $\bGamma_{\mathfrak D}$ is 
 not only invertible but also positive definite is that
 \begin{itemize}
     \item {\em for $\lam \in \Pi_{+}$ a point of analyticity for 
     $\Theta$, $\Theta(\lam)$ is $J$-bicontractive:}
 \begin{equation} \label{ThetaJcontr}
	J - \Theta(\lam)^{*} J \Theta(\lam) \succeq 0, \quad
	J - \Theta(\lam) J \Theta(\lam)^{*} \succeq 0 \text{ for } 
	\lam \in \Pi_{+}.
\end{equation}
 \end{itemize}
 Here we make use of the fact that $J$-co-contractive is equivalent 
 to $J$-contractive in the matrix case (see \cite{AI}). 
These last two observations have critical consequences.  Again 
   writing out $\Theta$ and $J$ as
$$ \Theta(\lam) = \begin{bmatrix} \Theta_{11}(\lam) & 
\Theta_{12}(\lam) \\ \Theta_{21}(\lam) & \Theta_{22}(\lam) 
\end{bmatrix}, \quad J = \begin{bmatrix} I_{p} & 0 \\ 0 & -I_{m} 
\end{bmatrix},
$$
relations \eqref{ThetaJunitary} and \eqref{ThetaJcontr}) 
give us (with the variable $\lam$ suppressed)
 $$
   \begin{bmatrix} \Theta_{11} \Theta_{11}^{*} - \Theta_{12} 
       \Theta_{12}^{*} & \Theta_{11} \Theta_{21}^{*} - \Theta_{12} 
       \Theta_{22}^{*} \\
  \Theta_{21} \Theta_{11}^{*} - \Theta_{22} \Theta_{12}^{*} &
  \Theta_{21} \Theta_{21}^{*} - \Theta_{22} \Theta_{22}^{*} 
\end{bmatrix} \preceq \begin{bmatrix} I_{p} & 0 \\ 0 & - I_{m} 
\end{bmatrix}
$$
for $\lam$ a point of analyticity of $\Theta$ in $\Pi_{+}$ with equality for 
$\lam$ in $i {\mathbb R} = \partial \Pi_{+}$ (including the point 
at infinity). In particular,
$$
  \Theta_{21} \Theta_{21}^{*} - \Theta_{22} \Theta_{22}^{*} \preceq 
  -I_{m}
$$
or equivalently,
\begin{equation}   \label{Theta22inv}
  \Theta_{21} \Theta_{21}^{*} + I_{m} \preceq \Theta_{22} 
  \Theta_{22}^{*}.
\end{equation}
Hence, $\Theta_{22}(\lam)$ is invertible at all points 
$\lam$ of analyticity in $\Pi_+$, namely, $\Pi_{+} \setminus 
\sigma(W)$, and then, since multiplying 
on the left by $\Theta_{22}^{-1}$ and on the right by its adjoint 
preserves the inequality, we get 
\begin{equation} \label{Theta22invTheta21}
  \Theta_{22}^{-1} \Theta_{21} \Theta_{21}^{*} \Theta_{22}^{*-1} + 
  \Theta_{22}^{-1} \Theta_{22}^{* -1} \preceq I_{m}.
\end{equation}
We conclude: 
\begin{itemize}
    \item {\em $\Theta_{22}^{-1}$ has analytic continuation to a 
contractive $m \times m$-matrix function on all of $\Pi_{+}$ and 
$\Theta_{22}^{-1} \Theta_{21}$ has analytic continuation to an 
analytic $m \times p$-matrix rational function which is pointwise 
strictly contractive on the closed right half 
plane $\overline{\Pi}_{+} = \Pi_{+} \cup i{\mathbb R}$.}
\end{itemize}

It remains to make the connection of $\Theta$ with the BTOA-NP interpolation 
problem.  Let us introduce some additional notation.
For $N$ a positive integer, 
$ H^{2}_{N}(\Pi_{+})$ is short-hand notation for the ${\mathbb 
C}^{N}$-valued Hardy space $H^{2}(\Pi_{+}) \otimes {\mathbb C}^{N}$ 
over the right half plane $\Pi_{+}$.  Similarly $L^{2}_{N}(i 
{\mathbb R}) = L^{2}(i {\mathbb R}) \otimes {\mathbb C}^{N}$ is the 
${\mathbb C}^{N}$-valued $L^{2}$-space over the imaginary line $i {\mathbb R}$.  

\smallskip

It is well known (see e.g.\ \cite{Hoffman}) that the space 
$H^{2}_{N}(\Pi_{+})$ (consisting of analytic functions on $\Pi_{+}$) 
can be identified with 
a subspace of $L^{2}_{N}(i {\mathbb R})$ (consisting of measurable 
functions on $i {\mathbb R}$ defined only almost everywhere with 
respect to linear Lebesgue measure) via the process of taking 
nontangential limits from $\Pi_{+}$ to a point on $i {\mathbb R}$.  
Similarly the Hardy space $H^{2}_{N}(\Pi_{-})$ over the left half 
plane can also be identified with a subspace (still denoted as 
$H^{2}_{N}(\Pi_{-})$) of $L^{2}_{N}(i {\mathbb R})$, and, after these 
identifications, $H^{2}_{N}(\Pi_{-}) = H^{2}_{N}(\Pi_{+})^{\perp}$ 
as subspaces of $L^{2}_{N}(i {\mathbb R})$:
$$
  L^{2}_{N}(i {\mathbb R}) = H^{2}_{N}(\Pi_{+}) \oplus 
  H^{2}_{N}(\Pi_{-}).
$$
We shall use these identifications freely in the discussion to follow.
Given the  $\Pi_{+}$-admissible interpolation data set
\eqref{data},  we define a  subspace of 
$L^{2}_{p+m}(i{\mathbb R})$ by
\begin{align}
    \cM_{\mathfrak D} = &\left \{ \begin{bmatrix} V \\ U  \end{bmatrix} (\lam I - W)^{-1} x + 
    \begin{bmatrix} f(\lam) \\ 
g(\lam) \end{bmatrix}  \colon x \in {\mathbb C}^{n_{W}} \text{ and } 
\begin{bmatrix} f \\ g \end{bmatrix} \in H^{2}_{p+m}(\Pi_{+} )  
    \right. \notag \\
& \left. \text{ such that }  
 \sum_{z_{0} \in \Pi_{+}} {\rm Res}_{\lam = z_{0}} (\lam I - Z)^{-1} 
\begin{bmatrix} X & -Y \end{bmatrix}   \begin{bmatrix} f(\lam) \\ 
    g(\lam) \end{bmatrix}  = \Gamma x \right\}.
    \label{cMrep}
\end{align}
and a subspace of $L^{2}_{m}(i {\mathbb R})$ by 
$$ 
\cM_{{\mathfrak D}, -}=\{ U (\lam I - W)^{-1} x \colon x \in {\mathbb C}^{n_{W}} \}
    \oplus H^{2}_{m}(\Pi_{+}).
$$
Using $\Pi_{+}$-admissibility assumptions  on the data set ${\mathfrak D}$ one can show
(we refer to \cite{BGR} for details, subject to the disclaimer in 
Remark \ref{R:Amaya} below) that 
$$
\cM_{{\mathfrak D}, -}= P_{\sbm{ 0 \\  L^{2}_{m}(i{\mathbb R})}}\cM_{\mathfrak D}.
$$
Furthermore, a variant of the Beurling-Lax Theorem assures us that there is a 
$m \times m$-matrix 
inner function $\psi$ on $\Pi_{+}$ so that
\begin{equation}   \label{cM-rep}
 \cM_{{\mathfrak D},-} = \psi^{-1} \cdot H^{2}_{m}(\Pi_{+}).
\end{equation}
Making use of \cite[Theorem 6.1]{BGR} applied to the null-pole triple 
$(U,W; \emptyset, \emptyset; \emptyset)$ over $\Pi_{+}$,
one can see that such a 
$\psi$ (defined uniquely up to a constant unitary factor on the left) 
is given by the state-space realization formula 
\begin{equation}
\psi(z)=I_m-UP^{-1}(zI+W^*)^{-1}U^*,
\label{sep1}
\end{equation}
where the positive definite matrix $P$ is uniquely defined from the Lyapunov equation 
$PW+W^*P=U^*U$, with $\psi^{-1}$ given by
\begin{equation}   \label{psi-inv}
\psi(z)^{-1}=I_m+U(zI-W)^{-1}P^{-1}U^*,
\end{equation}
i.e., that $(U,W)$ is the right null pair of $\psi$. 
Furthermore, a second application of  \cite[Theorem 6.1]{BGR} 
to the null-pole triple $(\sbm{ V \\ U}, W; Z, \sbm{ X & -Y}; \Gamma)$ 
over $\Pi_{+}$ leads to:
\begin{itemize}
    \item {\em $\cM_{\mathfrak D}$ has the Beurling-Lax-type 
    representation} 
\begin{equation}   \label{cMBLrep}
\cM_{\mathfrak D} = \Theta \cdot 
    H^{2}_{p+m}(\Pi_{+}).
\end{equation}
\end{itemize}
By projecting the identity \eqref{cMBLrep}  onto the bottom component and recalling 
the identity 
\eqref{cM-rep}, we see that
\begin{equation}   \label{Theta22-1}
\begin{bmatrix} \Theta_{21} & \Theta_{22} \end{bmatrix} H^{2}_{p + 
    m}(\Pi_{+}) = \cM_{{\mathfrak D}, -} =
    \psi^{-1} H^{2}_{m}(\Pi_{+}).
 \end{equation} 
 On the other hand, for any $\sbm{ f_{+} \\ f_{-} } \in 
 H^{2}_{p+m}(\Pi_{+})$, we have
 \begin{align}
 \begin{bmatrix} \Theta_{21} & \Theta_{22} \end{bmatrix} 
     \begin{bmatrix} f_{+} \\ f_{-} \end{bmatrix} &=
\Theta_{21} f_{+} + \Theta_{22} f_{-} \notag\\
&= \Theta_{22} 
(\Theta_{22}^{-1} \Theta_{21} f_{+} + f_{-}) \in \Theta_{22} 
H^{2}_{m}(\Pi_{+})\label{Theta22-2},
\end{align}
since $\Theta_{22}^{-1} \Theta_{21}$ is analytic on $\Pi_+$.  Since 
the reverse containment
$$
\Theta_{22}\cdot  H^{2}_{m}(\Pi_{+}) \subset \begin{bmatrix} \Theta_{21} & 
\Theta_{22} \end{bmatrix}\cdot H^{2}_{p+m}(\Pi_{+})
$$
is obvious, we may combine \eqref{Theta22-1} and \eqref{Theta22-2} to conclude that
\begin{equation}   \label{Theta22-3}
    \Theta_{22}\cdot  H^{2}_{m}(\Pi_{+}) = \begin{bmatrix} \Theta_{21} & 
    \Theta_{22} \end{bmatrix} \cdot H^{2}_{p+m}(\Pi_{+}) = \psi^{-1} 
   \cdot H^{2}_{m}(\Pi_{+}).
 \end{equation}
It turns out that the geometry of $\cM_{\mathfrak D}$ encodes the 
interpolation conditions:
\begin{itemize}
    \item {\em An analytic function $S \colon \Pi_{+} \to {\mathbb C}^{p \times m}$ 
    satisfies the interpolation conditions \eqref{BTOAint1}, 
    \eqref{BTOAint2}, \eqref{BTOAint3} if and only if}
 \begin{equation}  \label{sol-rep}
  \begin{bmatrix} S \\ I_m \end{bmatrix} \cdot \cM_{{\mathfrak D}, -}
      \subset \cM_{\mathfrak D}.
 \end{equation}
 \end{itemize}
 
 It remains to put the pieces together to arrive at the 
 linear-fractional parametrization \eqref{LFTparam} for the set of 
 all solutions (and thereby prove that solutions exist).
 Suppose that $S \in \cS^{p \times m}(\Pi_{+})$ satisfies the 
 interpolation conditions \eqref{BTOAint1}, \eqref{BTOAint2}, 
 \eqref{BTOAint3}.  As a consequence of the criterion \eqref{sol-rep}
 combined with \eqref{cM-rep} and \eqref{cMBLrep}, we have
 $$
   \begin{bmatrix} S \\ I_{m} \end{bmatrix} \psi^{-1}  \cdot
       H^{2}_{m}(\Pi_{+}) \subset \begin{bmatrix} \Theta_{11} & 
       \Theta_{12} \\ \Theta_{21} & \Theta_{22} \end{bmatrix}  \cdot
       H^{2}_{p+m}(\Pi_{+}).
 $$
 Hence there must be a $(p+m) \times m$ matrix function
 $\sbm{Q_{1} \\ Q_{2}} \in 
 H^2_{(p+m) \times m}(\Pi_{+})$ so that
  \begin{equation}  \label{Sparam}
 \begin{bmatrix} S \\ I_{m} \end{bmatrix} \psi^{-1} = 
\begin{bmatrix} \Theta_{11} & \Theta_{12} \\ 
     \Theta_{21} & \Theta_{22} \end{bmatrix} \begin{bmatrix} Q_{1}  
     \\ Q_{2} \end{bmatrix}.
\end{equation}
 
We next combine this identity with the $J$-unitary property of \eqref{ThetaJunitary}: 
for the (suppressed) argument $\lam \in i{\mathbb R}$ we have
\begin{align*}
    0 \succeq \psi^{-1 *} (S^{*} S - I) \psi^{-1}& = \psi^{-1 *} 
\begin{bmatrix} Q_{1}^{*} & Q_{2}^{*} \end{bmatrix}
    \Theta^{*} J \Theta \begin{bmatrix} Q_{1} \\ Q_{2} \end{bmatrix} 
    \psi^{-1} \\
    & =\psi^{-1 *} \begin{bmatrix} Q_{1}^{*} & Q_{2}^{*} \end{bmatrix} 
    J \begin{bmatrix} Q_{1} \\ Q_{2} \end{bmatrix} \psi^{-1}\\
    & = \psi^{-1 *} (Q_{1}^{*} Q_{1} - Q_{2}^{*} Q_{2}) \psi^{-1}.
\end{align*}
We conclude that
$$
   \| Q_{1}(\lam) x\|^{2} \le \| Q_{2}(\lam) x \|^{2} \; \text{ for all 
   } \; x \in {\mathbb C}^{m} \; \text{ and }\; \lam \in i {\mathbb R}.
$$
In particular, if $Q_{2}(\lam) x = 0$ 
for some $\lam \in i{\mathbb R}$ and $x \in {\mathbb C}^{m}$, then 
also $Q_{1}(\lam) x = 0$ and hence 
$$
 \psi(\lam)^{-1} x = \begin{bmatrix} \Theta_{21}(\lam) & 
 \Theta_{22}(\lam) \end{bmatrix} \begin{bmatrix} Q_{1}(\lam) \\ 
 Q_{2}(\lam) \end{bmatrix} x = 0,
$$
which forces $x=0$ since $\psi$ is rational matrix inner.  We conclude:
\begin{itemize}
    \item {\em for $\lam \in i{\mathbb R}$, $Q_{2}(\lam)$ is 
    invertible and $G(\lam) =  Q_{1}(\lam) Q_{2}(\lam)^{-1}$
is a contraction.}
\end{itemize}

The next step is to apply a winding-number argument to get similar 
results for $\lam \in \Pi_+$.  From the  bottom component of \eqref{Sparam} we have, 
again for the moment with $\lam \in i {\mathbb R}$,
\begin{equation}   \label{bottom}
    \psi^{-1} = \Theta_{21} Q_{1} + \Theta_{22} Q_{2} 
= \Theta_{22} (\Theta_{22}^{-1} \Theta_{21}  G + I_{m}) Q_{2}.
\end{equation}
We conclude that, for the argument $\lam \in i {\mathbb R}$, 
\begin{equation}   \label{wno'}
 \operatorname{wno} \det( \psi^{-1}) = 
 \operatorname{wno}\det(\Theta_{22}) + \operatorname{wno} 
 \det(\Theta_{22}^{-1} \Theta_{21} G + I_{m})  + 
 \operatorname{wno} \det(Q_{2}) 
\end{equation}
where we use the notation $\operatorname{wno}f$ to indicate {\em 
winding number} or {\em change of argument} of the function $f$ as 
the variable runs along the imaginary line. Since both $\det 
\Theta_{22}^{-1}$ and $\det \psi$ are analytic on 
$\overline{\Pi}_{+}$, a consequence of the 
identity \eqref{Theta22-3}  is that
\begin{equation}   \label{wno''}
\operatorname{wno} \det( \psi^{-1}) = \operatorname{wno}\det(\Theta_{22}).
\end{equation}
Combining the two last equalities gives
\begin{equation}
 \operatorname{wno}
 \det(\Theta_{22}^{-1} \Theta_{21} G + I_{m})  + \operatorname{wno} \det(Q_{2})=0.
\label{wno}
\end{equation}

We have already observed that
$$
 \| \Theta_{22}(\lam)^{-1} \Theta_{21}(\lam)\| < 1\quad\mbox{and}\quad
 \| G(\lam) \| \le 1 \; \text{ for } \; \lam \in i {\mathbb R}.
$$
Hence, for $0 \le t \le 1$ we have $\| t\Theta_{22}(\lam)^{-1} 
\Theta_{21}(\lam)  G(\lam) \| < 1$ and hence
$t \Theta_{22}(\lam)^{-1} \Theta_{21}(\lam) G(\lam) + I$ is 
invertible for $\lam \in i {\mathbb R}$ for all $0 \le t \le 1$.
Hence 
$$ 
i(t): = \operatorname{wno} \det(t \Theta_{22}(\lam)^{-1} 
\Theta_{21}(\lam) G(\lam) + I)
$$
is well defined and independent of $t$ for $0 \le t \le 1$.  
As clearly $i(0) = 0$, it follows that 
$$
i(1) =  \operatorname{wno} \det( \Theta_{22}(\lam)^{-1} 
\Theta_{21}(\lam) G(\lam) + I) = 0
$$
which, on account of \eqref{wno}, implies 
$  \operatorname{wno} \det(Q_{2})  = 0$.  As $Q_{2}$ is analytic 
on $\Pi_{+}$, we conclude that $\det Q_{2}$ has no zeros in 
$\Pi_{+}$, i.e., $Q_{2}^{-1}$ is analytic on $\Pi_{+}$.  By the 
maximum modulus theorem it then follows that
$G(\lam) : = Q_{1}(\lam) Q_{2}(\lam)^{-1}$ is in the Schur class $\cS^{p 
\times m}(\Pi_{+})$.  Furthermore, from \eqref{Sparam} we  have
\begin{equation}
\begin{bmatrix} S \\ I_{m} \end{bmatrix} = \begin{bmatrix} 
    \Theta_{11} & \Theta_{12} \\ \Theta_{21} & \Theta_{22} 
\end{bmatrix} \begin{bmatrix} G \\ I \end{bmatrix} Q_{2} \psi.
\label{ad3}
\end{equation}
From the bottom component we read off that $Q_{2} \psi = (\Theta_{21} 
G + \Theta_{22})^{-1}$.  From the first component we then get
$$
 S = (\Theta_{11} G + \Theta_{12}) Q_{2} \psi =  (\Theta_{11} G + 
 \Theta_{12}) (\Theta_{21} G + \Theta_{22})^{-1}
$$
and the representation \eqref{LFTparam} follows.

\smallskip

Conversely, if $G \in \cS^{p \times m}(\Pi_{+})$, we can reverse  the 
above argument (with $Q_{1}(\lam) = G(\lam)$ and  $Q_{2}(\lam) = I_{m}$) to see that 
$S$ of the form \eqref{LFTparam} is a Schur-class solution of the interpolation 
conditions \eqref{BTOAint1}, \eqref{BTOAint2}, \eqref{BTOAint3}.
\end{proof}

\begin{remark} \label{R:Amaya}
    The theory from \cite{BGR} is worked out explicitly only with 
$H^{2}_{m}(\Pi_{+})$ replaced by its rational subspace $\operatorname{\rm 
Rat} H^{2}_{m}$ consisting of elements of $H^{2}_{m}$ with 
rational-function column entries, and similarly $H^{2}_{m}(\Pi_{-})$ and 
$L^{2}(i {\mathbb R})$ replaced by  their respective rational subspaces $\operatorname{Rat} 
H^{2}_{m}(\Pi_{-})$ and $\operatorname{Rat} L^{2}(i {\mathbb R})$.
Nevertheless the theory is easily adapted to the $L^{2}$-setting here. 
Subspaces $\cM$ of $L^{2}_{p+m}(i {\mathbb R})$ having a representation of 
the form \eqref{cMrep} (with $\sbm{U \\ V}, W, \begin{bmatrix} X & -Y 
\end{bmatrix}, Z, \Gamma$ all equal to finite matrices rather than infinite-dimensional 
operators) are characterized by the conditions:  (1) 
$\cM$ is forward-shift invariant, i.e., $\cM$ is invariant under multiplication 
by the function $\chi(\lambda) = \frac{ \lambda -1}{ \lambda +1}$, (2) the subspace
$(\cM + H^{2}_{p+m}(\Pi_{+}))/ H^{2}_{p+m}(\Pi_{+})$ has finite dimension, and (3) the 
quotient space $\cM/\left(\cM \cap H^{2}_{p+m}(\Pi_{+})\right)$ has finite dimension.
The representation \eqref{cM-rep} with $\psi^{-1}$ of the form 
\eqref{psi-inv} with finite matrices $U,W,P$ is roughly the special case 
of the statement above where $\cM = \cM_{{\mathfrak D},-} \supset
H^{2}_{m}(\Pi_{+})$.  The analogue of such representations 
\eqref{cMrep} and \eqref{cM-rep}--\eqref{psi-inv} for more general 
full-range pure forward shift-invariant subspaces of 
$L^{2}_{p+m}(\Pi_{+})$ (or dually of full-range pure backward 
shift-invariant subspaces of $L^{2}_{p+m}(\Pi_{+})$) involving infinite-dimensional 
(even unbounded) operators  $\sbm{ U \\ V}, W, Z, \sbm{X & -Y }, \Gamma$   
is worked out in the Virginia Tech dissertation of Austin Amaya \cite{Amaya}.
\end{remark}

\subsection{The Fundamental Matrix Inequality approach of Potapov} 
\label{S:FMI}
The linear fractional parametrization formula \eqref{LFTparam} can be 
alternatively established by the Potapov's method of the Fundamental
Matrix Inequalities. As we will see, this method bypasses the winding number argument.

\smallskip

Consider a $\Pi_{+}$-admissible BTOA interpolation data set ${\mathfrak D}$ as in 
\eqref{data} giving rise to the  collection  \eqref{BTOAint1}, 
\eqref{BTOAint2}, \eqref{BTOAint3} of BTOA interpolation conditions 
imposed on a Schur-class function $S^{p \times m}(\Pi_{+})$.
We assume that $\bGamma_{\mathfrak D}$ is positive definite.  We form 
the matrix $\Theta(\lambda)$ as in \eqref{oct1}--\eqref{oct2}  and 
assume all knowledge of all the properties of $\Theta$ falling out of 
the positive-definiteness of $\bGamma_{\mathfrak D}$, specifically
\eqref{bigLySyl}--\eqref{Theta22inv} above.

\smallskip

The main idea is to extend the interpolation data by one extra interpolation node 
$z\in\Pi_+$ with the corresponding full-range value $S(z)$, i.e., by 
the tautological full-range interpolation condition
\begin{equation}  \label{generic-int}
 S(z) = S(z)
\end{equation}
where $z$ is a generic point in the right half plane.  
To set up this augmented problem as a BTOA problem, we have a choice 
as to how we incorporate the global generic interpolation condition 
\eqref{generic-int} into the BTOA formalism:  (a) 
as a LTOA interpolation condition:
\begin{equation}  \label{genericLTOA}
   (X_{z} S)^{\wedge L}(Z_{z}) = Y_{z} \;  \text{ where } \; X_{z} = I_{p}, 
   \, Y_{z} = S(z), \, Z_{z} = z I_{p},
\end{equation}
or  as a RTOA interpolation condition:
\begin{equation}   \label{genericRTOA}
    (S U_{z})^{\wedge R}(W_{z}) = V_{z} \; \text{ where } \;
    U_{z} = I_{m}, \, V_{z} = S(z), \, W_{z} = z I_{m}.
\end{equation}
We  choose here to work with the left versions \eqref{genericLTOA} exclusively; working 
with the right version \eqref{genericRTOA} will give seemingly different but in the end 
equivalent parallel results.  

\smallskip

As a first step, we wish to combine \eqref{BTOAint1} and 
\eqref{genericLTOA} into a single LTOA interpolation condition.
This is achieved by augmenting the matrices $(Z, X, Y)$ to the 
augmented triple $(Z_{\rm aug}, X_{\rm aug}, Y_{\rm aug})$ given by
$$
Z_{\rm aug} = \begin{bmatrix} Z & 0 \\ 0 & z I_{p} \end{bmatrix}, 
\quad
X_{\rm aug} = \begin{bmatrix} X \\ I_{p} \end{bmatrix}, \quad
Y_{\rm aug} = \begin{bmatrix} Y \\ S(z) \end{bmatrix}.
$$
Here all matrices indexed by aug depend on the parameter $z$, but for 
the moment we suppress this dependence from the notation.
As the RTOA-interpolation conditions for the augmented problem remain 
the same as in the original problem (namely, \eqref{BTOAint2}), we set
$$
U_{\rm aug} = U, \quad V_{\rm aug} = V, \quad W_{\rm aug} = W.
$$
We therefore take the \textbf{augmented data set} ${\mathfrak D}_{\rm 
aug}$ to have the form
\begin{equation} \label{Daug}
    {\mathfrak D}_{\rm aug} = (X_{\rm aug}, Y_{\rm aug}, Z_{\rm aug}; 
    U_{\rm aug}, V_{\rm aug}, W_{\rm aug}; \Gamma_{\rm aug})
\end{equation}
where the coupling matrix $\Gamma_{\rm aug}$ is still to be 
determined.

\smallskip

We know that $\Gamma_{\rm aug}$ must solve the Sylvester equation
\eqref{Sylvester'} associated with the data set ${\mathfrak D}_{\rm 
aug}$, i.e., $\Gamma_{\rm aug}$ must have the form $\Gamma_{\rm aug} 
= \sbm{ \Gamma_{{\rm aug},1} \\ \Gamma_{{\rm aug},2} }$  with
$$    \begin{bmatrix} \Gamma_{{\rm aug},1} \\ \Gamma_{{\rm aug},2} 
    \end{bmatrix} W - \begin{bmatrix} Z & 0 \\ 0 & z I_{p} 
\end{bmatrix} \begin{bmatrix} \Gamma_{{\rm aug},1} \\ \Gamma_{{\rm aug},2} 
    \end{bmatrix} = 
   \begin{bmatrix} X \\ I_{p} \end{bmatrix} V - \begin{bmatrix} Y \\ 
       S(z) \end{bmatrix} U.
$$
Equivalently, $\Gamma_{\rm aug} = \sbm{ \Gamma_{{\rm aug},1} \\ 
\Gamma_{{\rm aug},2}}$ is determined by the decoupled system of equations
\begin{align}  
   &  \Gamma_{{\rm aug}, 1} W - Z \Gamma_{{\rm aug}, 1} = X V - Y U, 
   \notag \\
& \Gamma_{{\rm aug},2} W - (z I_{p}) \Gamma_{{\rm aug}, 2} = V - S(z) U.
\label{aug-Sylvester''}
\end{align}
In addition, the third augmented interpolation condition takes the 
form
$$
 \left( \begin{bmatrix} X \\ I_{p} \end{bmatrix} S \, U 
 \right)^{\wedge L,R} \left( \begin{bmatrix} Z & 0 \\ 0 & zI_{p} 
\end{bmatrix}, W \right) = \begin{bmatrix} \Gamma_{{\rm aug}, 1} \\
\Gamma_{{\rm aug}, 2} \end{bmatrix}
$$
which can be decoupled into two independent bitangential 
interpolation conditions
\begin{equation}   \label{coupledBTint}
    ( X S U)^{\wedge L,R}(Z,W) = \Gamma_{{\rm aug},1}, \quad
    (I_{p} S U)^{\wedge L,R}(z I_{p}, W ) = \Gamma_{{\rm aug},2}.
\end{equation}
From the first of the conditions \eqref{coupledBTint} coupled with 
the interpolation condition \eqref{BTOAint3}, we are forced to take
$  \Gamma_{{\rm aug}, 1} = \Gamma$.

\smallskip

Since the point $z \in \Pi_{+}$ is generic, we may assume as a first 
case that $z$ is disjoint from the spectrum $\sigma(W)$ of $W$.  Then 
we can solve the second of the equations \eqref{aug-Sylvester''} 
uniquely for $\Gamma_{{\rm aug}, 2}$:
\begin{equation} \label{Gammaaug2}
\Gamma_{{\rm aug},2} = (S(z) U - V) (zI_{n_{W}} - W)^{-1}.
\end{equation}
A consequence of the RTOA interpolation condition \eqref{BTOAint2} is 
that the right-hand side of \eqref{Gammaaug2} has analytic 
continuation to all of $\Pi_{+}$.  It is not difficult to see that 
$(I_{p} S U)^{L,R}(z I_{p}, W)$ in general is just the value of this 
analytic continuation at the point $z$; we conclude that the formula 
\eqref{Gammaaug2} holds also at points $z$ in $\sigma(W)$ with proper 
interpretation.  In this way we have completed the computation of the 
augmented data set  \eqref{Daug}:
\begin{equation}   \label{Daug'}
    {\mathfrak D}(z):= {\mathfrak D}_{\rm aug} = 
 \left(   \begin{bmatrix} Z & 0 \\ 0 & zI_{p} \end{bmatrix}, \,
\begin{bmatrix} X \\ I_{p} \end{bmatrix}, \, 
\begin{bmatrix} Y \\ S(z) \end{bmatrix}; \, U,V, W; \begin{bmatrix} 
    \Gamma \\ T_{S,1}(z) \end{bmatrix} \right).
\end{equation}
where we set
\begin{equation}   \label{TS1}
  T_{S,1}(z) =  (S(z) U - V) (zI_{n_{W}} - W)^{-1}.
\end{equation}

We next compute the Pick matrix $\bGamma_{{\mathfrak D}_{\rm 
aug}(z)}$ for the augmented data set ${\mathfrak D}_{\rm aug}$ 
\eqref{Daug'} according to the recipe \eqref{GammaL}--\eqref{GammaData}.  
Thus 
$$
\bGamma_{{\mathfrak D}_{\rm aug}(z)} = 
\begin{bmatrix}\Gamma_{{\rm aug}, L} & \Gamma_{\rm aug} \\ (\Gamma_{\rm 
aug})^{*} & \Gamma_{{\rm aug}, R}\end{bmatrix},\quad\mbox{where}\quad
\Gamma_{\rm aug} = \begin{bmatrix} \Gamma \\ T_{S,1}(z) \end{bmatrix}, 
\; \Gamma_{{\rm aug},R} = \Gamma_{R},
$$
and where $\Gamma_{{\rm aug}, L} = \sbm{\Gamma_{{\rm aug}, L11} & 
\Gamma_{{\rm aug}, L12} \\ \Gamma_{{\rm aug}, L21} & \Gamma_{{\rm 
aug}, L22 }}$ is determined by the Lyapunov equation \eqref{GammaL} 
adapted to the interpolation data set ${\mathfrak D}(z)$:
\begin{align*}
&\begin{bmatrix} \Gamma_{{\rm aug}, L11} & \Gamma_{{\rm aug}, L12} \\ 
    \Gamma_{{\rm aug}, L21} & \Gamma_{{\rm aug}, L22} \end{bmatrix}
\begin{bmatrix} Z^{*} & 0 \\ 0 & \overline{z} I_{p} \end{bmatrix} +
\begin{bmatrix} Z & 0 \\ 0 & z I_{p} \end{bmatrix}
\begin{bmatrix} \Gamma_{{\rm aug}, L11} & \Gamma_{{\rm aug}, L12} \\ 
    \Gamma_{{\rm aug}, L21} & \Gamma_{{\rm aug}, L22} \end{bmatrix} \\
& \quad = \begin{bmatrix} X \\ I_{p} \end{bmatrix} \begin{bmatrix} X^{*} & 
 I_{p} \end{bmatrix} - \begin{bmatrix} Y \\ S(z) \end{bmatrix}
\begin{bmatrix} Y^{*} & S(z)^{*} \end{bmatrix}.
\end{align*}
One can solve this equation uniquely for $\Gamma_{{\rm aug}, Lij}$ ($i,j=1,2$) 
with the result
\begin{align*}
 & \Gamma_{{\rm aug}, L11} = \Gamma_{L}, \quad \Gamma_{{\rm aug}, L21} 
  = (\Gamma_{{\rm aug}, L12})^{*}  = T_{S,2}(z), \notag \\
&  \Gamma_{{\rm aug},L22} = \frac{I - S(z) S(z)^{*}}{z + \overline{z}}
\end{align*}
where we set
\begin{equation}   \label{TS2}
   T_{S,2}(z): = (X^{*} - S(z) Y^{*})(zI_{n_{Z}} + Z^{*})^{-1}.
\end{equation}
In this way we arrive at the  Pick matrix for data set ${\mathfrak 
D}(z)$, denoted for convenience as $\bGamma_{\mathfrak D}(z)$ 
rather than as $\bGamma_{{\mathfrak D}(z)}$:
$$
\bGamma_{\mathfrak D}(z) = \begin{bmatrix} \Gamma_{L} & 
T_{S,2}(z)^{*} & \Gamma \\ T_{S,2}(z) & \frac{ I - S(z) S(z)^{*}}{ z + 
\overline{z}} & T_{S,1}(z) \\ \Gamma^{*} & T_{S,1}(z)^{*} & 
\Gamma_{R} \end{bmatrix}.
$$
If we interchange the second and third rows and then also the second 
and third columns (i.e., conjugate by a permutation matrix), we get a 
new matrix having the same inertia; for simplicity from now on we use 
the same notation $\bGamma_{\mathfrak D}(z)$ for this transformed 
matrix:
$$
  \bGamma_{{\mathfrak D}}(z) = \begin{bmatrix} \Gamma_{L} & \Gamma & 
  T_{S,2}(z)^{*} \\ \Gamma^{*} & \Gamma_{R} & T_{S,1}(z)^{*} \\
  T_{S,2}(z) & T_{S,1}(z) & \frac{ I - S(z) S(z)^{*}}{z + 
  \overline{z}} \end{bmatrix}.
$$

Had we started with a finite number $\bz = \{z_{1}, \dots, z_{N}\}$ 
of generic interpolation nodes in $\Pi_{+}$ rather than a single 
generic point $z$ and augmented the interpolation conditions 
\eqref{BTOAint1}, \eqref{BTOAint2}, \eqref{BTOAint3} with the 
collection of tautological interpolation conditions
$$
  S(z_{i}) = S(z_{i}) \; \text{ for } \; i = 1, \dots, N
$$
modeled as the additional LTOA interpolation condition
$$
  (X_{\bz} S)^{\wedge L}(Z_{\bz}) = Y_{\bz}
$$
where
$$
Z_{\bz} = \sbm{ z_{1}I_{p} & & 0 \\ & \ddots & \\ 0 & & z_{N} I_{p} }, \quad
X_{\bz} = \sbm{I_{p} \\ \vdots \\ I_{p} }, 
\quad Y_{\bz} = \sbm{ S(z_{1}) \\ \vdots \\ S(z_{N})},
$$
the same analysis as above would lead us to the following conclusion: 
{\em there is a matrix function $S$ in the Schur class $\cS^{p \times 
m}(\Pi_{+})$ satisfying the interpolation conditions 
\eqref{BTOAint1}, \eqref{BTOAint2}, \eqref{BTOAint3} if and only if, 
for $\bz = \{z_{1}, \dots, z_{N}\}$ any collection of $N$ distinct 
points in $\Pi_{+}$, the associated augmented Pick matrix 
$\bGamma_{\mathfrak D}(\bz)$ is positive-semidefinite, where}
$$
\bGamma_{\mathfrak D}(\bz) =
\begin{bmatrix} \Gamma_{L} & \Gamma & T_{S,2}(z_{1})^{*} & \cdots & 
    T_{S,2}(z_{N})^{*} \\
\Gamma^{*} & \Gamma_{R} & T_{S,1}(z_{1})^{*} & \cdots & 
T_{S,1}(z_{N})^{*} \\
T_{S,2}(z_{1}) & T_{S,1}(z_{1}) & 
\frac{I - S(z_{1})S(z_{1})^{*}}{z_{1} + \overline{z}_{1}} & \cdots  &
\frac{I - S(z_{1}) S(z_{N})^{*}}{z_{1} + \overline{z}_{N}}  \\
\vdots & \vdots & \vdots & \ddots & \vdots \\
T_{S,2}(z_{N}) & T_{S,1}(z_{N}) & 
\frac{ I - S(z_{N}) S(z_{1})^{*}}{z_{1} + \overline{z}_{N}} & \cdots & 
\frac{ I - S(z_{N}) S(z_{N})^{*}}{z_{N} + \overline{z}_{N}} \end{bmatrix} \succeq 0.
$$
As the finite set of points $\bz = \{z_{1}, \dots, z_{N}\}$ ($N=1,2, 
\dots$) is an arbitrary finite subset of $\Pi_{+}$, this condition in 
turn amounts to the assertion that the kernel 
$\bGamma_{\mathfrak D}(z, \zeta)$ defined by
\begin{equation}  \label{sep8a}
  \bGamma_{\mathfrak D}(z, \zeta) = \begin{bmatrix}  \Gamma_{L} & \Gamma & 
  T_{S,2}(\zeta)^{*} \\
 \Gamma^{*} & \Gamma_{R} & T_{S,1}(\zeta)^{*} \\
 T_{S,2}(z) & T_{S,1}(z) & \frac{ I - S(z) S(\zeta)^{*}}{z + 
 \overline{\zeta}}  \end{bmatrix}
\end{equation}
is a positive kernel on $\Pi_{+}$ (see \eqref{posker}).
Observe from \eqref{TS2}, \eqref{TS1} that
\begin{align*}
\begin{bmatrix}T_{S,2}(z) & T_{S,1}(z)\end{bmatrix}&=
- \begin{bmatrix}I_p & -S(z)\end{bmatrix}\begin{bmatrix}
 -X^{*} & V \\ -Y^{*} & U \end{bmatrix}
 \begin{bmatrix}(zI+Z^*)^{-1} &0 \\ 0 & (zI-W)^{-1}\end{bmatrix}\notag\\
&= - \begin{bmatrix}I_p & -S(z)\end{bmatrix}{\bf C}(z I-{\bf A})^{-1},
\end{align*}
where ${\bf C}$ and ${\bf A}$ are defined as in \eqref{oct2}. Taking the latter 
formula into account, we way write
\eqref{sep8a} in a more structured form as
\begin{equation} \label{sep8c}
\bGamma_{\mathfrak D}(z,\zeta)=\begin{bmatrix} \Gamma_{\mathfrak D} &
-(\overline\zeta I-{\bf A}^*)^{-1}
{\bf C}^*\begin{bmatrix}I_p \\ -S(\zeta)^*\end{bmatrix}\\
-\begin{bmatrix}I_p & -S(z)\end{bmatrix}{\bf C}(z I-{\bf A})^{-1}&
{\displaystyle\frac{I - S(z)S(\zeta)^*}{z+\overline{\zeta}}}\end{bmatrix}.
\end{equation}
Since the matrix $\bGamma_{\mathfrak D}$ is positive definite, the kernel \eqref{sep8c} 
is positive if and only if the 
Schur complement of $\Gamma_{\mathfrak D}$ is a positive kernel on  
$\Pi_+\backslash  \sigma(W)$ 
and therefore, admits a unique positive extension to the whole $\Pi_+$:
$$
\frac{I - S(z)S(\zeta)^*}{z+\overline{\zeta}}-
\begin{bmatrix}I_p & -S(z)\end{bmatrix}{\bf C}(z I-{\bf A})^{-1}\Gamma_{\mathfrak D}^{-1}
(\overline\zeta I-{\bf A}^*)^{-1}{\bf C}^*\begin{bmatrix}I_p \\ 
-S(\zeta)^*\end{bmatrix}\succeq 0.
$$
The latter can be written as
$$
\begin{bmatrix}I_p & -S(z)\end{bmatrix}\left\{\frac{J}{z+\overline{\zeta}}-
{\bf C}(z I-{\bf A})^{-1}\Gamma_{\mathfrak D}^{-1}
(\overline\zeta I-{\bf A}^*)^{-1}{\bf C}^*\right\}
\begin{bmatrix}I_p \\ -S(\zeta)^*\end{bmatrix}\succeq 0,
$$
and finally, upon making use of \eqref{ad1}, as
\begin{equation}
\begin{bmatrix}I_p & -S(z)\end{bmatrix}\frac{\Theta(z)J\Theta(\zeta)^*}
    {z+\overline{\zeta}}
\begin{bmatrix}I_p \\ -S(\zeta)^*\end{bmatrix}\succeq 0.
\label{sep8d}
\end{equation}
We next define two functions $Q_1$ and $Q_2$ by the formula
\begin{equation}
\begin{bmatrix}Q_2(z) & -Q_1(z)\end{bmatrix}=\begin{bmatrix}I_p & -S(z)\end{bmatrix}
\begin{bmatrix}\Theta_{11}(z) & \Theta_{12}(z)\\ \Theta_{21}(z) & \Theta_{22}(z)
\end{bmatrix},   
\label{sep8f}
\end{equation}
and write \eqref{sep8d} in terms of these functions as
$$
\begin{bmatrix}Q_2(z) & -Q_1(z)\end{bmatrix}\frac{J}{z+\overline{\zeta}}
\begin{bmatrix}Q_2(\zeta)^* \\ -Q_1(\zeta)^*\end{bmatrix}=
\frac{Q_2(z)Q_2(\zeta)^*-Q_1(z)Q_1(\zeta)^*}{z+\overline{\zeta}}
\succeq 0.
$$
By Leech's theorem \cite{leech}, there exists a Schur-class function 
$G\in\mathcal S^{p\times m}(\Pi_+)$ such that 
$$
Q_2 G=Q_1,
$$
which, in view of \eqref{sep8f} can be written as 
$$
(\Theta_{11}-S\Theta_{21})G=S\Theta_{22}-\Theta_{12},
$$
or equivalently, as
\begin{equation}   \label{sep8g}
    S (\Theta_{21} G + \Theta_{22}) = \Theta_{11} G + \Theta_{12}.
\end{equation}
Note that $\Theta_{22}(z)$ is invertible and that 
$\Theta_{22}(z)^{-1} \Theta_{21}(z)$ is strictly contractive on all of 
$\Pi_{+} \setminus \sigma(W)$ (and then on all of $\Pi_{+}$ by 
analytic continuation) as a consequence of the bullet immediately 
after \eqref{Theta22invTheta21} above.  As $G$ is in the Schur class 
and hence is contractive on all of $\Pi_{+}$, 
it follows that $\Theta_{22}(z)^{-1}\Theta_{21}(z)G(z)+I_m$ is 
invertible on all of $\Pi_{+}$.
Hence 
$$\Theta_{21}(z)G(z)+\Theta_{22}(z)=\Theta_{22}(z)(\Theta_{22}(z)^{-1}
\Theta_{21}(z)G(z)+I_m)
$$ 
is invertible for all $z\in\Pi_+$ and we can solve \eqref{sep8g} for $S$ 
arriving at at the formula \eqref{LFTparam}. 

\begin{remark} \label{R:FMI}  Note that in this Potapov approach to 
    the derivation of the linear-fractional parametrization via the 
    Fundamental Matrix Inequality, the winding-number argument 
    appearing in the state-space approach never appears.  What 
    apparently replaces it, once everything is properly organized, is 
    the theorem of Leech.
\end{remark}

\subsection{Positive kernels and reproducing kernel Hilbert spaces}  \label{S:RKHS}
Assume now that we are given a $\Pi_{+}$-admissible interpolation data 
set and that the PIck matrix $\bGamma_{\mathfrak D}$ is invertible. 
Then one can define the matrix function $\Theta(z)$ as in 
\eqref{oct1} and then $K_{\Theta,J}(z, \zeta) = 
\frac{J - \Theta(z) J \Theta(\zeta)^{*}}{z + \overline{\zeta}}$ is 
given by \eqref{ad1}. A straightforward computation then shows that, 
for any $N=1,2,\dots$ with points $z_{1}, \dots, z_{N}$ in 
$\Pi_{+}\setminus \sigma(W)$ and vectors $\by_{1}, \dots, \by_{N}$ in 
${\mathbb C}^{p+m}$, we have
\begin{align*}
&\sum_{i,j=1}^{N} \langle K_{\Theta,J}(z_{i}, z_{j}) \by_{j}, \by_{i} 
\rangle_{{\mathbb C}^{p+m}}  = \\
& \quad = \sum_{i,j=1}^{N} \bigg\langle \bGamma_{\mathfrak D}^{-1}
\bigg( \sum_{i=1}^{N} (\overline{z}_{i} I - \bA^*)^{-1}\bC^{*} _{i} 
\bigg), \; \sum_{j=1}^{N} (\overline{z}_{j} I - \bA^*)^{-1}\bC^{*} _{j} 
 \bigg\rangle, 
\end{align*}
and hence
 $K_{\Theta,J}$ is a positive kernel on $\Pi_{+} \setminus \sigma(W)$ 
 if $\bGamma_{\mathfrak D} \succ 0$.  More generally, if 
 $\bGamma_{\mathfrak D}$ has some number $\kappa$ of negative 
 eigenvalues, then for any choice of points $z_{1}, \dots, z_{N} \in 
 \Pi_{+} \setminus \sigma(W)$ the block Hermitian matrix
\begin{equation} \label{KThetaJblock}
 \left[ K_{\Theta,J}(z_{i}, z_{j}) \right]_{i,j=1, \dots, N}
 \end{equation}
has at most $\kappa$ negative eigenvalues.  If we impose the controllability and observability 
assumptions on the matrix pairs $(U,W)$ and $(Z,X)$,  then there 
exist a choice of $z_{1}, \dots, z_{N} \in \Pi_{+} \setminus 
\sigma(W)$ so that the matrix \eqref{KThetaJblock} has exactly 
$\kappa$ negative eigenvalues, in which case we say that $\Theta$ is 
in the generalized $J$-Schur class $\cS_{J, \kappa}(\Pi_{+})$ 
(compare with the Kre\u{\i}n-Langer generalized Schur class discussed 
at the beginning of Section \ref{S:negsquares} below).  In the case 
where $\Theta \in \cS_{J, \kappa}(\Pi_{+})$ with $\kappa > 0$, there 
is still associated a space of functions $\cH(K_{\Theta,J})$ as in 
\eqref{RKHSa}--\eqref{RKHSb}; the space $\cH(K_{\Theta,J})$ is 
now a {\em Pontryagin space} with negative index equal to $\kappa$
(see Section \ref{S:Kreinprelim} for background on Pontryagin and 
Kre\u{\i}n spaces).  In any case, in this way we arrive at yet another 
interpretation of the condition that $\bGamma_{\mathfrak D}$ be 
positive definite.

\begin{theorem} \label{T:RKHS}
    Assume that we are given a $\Pi_{+}$-admissible interpolation 
    data set with $\Gamma_{\mathfrak D}$ is invertible (so  
    $\Theta$ and $K_{\Theta,J}$ are defined).  Then 
    $\cH(K_{\Theta,J})$ is a Hilbert space if and only if 
    $\bGamma_{\mathfrak D} \succ 0$.
\end{theorem}

In Section \ref{S:synthesis} below (see display \eqref{HKThetaJ})
we shall spell this criterion out in more detail and arrive at another 
condition equivalent to positive-definiteness 
of the Pick matrix $\bGamma_{\mathfrak D}$.

\section{The Grassmannian/Kre\u{\i}n-space-geometry approach to the 
BTOA-NP interpolation problem}  
\label{S:proof2}
In this section we sketch the Grassmannian/Kre\u{\i}n-space geometry 
proof of Theorem \ref{T:BTOA-NP} based on the work in \cite{BH}---see 
also \cite{Ball-Indiana} for a more expository account and \cite{BF} 
for a more recent overview which also highlights the method in 
various multivariable settings.  These treatments work with the 
Sarason \cite{Sarason67} or Model-Matching \cite{Francis} formulation 
of the Nevanlinna-Pick interpolation problem, while we work with the 
LTOA-interpolation formulation.
The translation between the two is given in \cite[Chapter 16]{BGR}
(where the Sarason/Model Matching formulation is called 
{\em divisor-remainder form}.

\subsection{Kre\u{\i}n-space preliminaries}  \label{S:Kreinprelim}
Let us first review a few preliminaries concerning Kre\u{\i}n spaces.
A Kre\u{\i}n space by definition is a linear space $\cK$ endowed
with an indefinite inner product $[ \cdot, \cdot]$ which is
{\em complete} in the following sense:  there are two subspaces
$\cK_{+}$ and $\cK_{-}$ of $\cK$ such that the restriction of $[
\cdot, \cdot ]$ to $\cK_{+} \times \cK_{+}$ makes $\cK_{+}$ a
Hilbert space while the restriction of $-[ \cdot, \cdot]$ to
$\cK_{-} \times \cK_{-}$ makes $\cK_{-}$ a Hilbert space, and 
\begin{equation}
\cK = \cK_{+} [\dot +] \cK_{-}
\label{dec}
\end{equation}
is a $[ \cdot, \cdot]$-orthogonal direct sum decomposition
of $\cK$.  In this case the decomposition \eqref{dec}
is said to form a {\em fundamental decomposition} for $\cK$.
Fundamental decompositions are never unique except in the trivial
case where one of $\cK_{+}$ or $\cK_{-}$ is equal to the zero space.
If $\min(\dim \cK_{+}, \, \dim \cK_{-})=\kappa<\infty$, then 
$\cK$ is called a Pontryagin space of index $\kappa$.

\smallskip

Unlike the case of Hilbert spaces where closed subspaces all look
the same, there is a rich geometry for subspaces of a Kre\u{\i}n
space. A subspace $\cM$ of a Kre\u{\i}n space $\cK$ is said to be
{\em positive}, {\em isotropic}, or {\em negative} depending on
whether $[u, u] \ge 0$ for all $u \in \cM$, $[ u, u ] =0$ for all
$u \in \cM$ (in which case it follows that $[u,v] = 0$ for all
$u,v \in \cM$ as a consequence of the Cauchy-Schwarz inequality),
or $[u,u] \le 0$ for all $u \in \cM$.  Given any subspace $\cM$,
we define the Kre\u{\i}n-space orthogonal complement
$\cM^{[\perp]}$ to consist of all $v \in \cK$ such that $[u,v] =
0$ for all $u \in \cK$.  Note that the statement that $\cM$ is
isotropic is just the statement that $\cM \subset \cM^{[ \perp
]}$.  If it happens that $\cM = \cM^{[ \perp]}$, we say that $\cM$
is a {\em Lagrangian} subspace of $\cK$.  Simple examples show that 
in general, unlike the Hilbert space case, it can happen that $\cM$ 
is a closed subspace of the Kre\u{\i}n space $\cK$ yet the space 
$\cK$ cannot be split at the $\cK$-orthogonal direct sum of $\cM$ and 
$\cM^{[\perp]}$ (e.g., this happens dramatically if $\cM$ is an 
isotropic subspace of $\cK$). If $\cM$ is a subspace of $\cK$ for which this does 
happen, i.e., such that $\cK = \cM [+] \cM^{[\perp]}$, we say 
that $\cM$ is a {\em regular subspace} of $\cK$.

\smallskip

Examples of such subspaces arise from placing appropriate
Kre\u{\i}n-space inner products on the direct sum $\cH^\prime \oplus
\cH$ of two Hilbert spaces and looking at graphs of operators of
an appropriate class.

\begin{example}  \label{E:unitary}   Suppose that $\cH'$ and
    $\cH$ are two Hilbert spaces and we take $\cK$ to be the
    external direct sum $\cH' \oplus \cH$ with inner product
$$
    \left[ \begin{bmatrix} x \\ y \end{bmatrix}, \,
    \begin{bmatrix} x' \\ y' \end{bmatrix} \right] =
    \left\langle \begin{bmatrix} I_{\cH'} & 0 \\ 0 & -I_{\cH}
    \end{bmatrix} \begin{bmatrix} x \\ y \end{bmatrix}, \,
    \begin{bmatrix} x' \\ y' \end{bmatrix} \right\rangle_{\cH' \oplus \cH}
$$
where $\langle \cdot, \cdot \rangle_{\cH' \oplus \cH}$ is the
standard Hilbert-space inner product on the direct-sum Hilbert
space $\cH' \oplus \cH$.  In this case it is easy to find a
fundamental decomposition: take $\cK_{+} = \sbm{ \cH \\ \{0\}}$
and $\cK_{-} = \sbm{ \{0\} \\ \cH' }$. Now let $T$ be a bounded
linear operator from $\cH$ to $\cH'$ and let $\cM$ be the graph of
$T$:
$$
  \cM = \cG_{T} = \left\{ \begin{bmatrix} T x \\ x \end{bmatrix}
  \colon x \in \cH \right\} \subset \cK.
  $$
  Then a nice exercise is to work out the following facts:
  \begin{itemize}
      \item $\cG_{T}$ is negative if and only if $\|T\| \le 1$, in 
      which case $\cG_{T}$ is {\em maximal negative}, i.e.,  the subspace
      $\cG_{T}$ is not contained in any strictly larger negative 
      subspace.

      \item $\cG_{T}$ is isotropic if and only if $T$ is isometric
      ($T^{*} T = I_{\cH}$).

      \item $\cG_{T}$ is Lagrangian if and only if $T$ is unitary:
      $T^{*} T = I_{\cH}$ and $T T^{*} = I_{\cH'}$.
   \end{itemize}
\end{example}

Let $\cM$ be a fixed subspace of a Kre\u{\i}n space $\cK$ and $\cG$ a 
closed subspace of $\cM$.  In order that $\cG$ be maximal negative as 
a subspace of $\cK$, it is clearly necessary that $\cG$ be maximal 
negative as a subspace of $\cM$.  The following lemma (see \cite{BH} or 
\cite{Ball-Indiana}
for the proof) identifies when the converse holds.

\begin{lemma}  \label{L:maxneg}  Suppose that $\cM$ is a closed 
    subspace of a Kre\u{\i}n-space $\cK$ and $\cG$ is a negative 
    subspace of $\cM$.  Then a subspace $\cG \subset \cM$ which is maximal-negative 
    as a subspace of $\cM$ is automatically also maximal negative as 
    a subspace of $\cK$  if and only if the Kre\u{\i}n-space 
    orthogonal complement
    $$
    \cK [-] \cM = \{ k \in \cK \colon [k, m ]_{\cK} = 0 \text{ for 
    all } m \in \cM\}
    $$
    is a positive subspace of $\cK$.
\end{lemma}

\subsection{The Grassmannian/Kre\u{\i}n-space approach to 
interpolation}
Suppose now that we are given a $\Pi_{+}$-admissible 
BTOA-interpolation data set as in \eqref{data}.
Let $\cM_{\mathfrak D} \subset L^{2}_{p+m}(i {\mathbb R})$ be as in 
\eqref{cMrep}. We view $\cM_{\mathfrak D}$ as a subspace of the 
Kre\u{\i}n space 
\begin{equation}    \label{defcK}  
\cK  =\begin{bmatrix} L^{2}_{p}(i {\mathbb R}) \\ \mathcal M_{\mathfrak D,-} 
\end{bmatrix}=
\begin{bmatrix} L^{2}_{p}(i {\mathbb R}) \\ \psi^{-1} 
    H^{2}_{m}(\Pi_{+}) \end{bmatrix}
\end{equation}
(where we use the notation in \eqref{cM-rep})
with Kre\u{\i}n-space inner product $[ \cdot, \cdot ]_{J}$ induced by the matrix $J = 
\sbm{ 
I_{p} & 0 \\ 0 & -I_{m} }$:
$$
\left[ \begin{bmatrix} f_{1} \\ f_{2}  \end{bmatrix}, \, 
\begin{bmatrix} g_{1} \\ g_{2}  \end{bmatrix} 
\right]_{J}: =
\langle f_{1}, \, g_{1} \rangle_{L^{2}_{p}(i {\mathbb R})} -
\langle f_{2}, \, g_{2} \rangle_{ \psi^{-1} H^{2}_{m}(\Pi_{+})}.
$$
A key subspace in the Kre\u{\i}n-space geometry approach to the BTOA-NP problem is 
the $J$-orthogonal complement of $\cM_{\mathfrak D}$ inside $\cK$:
\begin{equation}  \label{cPdata}
 \cM_{\mathfrak D}^{[ \perp \cK]} : = \cK [-]_{J} \cM_{\mathfrak D} = 
 \{ f \in \cK \colon [ f, g ]_{J} = 0 \text{ for all } g \in 
 \cM_{\mathfrak D} \}.
\end{equation}
We then have the following result.

\begin{theorem}  \label{T:BTOA-NP'} 
The BTOA-NP has a solution $S \in \cS^{p \times m}(\Pi_{+})$ if and 
only if the subspace $\cM_{\mathfrak D}^{[\perp \cK]}$ \eqref{cPdata} 
is a positive subspace of $\cK$ \eqref{defcK}, i.e.,
 $$
 [ f, f ]_{_J} \ge 0\; \text{ for all } \; f \in \cM_{\mathfrak 
 D}^{[\perp \cK]}.
 $$
 If it is the case that $\cM_{\mathfrak D}^{[\perp \cK]}$ is a Hilbert 
 space in the $\cK$-inner product, then there is rational  
 $J$-inner function $\Theta$ so that 
 \begin{enumerate}
     \item $\Theta$ provides a 
 Beurling-Lax representation for $\cM_{\mathfrak D}$ \eqref{cMBLrep}, 
 and
 \item the set of all Schur-class solutions $S \in \cS^{p \times 
 m}(\Pi_{+})$ of the interpolation conditions \eqref{BTOAint1}, 
 \eqref{BTOAint2}, \eqref{BTOAint3} is given by the linear-fractional 
 parametrization formula \eqref{LFTparam} with $G \in \cS^{p \times 
 m}(\Pi_{+})$.
 \end{enumerate}
 \end{theorem}
 
 \begin{proof}[Sketch of the proof of Theorem \ref{T:BTOA-NP'}]
We first argue the $\cM_{\mathfrak D}^{[\perp \cK]}$ being a positive 
subspace of $\cK$ is necessary for the BTOA-NP to have a solution.
Let $S\in \cS^{p \times m}(\Pi_{+})$ be such a solution and let 
$M_{S} \colon \psi^{-1} H^{2}_{m}(\Pi_{+}) \to L^{2}_{p}(i {\mathbb R})$ be the 
operator of multiplication by $S$:
$$
  M_{S} \colon \psi^{-1} h \mapsto  S \cdot \psi^{-1} h.
$$
The operator norm of $M_{S}$ is the same as the supremum norm of $S$
over $i {\mathbb R}$:
$$
 \| M_{S}\|_{\rm op} = \|S\|_\infty:=\sup_{\lam \in i {\mathbb R}} \| S(\lam) \|.
$$
Let us consider the graph space of $M_{S}$, namely
\begin{equation}
  \cG_{S} = \begin{bmatrix} M_{S} \\ I_m \end{bmatrix} \psi^{-1}  H^{2}_{m}(\Pi_{+}) 
   = \begin{bmatrix} S \\ I_{m} \end{bmatrix} \cdot \psi^{-1} 
   H^{2}_{m}(\Pi_{+}).
\label{ad2}
\end{equation}
By the first bullet in Example \ref{E:unitary}, it follows that 
\begin{itemize} 
    \item {\em $\| S\|_{\infty} \le 1$ if and only if $\cG_{S}$ 
    is a maximal negative subspace of $\cK$.}
\end{itemize}
Moreover, as a consequence of the criterion \eqref{sol-rep} for $S$ 
to satisfy the interpolation conditions, we have
\begin{itemize}
    \item {\em $S$ satisfies the interpolation conditions if and only 
    if $\cG_{S} \subset \cM_{\mathfrak D}$.}
    \end{itemize}
By combining these two observations, we see that if $S$ is a 
solution to the BTOA-NP, then the 
subspace $\cG_{S}$ is contained in $\cM_{\mathfrak D}$ and is maximal 
negative in $\cK$.  It follows that $\cM_{\mathfrak D}^{[\perp \cK]}$ is 
a positive subspace in $\cK$  as a consequence of Lemma \ref{L:maxneg}.
This verifies the necessity part in Theorem \ref{T:BTOA-NP'}.

\smallskip

Conversely, suppose that ${\mathfrak D}$ is a $\Pi_{+}$-admissible 
BTOA-interpolation data set.  Then we can form the space 
$$
\cM_{\mathfrak D}^{[\perp \cK]} \subset \cK = \sbm{ L^{2}_{p}(i{\mathbb 
R}) \\  \psi^{-1} H^{2}_{m}(\Pi_{+}) }.
$$
Suppose that $\cM_{\mathfrak D}^{[\perp \cK]}$ is a positive subspace of $\cK$.  
By Lemma \ref{L:maxneg}, a subspace $\cG$ of $\cM_{\mathfrak 
D}$ which is maximal negative as a subspace of $\cM_{\mathfrak D}$ is 
also maximal negative as a subspace of $\cK$.  We also saw in the necessity 
argument that if the subspace $\cG$ has the form $\cG_{S}$ \eqref{ad2} 
for a matrix function $S$ 
and $\cG_{S}\subset \cM_{\mathfrak D}$, then $S$ satisfies 
the interpolation conditions \eqref{BTOAint1}, \eqref{BTOAint2}, 
\eqref{BTOAint3}.  However, not all maximal negative subspaces $\cG = 
\sbm{ T \\ I} \psi^{-1} H^{2}_{m}(\Pi_{+})$ 
of $\cK$ have the form $\cG = \cG_{S}$ for a matrix function $S$; the 
missing property is {\em shift-invariance}, i.e.,  one must require 
in addition that $\cG$ is invariant under multiplication by the 
coordinate function $\chi(\lam) = \frac{\lam - 1}{\lam +1}$.  Then 
one gets that $T$ and $M_{\chi}$ commute and one can conclude that $T$ 
is a multiplication operator:  $T = M_{S}$ for some multiplier 
function $S$.  Thus the issue is to construct maximal negative 
subspaces of $\cM_{\mathfrak D}$ (which are then also maximal 
negative as subspaces of $\cK$ by Lemma \ref{L:maxneg}) which are 
also shift-invariant.

To achieve this goal, it is convenient to assume that $\cM_{\mathfrak 
D}^{[\perp \cK]}$ is strictly positive, i.e., that $\cM_{\mathfrak 
D}^{[\perp \cK]}$ is a Hilbert space.  It then follows in particular that 
$\cM_{\mathfrak D}^{[\perp \cK]}$ is regular, i.e., $\cM_{\mathfrak D}^{[\perp \cK]}$ 
and its $J$-orthogonal complement (relative to $\cK$) $\cM_{\mathfrak D}$ form
a $J$-orthogonal decomposition of $\cK$:
$$
  \cK = \cM_{\mathfrak D}  [+]_{J} \cM_{\mathfrak D}^{[\perp \cK]}.
$$
One can argue that one can use an approximation/normal-families 
argument to reduce the general case to this special case, but we do 
not go into details on this point here.  Then results from \cite{BH} 
imply that there is a $J$-Beurling-Lax representer for 
$\cM_{\mathfrak D}$, i.e., there is a $J$-phase function 
$$
\Theta \in L^{2}_{(p+m) \times (p+m)}(i {\mathbb R})\quad\mbox{with}\quad 
\Theta(\lam)^{*} J \Theta(\lam) = J \; \text{ for a.e. } \; \lam \in \Pi_{+}
$$ 
such that \eqref{cMBLrep} holds.  As both 
$$
\cM_{\mathfrak D} \ominus (\cM_{\mathfrak D} \cap 
H^{2}_{p+m}(\Pi_{+}))\quad\mbox{and}\quad H^{2}_{p+m}(\Pi_{+})\ominus(
H^{2}_{p+m}(\Pi_{+}) \cap \cM_{\mathfrak D})
$$ 
are finite-dimensional, in 
fact one can show that $\Theta$ is rational and bounded on $i 
{\mathbb R}$.  Then the multiplication operator $M_{\Theta} \colon k 
\mapsto \Theta \cdot k$ is a Kre\u{\i}n-space isomorphism from 
$H^{2}_{p+m}(\Pi_{+})$ (a Kre\u{\i}n space with inner product induced 
by $J = \sbm{ I_{p} & 0 \\ 0 & -I_{m}}$) onto $\cM_{\mathfrak D}$ 
which also intertwines the multiplication operator $M_{\chi}$ on the 
respective spaces.  It follows that shift-invariant $\cM_{\mathfrak 
D}$-maximal-negative subspaces $\cG$ are exactly those of the form
$$
\cG = \Theta \cdot \begin{bmatrix} G \\ I_{m} \end{bmatrix} \cdot 
H^{2}_{m}(\Pi_{+}), \;  \text{ where } G \in \;  \cS^{p \times m}(\Pi_{+}).
$$
By the preceding analysis, any such subspace $\cG$ also has the form
$$
  \cG = \begin{bmatrix} S  \\ I_{m} \end{bmatrix} \cdot \psi^{-1} 
  H^{2}_{m}(\Pi_{+})
$$
where $S \in \cS^{p \times m}(\Pi_{+})$ is a Schur-class solution of 
the interpolation conditions \eqref{BTOAint1}, \eqref{BTOAint2}, 
\eqref{BTOAint3}.  Moreover one can reverse this analysis to see that 
any solution $S$ of the BTOA-NP interpolation problem arises in this 
way from a $G \in \cS^{p \times m}(\Pi_{+})$.  From the subspace 
equality
$$
\begin{bmatrix} S  \\ I_{m} \end{bmatrix} \cdot \psi^{-1} 
  H^{2}_{m}(\Pi_{+}) = \begin{bmatrix} \Theta_{11} & \Theta_{12} \\ 
  \Theta_{21} & \Theta_{22} \end{bmatrix} \cdot \begin{bmatrix} G \\ I_{m} 
\end{bmatrix} \cdot 
H^{2}_{m}(\Pi_{+})
$$
one can solve for $S$ in terms of $G$: in particular we have
$$
\begin{bmatrix} S  \\ I_{m} \end{bmatrix} \cdot \psi^{-1} I_{m} 
    \in \begin{bmatrix} \Theta_{11} & \Theta_{12} \\ 
  \Theta_{21} & \Theta_{22} \end{bmatrix} \cdot \begin{bmatrix} G \\ I_{m} 
\end{bmatrix} \cdot  H^{2}_{m}(\Pi_{+}),
$$
so there must be a function $Q \in H^{\infty}_{m \times m}(\Pi_{+})$ 
so that
$$ 
\begin{bmatrix} S  \\ I_{m} \end{bmatrix} \cdot \psi^{-1} I_{m} 
   = \begin{bmatrix} \Theta_{11} & \Theta_{12} \\ 
  \Theta_{21} & \Theta_{22} \end{bmatrix} \cdot \begin{bmatrix} G \\ I_{m} \end{bmatrix}
  \cdot Q.
$$
As we saw in Section 2, the latter equality (which is the same as \eqref{ad3}) implies the 
representation formula \eqref{LFTparam} for the set of solutions $S$.
This completes the proof of Theorem \ref{T:BTOA-NP'}.
\end{proof}

\begin{remark}  \label{R:no-wno}
Note that in this Grassmannian/Kre\u{\i}n-space approach
we have not even mentioned that the $J$-phase $\Theta$ is actually $J$-inner 
(i.e., $\Theta(\lam)$ is $J$ contractive at its points of analyticity 
in $\Pi_{+}$);  this condition and the winding number argument in the 
proof via the state-space approach in Section \ref{S:proof1} have been 
replaced by the condition that $\cM_{\mathfrak D}^{[\perp \cK]}$ is a 
positive subspace and consequences of this assumption coming out of 
Lemma \ref{L:maxneg}.
\end{remark}

\section{State-space versus Grassmannian/Kre\u{\i}n-space-geometry 
solution criteria}  \label{S:synthesis}

Assume that we are given a $\Pi_{+}$-admissible interpolation data 
set ${\mathfrak D}$ with $\bGamma_{\mathfrak D}$ invertible.
When we combine the results of Theorems \ref{T:BTOA-NP}, 
\ref{T:BTOA-NP'} and \ref{T:RKHS}, we see immediately that $\bGamma_{\mathfrak D} 
\succ 0$ if and only if the subspace $\cM_{\mathfrak D}^{[\perp]}$ 
is positive as a subspace of the Kre\u{\i}n-space $\cK$ 
\eqref{defcK}, since each of these two conditions is equivalent to 
the existence of solutions for the BTOA-NP interpolation problem with 
data set ${\mathfrak D}$.  It is not too much of a stretch to 
speculate that the strict positive definiteness of $\bGamma_{\mathfrak D}$ 
is equivalent to strict positivity of $\cM_{\mathfrak D}^{[\perp \cK]}$.
Furthermore, in the case where $\bGamma_{\mathfrak D}$ is invertible, 
by the analysis in Section \ref{S:RKHS} we know that 
positive-definiteness of $\bGamma_{\mathfrak D}$ is equivalent to 
positivity of the kernel $K_{\Theta,J}$ \eqref{ad1}, or to the 
reproducing kernel space $\cH(K_{\Theta,J})$ being a Hilbert space.
The goal of this section is to carry out some additional geometric 
analysis to verify these equivalences for the nondegenerate case 
($\bGamma_{\mathfrak D}$ invertible)  directly.

\begin{corollary}  \label{C:BTOA-NP}  Suppose that ${\mathfrak D}$ is 
    a $\Pi_{+}$-admissible BTOA interpolation data set, let 
    $\bGamma_{\mathfrak D}$ be the matrix given in \eqref{GammaData} and let 
    $\cM_{\mathfrak D}^{[\perp \cK]} \subset \cK$ be the subspace defined in
    \eqref{cPdata}.  Then the following are equivalent:
  \begin{enumerate}
 \item $\bGamma_{\mathfrak D} \succ 0$.
\item $\cM_{\mathfrak D}^{[\perp \cK]}$ is a strictly positive subspace 
      of $\cK$ (i.e., $\cM^{[ \perp \cK]}$ is a Hilbert space in the 
      $J$-inner product).
\item The reproducing kernel Pontryagin space $\cH(K_{\Theta,J})$ is 
actually a Hilbert space.
\end{enumerate}
\end{corollary}

\smallskip

\begin{proof} For simplicity we consider first the case where the data set 
    $\mathfrak D$ has the form
\begin{equation}  \label{dataL}
{\mathfrak D}_{L} = (Z, X, Y; \emptyset, \emptyset,  \emptyset;\emptyset),
\end{equation}
i.e., there are only Left Tangential interpolation conditions 
\eqref{BTOAint1}. 

\smallskip

\noindent
\textbf{Case 1: The LTOA setting.}
In case $\mathfrak D$ has the form ${\mathfrak D} = {\mathfrak 
D}_{L}$ as in \eqref{dataL}, the matrix $\bGamma_{\mathfrak D}$ collapses down to 
$\bGamma_{{\mathfrak D}_{L}} = \Gamma_{L}$ and $\cM_{{\mathfrak D}_{L}}$ collapses down to
$$
    \cM_{{\mathfrak D}_{L}} = \left\{ \begin{bmatrix} f \\ g 
\end{bmatrix} 
 \in H^{2}_{p+m}(\Pi_{+}) \colon
\left( \begin{bmatrix}X & -Y \end{bmatrix} \begin{bmatrix} f \\ g 
\end{bmatrix} \right)^{\wedge L}(Z) = 0 \right\}.
$$
Furthermore, in the present case, $\cM_{{\mathfrak D}_{L},-}=H^2_m(\Pi_+)$ and 
therefore,  $\cK$ given by \eqref{defcK} is simply
 $\cK = \sbm{L^{2}_{p}(i {\mathbb R}) \\ H^{2}_{m}(\Pi_{+})}$.

\smallskip

 We view the map $\sbm{ f \\ g} \mapsto \left(\sbm{X & -Y } \sbm{ f 
 \\ g} \right)^{\wedge L}(Z)$ as an operator
 $$
  \cC_{Z, \sbm{ X & -Y}} \colon H^{2}_{p+m}(\Pi_{+}) \to {\mathbb 
  C}^{n_{Z}} 
 $$
 which can be written out more explicitly as an integral operator 
 along the imaginary line:\footnote{We view operators of this form as 
 {\em control-like} operators; they and their cousins ({\em observer-like} operators)
will be discussed in a broader context as part of the analysis of Case 2 to 
come below.}
$$ 
  \cC_{Z, \, \sbm{ X & -Y }} \colon
  \begin{bmatrix} f_{+} \\ f_{-} \end{bmatrix} \mapsto
  \frac{1}{2 \pi} \int_{-\infty}^{\infty}
- (iyI - Z)^{-1} \begin{bmatrix} X & -Y \end{bmatrix} \begin{bmatrix} 
f_{+}(iy) \\ f_{-}(iy) \end{bmatrix} \, dy.
$$
Then we can view $\cM_{{\mathfrak D}_{L}}$ as an operator kernel:
 $$
  \cM_{{\mathfrak D}_{L}} = \operatorname{Ker} \,  \cC_{Z, \sbm{ X & -Y}}.
  $$
 We are actually interested in the $J$-orthogonal complement 
\begin{align*}
 \cM_{{\mathfrak D}_{L}}^{[\perp \cK]}:=\cK [-]_J \cM_{{\mathfrak D}_{L}}
&=\sbm{L^{2}_{p}(i {\mathbb R}) \\ H^{2}_{m}(\Pi_{+})} [-]_J \cM_{{\mathfrak D}_{L}}\\
&=\sbm{ H^{2}_{p}(\Pi_{-}) \\ 0 } \oplus\left( H^{2}_{p+m}(\Pi_+)
 [-]_{J} \cM_{{\mathfrak D}_{L}}\right).
\end{align*}
 As the subspace $\sbm{ H^{2}_{p}(\Pi_{-}) \\ 0 }$ is clearly 
 positive, we see that   $\cM_{{\mathfrak D}_{L}}^{[\perp \cK]}$ is 
 positive if and only if its subspace  
$$
\cM_{{\mathfrak  D}_{L}}^{[\perp H^{2}_{p+m}(\Pi_{+})]}: = H^{2}_{p+m}(\Pi_{+}) 
 [-]_{J} \cM_{{\mathfrak D}_{L}}
$$
is positive. By standard operator-theory duality, we 
 can express the latter (finite-dimensional and hence closed) subspace as an operator range:
 $$
 \cM_{{\mathfrak D}_{L}}^{[\perp H^{2}_{p+m}(\Pi_{+})]} = \operatorname{Ran} J \left(\cC_{Z,\sbm{ X & -Y}} 
 \right)^{*},
 $$
 where the adjoint is with respect to the standard Hilbert-space 
 inner product on $H^{2}_{p+m}(\Pi_{+})$ and the standard Euclidean 
 inner product on ${\mathbb C}^{n_{Z}}$. 
 One can compute the adjoint $\left(\cC_{Z, \sbm{ X & -Y}} 
 \right)^{*}: \; {\mathbb C}^{n_{Z}}\to \sbm{ H^{2}_{p}(\Pi_+) \\
H^{ 2}_{m}(\Pi_+) }$ explicitly as 
$$
(\cC_{Z, \sbm{ X & -Y}})^{*} \colon x \mapsto 
\begin{bmatrix} -X^{*} \\ Y^{*} \end{bmatrix} (\lam I + Z^{*})^{-1} x.
$$
Then the Kre\u{\i}n-space orthogonal complement $H^{2}_{p+m}(\Pi_{+}) [-]_{J} \cM_{{\mathfrak D}_{L}}$ 
can be identified with
\begin{align}  
\cM_{{\mathfrak D}_{L}}^{[\perp H^{2}_{p+m}(\Pi_{+}) ]} &=
J \cdot {\rm Ran} (\cC_{Z, \sbm{ X & -Y}})^{*} \notag\\
&=\left\{ \begin{bmatrix} -X^{*} \\ -Y^{*} \end{bmatrix} (\lam I + 
Z^{*})^{-1} x \colon x \in {\mathbb C}^{n_Z} \right\}. \label{cMDL}
\end{align}

To characterize when $\cM_{{\mathfrak D}_{L}}^{[\perp 
H^{2}_{p+m}(\Pi_{+})]}$ is a 
positive subspace, it suffices to compute the Kre\u{\i}n-space 
inner-product gramian matrix ${\mathbb G}$ for $\cM_{{\mathfrak 
D}_{L}}^{[\perp H^{2}_{p+m}(\Pi_{+})]}$ with respect to its parametrization by 
${\mathbb C}^{n_{Z}}$ in \eqref{cMDL}:
\begin{align*}
& \langle {\mathbb G} x, x' \rangle_{{\mathbb C}^{n_{Z}}}  \\
& \quad = 
 \frac{1}{2 \pi}\left\langle J \begin{bmatrix} -X^{*} \\ -Y^{*} 
\end{bmatrix} (\lam I + Z^{*})^{-1} x,\, 
\begin{bmatrix} -X^{*} \\ -Y^{*} \end{bmatrix} (\lam I 
+ Z^{*})^{-1} x' \right\rangle_{H^{2}_{p+m}(\Pi_{+})} \\
&  \quad  = \frac{1}{2 \pi}\int_{-\infty}^{\infty}\langle (-iy I + Z)^{-1} (X X^{*} -  
Y Y^{*}) (iyI + Z^{*})^{-1} x,\, x' \rangle_{{\mathbb C}^{n_{Z}}}\, dy. 
\end{align*}
Thus ${\mathbb G}$ is given by
$$   
{\mathbb G} = \frac{1}{2 \pi} \int_{-\infty}^{\infty} (-iyI + Z)^{-1} (X X^{*} - Y 
Y^{*}) (iyI + Z^{*})^{-1}\, dy.
$$
Introduce the change of variable $\zeta = i y$, $d\zeta = i \, dy$ to 
write this as a complex line 
integral
\begin{align*}
 {\mathbb G} & = \frac{1}{2 \pi i} \lim_{R \to \infty} \int_{\Gamma_{R,1}} 
 (-\zeta I + Z)^{-1} (X X^{*} - 
  Y Y^{*}) (\zeta I + Z^{*})^{-1}\, d \zeta \\
  & =  \frac{1}{2 \pi i}  \lim_{R \to \infty} 
  \int_{-\Gamma_{R,1}} (\zeta I - Z)^{-1} (X  X^{*} - YY^{*}) (\zeta I 
  + Z^{*})^{-1}\, d\zeta
\end{align*}
where $\Gamma_{R,1}$ is the straight line from $-iR$ to $iR$ and 
$-\Gamma_{R,1}$ is the same path but with reverse orientation (the 
straight line from $iR$ to $-iR$).  Since the integrand
\begin{equation}   \label{fzeta}
 f(\zeta) =
(\zeta I - Z)^{-1} (X  X^{*} - YY^{*}) (\zeta I  + Z^{*})^{-1}
\end{equation}
satisfies an estimate of the form
$ \| f(\zeta \| \le \frac{M}{ |\zeta|^{2}}\; $ as $\; |\zeta| \to \infty$, 
it follows that
$$
\lim_{R \to \infty} \int_{\Gamma_{R,2}}  (\zeta I - Z)^{-1} (X  X^{*} - YY^{*}) (\zeta I 
  + Z^{*})^{-1}\, d\zeta = 0
$$
where $\Gamma_{R,2}$ is the semicircle of radius $R$ with counterclockwise 
orientation starting at the point $-iR$ and ending at the point 
$iR$ (parametrization:  $\zeta = R e^{i \theta}$ with $-\pi/2 \le 
\theta \le \pi/2$).  Hence we see that  
$$
{\mathbb G} = \frac{1}{ 2 \pi i} \lim_{R \to \infty} \int_{\Gamma_{R}} 
(\zeta I - Z)^{-1} (X  X^{*} - YY^{*}) (\zeta I 
  + Z^{*})^{-1}\, d\zeta
$$
where $\Gamma_{R}$ is the simple closed curve $-\Gamma_{R,1} + 
\Gamma_{R,2}$.  By the residue theorem, this last expression is 
independent of $R$ once $R$ is so large that all the RHP poles of the 
integrand $f(\zeta)$ \eqref{fzeta} are inside the curve $\Gamma_{R}$, 
and hence
$$
{\mathbb G} =  \frac{1}{2 \pi i} \int_{\Gamma_{R}}
(\zeta I - Z)^{-1} (X X^{*} - Y Y^{*} ) (\zeta I + Z^{*})^{-1}\, 
d\zeta 
$$
for any $R$ large enough. This enables us to compute ${\mathbb G}$ via residues:
\begin{equation}   \label{bbG}
    {\mathbb G} = \sum_{z_{0} \in\Pi_+} {\rm Res}_{\zeta = z_{0}}
(\zeta I - Z)^{-1} (X X^{*} - Y Y^{*} ) (\zeta I +  Z^{*})^{-1}.
\end{equation}

We wish to verify that ${\mathbb G}$ satisfies the Lyapunov equation
\begin{equation}  \label{LyapunovG}
    {\mathbb G} Z^{*} + Z {\mathbb G} = X X^{*} - Y Y^{*}.
\end{equation}
Toward this end let us first note that
\begin{align*}
&(\zeta I - Z)^{-1}A(\zeta I + Z^{*})^{-1}Z^*+Z(\zeta I - Z)^{-1}A
(\zeta I + Z^{*})^{-1}\\
&=(\zeta I - Z)^{-1}A-A(\zeta I + Z^{*})^{-1}
\end{align*}
for any $A\in\mathbb C^{n_Z\times n_Z}$. Making use of the latter equality with 
$A=X X^{*} - YY^{*}$ 
we now deduce from the formula \eqref{bbG} for ${\mathbb G}$ that
\begin{align*}
    {\mathbb G} Z^{*} + Z {\mathbb G} & = 
  \sum_{z_{0} \in \Pi_+} {\rm Res}_{\zeta = z_{0}} 
 \left( (\zeta I - Z)^{-1} (XX^{*} - Y Y^{*})  \right. \\
 & \quad \quad \quad \quad  \quad \quad \left. - (XX^{*} - Y Y^{*}) 
 (\zeta I + Z^{*})^{-1} \right) \\
 & =  I \cdot (XX^{*} - Y Y^{*}) - (XX^{*} - Y Y^{*})  \cdot 0
  = XX^{*} - YY^{*}
\end{align*}
where for the last step we use that $Z$ has all its spectrum in the 
right half plane while $-Z^{*}$ has all its spectrum in the left
half plane; also note that in general the sum of the residues of any 
resolvent matrix $R(\zeta) = ( \zeta I - A)^{-1}$ is the identity 
matrix, due to the Laurent expansion at infinity for $R(\zeta)$:  
$R(\zeta) = \sum_{n=0}^{\infty} A^{n} \zeta^{-n-1}$.
This completes the verification of \eqref{LyapunovG}.

\smallskip

Since both $\Gamma_{L}$ and ${\mathbb G}$ satisfy the same Lyapunov 
equation \eqref{GammaL} which has a 
unique solution since $\sigma(Z) \cap \sigma(-Z^{*}) = \emptyset$, we 
conclude that ${\mathbb G} = \Gamma_{L}$.   This completes the direct proof 
of the equivalence of conditions (1) and (2) in Corollary \ref{C:BTOA-NP}  for the case 
that $\mathfrak D = {\mathfrak D}_{L}$.

To make the connection with the kernel $K_{\Theta,J}$, we note that there 
is a standard way to identify a reproducing kernel Hilbert space $\cH(K)$ of a 
particular form with an operator range (see e.g.\ \cite{Sarason} 
or \cite{BB-HOT}).  Specifically, let $M_{\Theta}$ be the 
multiplication operator
$$
 M_{\Theta} \colon f(\lam) \mapsto \Theta(\lam) f(\lam)
$$
acting on $H^{2}_{p+m}(\Pi_{+})$, identify $J$ with $J \otimes 
I_{H^{2}}(\Pi_{+})$ acting on $H^{2}_{p+m}(\Pi_{+})$, and define $W 
\in \cL(H^{2}_{p+m}(\Pi_{+}))$ by
$$
  W = J - M_{\Theta} J (M_{\Theta})^{*}.
$$
For $w \in \Pi_{+}$ and $\by \in {\mathbb C}^{p+m}$, let 
$k_{w,\by}(z) = \frac{1}{z - \overline{w}}\by$ by the kernel element 
associate with the Szeg\H{o} kernel $k_{\rm Sz} \otimes I_{{\mathbb 
C}^{p+m}}$. One can verify
$$
 W k_{w,\by} = K_{\Theta,J}( \cdot, w) \by \in \cH(K_{\Theta, J}),
$$
and furthermore,
\begin{align*}
    \langle W k_{w_{j},\by_{j}}, \, W k_{w_{i},\by_{i}} 
    \rangle_{\cH(K_{\Theta,J})} & =
    \langle K_{\Theta,J}((w_{i},w_{j}) \by_{j},\, \by_{i} 
    \rangle_{{\mathbb C}^{p+m}} \\
    & = \langle W k_{w_{j}, \by_{j}}, \, k_{w_{i}, \by_{i}} 
    \rangle_{H^{2}_{p+m}(\Pi_{+})}.
\end{align*}
As $\Theta$ is rational and $M_{\Theta}$ is a $J$-isometry, one can see that 
$\operatorname{Ran} W$ is already closed. Hence we have 
the concrete identification $\cH(K_{\Theta,J}) = \operatorname{Ran} 
W$ with lifted inner product
$$
\langle W f, W g\rangle_{\cH(K_{\Theta,J})} = \langle W f, g 
\rangle_{H^{2}_{p+m}(\Pi_{+})}.
$$
As $M_{\Theta}$ is a $J$-isometry, the operator 
$M_{\Theta} J (M_{\Theta})^{*} =: M_{\Theta}(M_{\Theta})^{[*]}$ is 
the $J$-selfadjoint projection onto $\Theta \cdot 
H^{2}_{p+m}(\Pi_{+})$ and $WJ = I - M_{\Theta} (M_{\Theta})^{[*]}$ is 
the $J$-self-adjoint projection onto $H^{2}_{p+m} [-] \Theta \cdot 
H^{2}_{p+m}(\Pi_{+}) = \cM_{\mathfrak D}^{[\perp \cK]}$.  We then see 
that, for all $f,g \in H^{2}_{p+m}(\Pi_{+})$,
$$\langle W J f, W J g \rangle_{\cH(K_{\Theta,J})} =
 \langle  W J f, J g \rangle_{H^{2}_{p+m}(\Pi_{+})} =
\langle J \cdot WJ f, WJg \rangle_{H^{2}_{p+m}(\Pi_{+})},
$$
i.e., the identity map is a Kre\u{\i}n-space isomorphism between 
$\cH(K_{\Theta,J})$ and $\cM_{\mathfrak D}^{[ \perp \cK]}$ with the 
$J$-inner product.
In particular, we arrive at the equivalence of conditions (2) and (3) 
in Corollary \ref{C:BTOA-NP} for Case 1.

\smallskip

\noindent
\textbf{Case 2: The general BTOA setting:}  
To streamline formulas to come, we introduce two types of 
control-like operators and two types of observer-like operators
as follows (for fuller details and systems-theory motivation, we 
refer to \cite{BR} for the discrete-time setting and \cite{Amaya} for 
the continuous-time setting). Suppose that $(A,B)$ is an input pair of matrices 
(so $A$ has, say, 
size $N \times N$ and $B$ has size $N \times n$).  We assume that either $A$ is 
stable ($\sigma(A) 
\subset \Pi_{-}$) or $A$ is antistable ($\sigma(A) \subset \Pi_{+}$).  
In case $\sigma(A) \subset \Pi_{+}$, we define a 
control-like operator as appeared in the Case 1 analysis
$$
\cC_{A,B} \colon H^{2}_{n}(\Pi_{+}) \to {\mathbb C}^{N}
$$
by
\begin{align*}
  \cC_{A,B} \colon g & \mapsto (Bg)^{\wedge L}(A):  = \sum_{z \in \Pi_{+}} {\rm 
  Res}_{\lambda = z} (\lambda I - A)^{-1} B g(\lambda) \\
  & = - \frac{1}{2 \pi} \int_{-\infty}^{\infty} (iy I - A)^{-1} B 
  g(iy)\, dy.
\end{align*}
In case $\sigma(A) \subset \Pi_{-}$, we define a complementary 
control-like operator 
$$
\cC^{\times}_{A,B} \colon H^{2}_{n}(\Pi_{-}) \to {\mathbb C}^{N}
$$
by 
\begin{align*} 
 \cC^{\times}_{A,B} \colon g & \mapsto (Bg)^{\wedge L}(A):  = \sum_{z 
 \in \Pi_{-}} {\rm Res}_{\lambda = z} (\lambda I - A)^{-1} B g(\lambda) \\
  & =  \frac{1}{2 \pi} \int_{-\infty}^{\infty} (iy I - A)^{-1} B 
  g(iy)\, dy.
\end{align*}
Suppose next that $(C,A)$ is an output-pair, say of respective sizes 
$n \times N$ and $N \times N$, and that $A$ is either stable or 
antistable.  In case $A$ is antistable ($\sigma(A) \subset \Pi_{+}$), 
we define the observer-like operator 
$$
  \cO_{C,A} \colon {\mathbb C}^{N} \to H^{2}_{n}(\Pi_{-})
$$
by
$$
\cO_{C,A} \colon x \mapsto C (\lambda I - A)^{-1} x.
$$
In case $A$ is stable (so $\sigma(A) \subset \Pi_{-})$, then the 
complementary observer-like operator  is given by the same formula 
but maps to the complementary $H^{2}$ space:
$$
\cO^{\times}_{C,A} \colon {\mathbb C}^{N} \mapsto  H^{2}_{n}(\Pi_{+})
$$
given again by
$$
  \cO^{\times}_{C,A} \colon x \mapsto C (\lambda I - A)^{-1} x.
$$
We are primarily interested in the case where $A$ is antistable and 
we consider the operators $\cC_{A,B} \colon H^{2}_{n}(\Pi_{+}) \to 
{\mathbb C}^{N}$ and $\cO_{C,A} \colon {\mathbb C}^{N} \mapsto 
H^{2}_{n}(\Pi_{-})$.  However a straightforward exercise is to show 
that the complementary operators come up when computing adjoints:
for $A$ antistable, $-A^{*}$ is stable and we have the formulas
$$
(\cO_{C,A})^{*} = -\cC^{\times}_{-A^{*}, C^{*}} \colon 
H^{2}_{n}(\Pi_{-}) \mapsto {\mathbb C}^{N}, \quad
(\cC_{A,B})^{*} = \cO^{\times}_{B^{*}, -A^{*}} \colon {\mathbb C}^{N} 
\mapsto H^{2}_{n}(\Pi_{+}).
$$

Assume now that 
$\cM_{\mathfrak D} \subset L^{2}_{p+m}(\Pi_{+})$ is defined as in 
\eqref{cMrep} for a $\Pi_{+}$-admissible interpolation data set 
${\mathfrak D} = (U,V,W; Z, X, Y; \Gamma)$. Thus $(U,W)$ and $(V,W)$ are 
output pairs with $\sigma(W) \subset \Pi_{+}$ and $(Z,X)$ and $(Z,Y)$ 
are input pairs with $\sigma(Z) \subset \Pi_{+}$.  We therefore have 
observer-like and control-like operators
\begin{align*}
&\cO_{V,W} \colon {\mathbb C}^{n_{W}} \to H^{2}_{p}(\Pi_{-}), 
\quad \cO_{U,W} \colon {\mathbb C}^{n_{W}} \to H^{2}_{m}(\Pi_{-}), \\
& \cC_{Z,X} \colon H^{2}_{p}(\Pi_{+}) \to {\mathbb C}^{n_{Z}}, \quad\quad
\cC_{Z,Y} \colon H^{2}_{m}(\Pi_{+}) \to {\mathbb C}^{n_{Z}}
\end{align*}
defined as above, as well as the observer-like and control-like operators
$$
 \cO_{\sbm{ V \\ U}, W} : = \begin{bmatrix} \cO_{V,W} \\ \cO_{U,W} 
 \end{bmatrix},\quad 
 \cC_{Z, [ X \; \;  -Y]} = \begin{bmatrix} 
 \cC_{Z,X} & - \cC_{Z,Y} \end{bmatrix}.
 $$
Then the adjoint  operators have the form
 \begin{align*}
&(\cO_{V,W})^{*} = -\cC^{\times}_{-W^{*}, V^{*}} \colon H^{2}_{p}(\Pi_{-}) \to
{\mathbb C}^{n_{W}}, \\ 
& (\cO_{U,W})^{*} = -\cC^{\times}_{-W^{*}, U^{*}} \colon 
H^{2}_{m}(\Pi_{-}) \to {\mathbb C}^{n_{W}}, \\
& (\cC_{Z,X})^{*} = \cO^{\times}_{X^{*}, -Z^{*}} \colon {\mathbb 
C}^{n_{Z}} \to H^{2}_{p}(\Pi_{+}), \\
& (\cC_{Z,Y})^{*} = \cO^{\times}_{Y^{*}, -Z^{*}} \colon {\mathbb C}^{n_{Z}}
\to H^{2}_{m}(\Pi_{+})  
\end{align*}
and are given explicitly by: 
\begin{align*}
  & (\cO_{V,W})^{*} \colon g_{1} \mapsto  
  -\frac{1}{2 \pi} \int_{-\infty}^{\infty} (iy I + W^{*})^{-1} V^{*} 
  g_1(iy)\, dy, \\
  &  (\cO_{U,W})^{*} \colon g_{2} \mapsto  
  -\frac{1}{2 \pi} \int_{-\infty}^{\infty} (iy I + W^{*})^{-1} 
  U^{*}g_2(iy)\, dy, \\
& (\cC_{Z,X})^{*} \colon x \mapsto  X^{*} (\lam I + Z^{*})^{-1} x,
\quad (\cC_{Z,Y})^{*} \colon x \mapsto  Y^{*} (\lam I + Z^{*})^{-1} x.
  \end{align*}
Furthermore one can check via computations as in the derivation of 
\eqref{bbG} above that the $J$-observability and  $J$-controllability gramians
\begin{align*}  & \cG^{J}_{Z,\sbm{X & -Y}} : = 
    \cC_{Z,X} \cC_{Z,X}^{*} - \cC_{Y,Z} \cC_{Z,Y}^{*} =: \cG_{Z,X} - 
    \cG_{Z,Y}, \\
  & \cG^{J}_{\sbm{ V \\ U}, W} : = \cO^{*}_{V,W} \cO_{V,W} - 
  \cO_{U,W}^{*} \cO_{U,W} =: \cG_{V,W} - \cG_{U,W}
\end{align*}
satisfy the respective Lyapunov equations
\begin{align*}
    & \cG^{J}_{Z,\sbm{ X & -Y}} Z^{*} + Z \cG^{J}_{Z,\sbm{ X & -Y}} 
    = X X^{*} - Y Y^{*}, \\
    & \cG^{J}_{\sbm{ V \\ U}, W} W + W^{*} \cG^{J}_{\sbm{V \\ U}, W} = 
    V^{*}V - U^{*} U.
\end{align*}
Hence, by the uniqueness of such solutions and the characterizations 
of $\Gamma_{L}$ and $\Gamma_{R}$ in \eqref{GammaL}, \eqref{GammaR},
we get 
\begin{equation} \label{GammaLRid}
    \cG^{J}_{\sbm{X & -Y}, Z} = \Gamma_{L}, \quad
    \cG^{J}_{\sbm{V \\ U}, W} = -\Gamma_{R}.
\end{equation}
Then the representation \eqref{cMrep} for $\cM_{\mathfrak D}$ can be rewritten more 
succinctly as
\begin{align}  
    \cM_{\mathfrak D} = &\left\{ \cO_{\sbm{V \\ U}, W} x + \sbm{f_{1} 
    \\ f_{2}} \colon x \in {\mathbb C}^{n_{W}} \text{ and } \sbm{ 
    f_{1} \\ f_{2} }  \in H^{2}_{p+m}(\Pi_{+})  \right. \notag \\
& \left. \quad \quad  \text{such that } \cC_{Z,\sbm{ X & -Y}} \sbm{ f_{1} \\ f_{2}} = 
 \Gamma x \right\}.   
\label{cMrep'}
\end{align}
It is readily seen from the latter formula that 
    \begin{align}
P_{H^{2}_{p+m}(\Pi_-)} \cM_{\mathfrak D}&=\operatorname{Ran}\cO_{\sbm{ V \\ U}, W},
\label{ad8} \\
\cM_{\mathfrak D} \cap H^{2}_{p+m}(\Pi_{+}) &=
\operatorname{Ker} \cC_{Z, \sbm{ X & -Y}}, \notag
\end{align}
and therefore, 
$$
\cM_{\mathfrak D} \cap \sbm{ H^{2}_{p}(\Pi_{+}) \\ 0 } =
\sbm{ \operatorname{Ker} \cC_{Z, X} \\ 0 }, \quad
\cM_{\mathfrak D} \cap \sbm{0 \\ H^{2}_{m}(\Pi_{+})} =
\sbm{ 0 \\ \operatorname{Ker} \cC_{Z, Y} }.
$$

\begin{lemma}  \label{L:cMDperp} If $\cM_{\mathfrak D}$ is given by 
    \eqref{cMrep'}, then the $J$-orthogonal complement 
$\cM_{\mathfrak D}^{[\perp]} = L^{2}_{p+m}(i {\mathbb R}) [-]_{J} \cM_{\mathfrak D}$
with respect to the space $L^{2}_{p+m}(i {\mathbb R})$
is given by
\begin{align} 
    \cM_{\mathfrak D}^{[\perp]} = &\left\{ J( \cC_{Z, \sbm{ X & 
    -Y}})^{*} y + \sbm{ g_{1} \\ g_{2}} \colon y \in {\mathbb 
    C}^{n_{Z}} \text{ and } \sbm{ g_{1} \\ g_{2}} \in 
    H^{2}_{p+m}(\Pi_{-}) \right. \notag \\
& \left. \quad \quad \text{such that } \; (\cO_{V,W})^{*} g_{1} - (\cO_{U,W})^{*} g_{2} = 
-\Gamma^{*} y \right\}.
\label{cMperprep'}
\end{align}
    \end{lemma}
    
\begin{proof} Since $\cM_{\mathfrak D}^{[\perp]}$ is
    $J$-orthogonal to $\cM_{\mathfrak D} \cap H^{2}_{p+m}(\Pi_{+}) =
    \operatorname{Ker} \cC_{Z, \sbm{ X & -Y}}$, it follows that 
$P_{H_{p+m}^{2}(\Pi_+)} \cM_{\mathfrak D}^{[\perp]}$ is also 
    $J$-orthogonal to $\operatorname{Ker} \cC_{Z, \sbm{ X & -Y}}$.  Hence
    $P_{H^{2}_{p+m}(\Pi_+)} \cM_{\mathfrak D}^{[\perp]} \subset J 
    \operatorname{Ran}( (\cC_{Z, \sbm{X & -Y}})^{*}$ and 
    each $\bg\in\cM_{\mathfrak D}^{[\perp]}$ has the form 
$$
\bg = J(\cC_{Z, \sbm{X & -Y}})^{*}y + \sbm{ g_{1} \\ g_{2}}\quad\mbox{with}\quad 
    y \in {\mathbb C}^{n_{Z}} \; \; \mbox{and} \; \;  \sbm{g_{1} \\ g_{2}} \in H^{2 
    \perp}_{p+m}(\Pi_{+}).
$$
  For such an element to be in 
    $\cM_{\mathfrak D}^{[\perp]}$, we compute the $J$-inner product of such an 
    element against a generic element of $\cM_{\mathfrak D}$: for all 
    $\sbm{f_{1} \\ f_{2}} \in H^{2}_{p+m}(\Pi_{+})$ and $x \in 
    {\mathbb C}^{n_{Z}}$ such that $\cC_{Z, X} f_{1} - \cC_{Z, Y} 
    f_{2} = \Gamma x$, we must have
 \begin{align*}
     0 & = \left\langle J \left( J (\cC_{Z, \sbm{ X & -Y}})^{*} y + 
     \sbm{ g_{1} \\ g_{2}} \right), \cO_{\sbm{V \\ U}, W} x + \sbm{ 
     f_{1} \\ f_{2}} \right\rangle_{L^{2}_{p+m}(i {\mathbb R})} \\
     & = \langle y, \cC_{Z,X} f_{1} - \cC_{Z,Y} f_{2} 
     \rangle_{{\mathbb C}^{n_{_Z}}} +
  \langle (\cO_{V,W})^{*} g_{1} - (\cO_{U,W})^{*} g_{2}, x 
  \rangle_{{\mathbb C}^{n_{_W}}}  \\
  & = \langle y, \Gamma x \rangle_{{\mathbb C}^{n_{Z}}} + 
   \langle (\cO_{V,W})^{*} g_{1} - (\cO_{U,W})^{*} g_{2}, x 
  \rangle_{{\mathbb C}^{n_{_W}}} 
\end{align*} 
which leads to the coupling condition
$(\cO_{V,W})^{*} g_{1} - (\cO_{U,W})^{*} g_{2} = - \Gamma^{*} y$ in \eqref{cMperprep'}.
\end{proof}
As a consequence of the representation \eqref{cMperprep'} we see that
\begin{align}
& P_{H^{2}_{p+m}(\Pi_{+})} \cM_{\mathfrak D}^{[\perp]} =
\operatorname{Ran} J  (\cC_{Z,\sbm{ X & -Y}})^{*},  \notag\\
& \cM_{\mathfrak D}^{[\perp]} \cap H^{2}_{p+m}(\Pi_{-}) =
\operatorname{Ker} \begin{bmatrix} (\cO_{V,W})^{*} & - 
(\cO_{U,W})^{*} \end{bmatrix}
\label{ad9}
\end{align}
and therefore,
$$
 \cM_{\mathfrak D}^{[\perp]} \cap \sbm{ H^{2}_{p}(\Pi_{-}) \\ 0 } =
\sbm{\operatorname{Ker} (\cO_{V,W})^{*} \\ 0},\quad 
 \cM_{\mathfrak D}^{[\perp]} \cap \sbm{ 0 \\ H^{2}_{m}(\Pi_{-})} =
\sbm{ 0 \\ \operatorname{Ker} (\cO_{U,W})^{*}}.
$$
In this section we shall impose an additional assumption:
\smallskip

\noindent
\textbf{Nondegeneracy assumption:}
{\sl Not only $\cM_{\mathfrak D}$ but also $\cM_{\mathfrak D} \cap H^{2}_{p+m}(\Pi_{+})$ 
and $\cM_{\mathfrak D}^{[
\perp]} \cap H^{2}_{p+m}(\Pi_{-})$ (see \eqref{ad8} and 
\eqref{ad9}) are regular subspaces (i.e., have good Kre\u{\i}n-space 
orthogonal complements---as explained in Section \ref{S:Kreinprelim}) of the
Kre\u{\i}n space $L^{2}_{p+m}(\Pi_{+})$ (with the $J$-inner product).}

\smallskip

We proceed via a string of lemmas.

\begin{lemma} \label{L:decom}
{\rm (1)} The space $\cM_{\mathfrak D}$ given in \eqref{cMrep'} decomposes as
\begin{equation}   \label{Mdecom}
  \cM_{\mathfrak D} = \widehat{\mathbb G}_{T} [+] \cM_{{\mathfrak D},1} [+] 
  \cM_{{\mathfrak D},2},
\end{equation}
where 
\begin{align}
    \widehat{\mathbb G}_{T} &= \cM_{\mathfrak D} [-]_{J} 
    \operatorname{Ker} \cC_{Z, \sbm{ X & -Y}}, \notag \\
    \cM_{{\mathfrak D},1} &= \operatorname{Ker} \cC_{Z, \sbm{ X & -Y}}[-]_{J} 
\left(\sbm{\operatorname{Ker} \cC_{Z, X}\\ 0} \oplus  
\sbm{0\\ \operatorname{Ker} \cC_{Z, Y}}\right), \notag \\
     \cM_{{\mathfrak D}, 2} &=
\sbm{\operatorname{Ker} \cC_{Z, X}\\ 0} \oplus  \sbm{0\\ \operatorname{Ker} \cC_{Z, Y}}.
    \label{def012}
\end{align}
More explicitly, the operator $T: \operatorname{Ran}
\cO_{\sbm{ V \\ U}, W}\to \operatorname{Ran} J (\cC_{Z,\sbm{X & -Y}})^{*}$ 
is uniquely determined by the identity
\begin{equation}  \label{TGamma}
    \cC_{Z, \sbm{ X & -Y}} T \cO_{\sbm{ V \\ U}, W} = -\Gamma,
\end{equation}
and  $\widehat{\mathbb G}_{T}$ is the graph space for $-T$ parametrized as
\begin{align}
\widehat{\mathbb G}_{T} & = \left\{- \boldf + T \boldf \colon 
\boldf\in \operatorname{Ran}
\cO_{\sbm{V \\ U}, W}\right\} \notag \\
& =  \left\{- \cO_{\sbm{V \\ U}, W} x + T \cO_{\sbm{V \\ U}, W} x \colon
x \in {\mathbb C}^{n_{W}} \right\},
\label{graphT}
\end{align}
while $\cM_{{\mathfrak D},1}$ is given
explicitly by
\begin{equation}   \label{cMD1}
    \cM_{{\mathfrak D},1} = \operatorname{Ran}  \begin{bmatrix}
   ( \cC_{Z,X})^{*} (\cG_{Z,X})^{-1} \cG_{Z,Y} \\  (\cC_{Z,Y})^{*} \end{bmatrix}.
\end{equation}

{\rm (2)} Dually, the subspace $\cM_{\mathfrak D}^{[\perp]} =
L^{2}_{p+m}(i {\mathbb R}) [-]_{J} \cM_{\mathfrak D}$ 
   decomposes as
  \begin{equation}   \label{Mperpdecom}
    \cM_{\mathfrak D}^{[\perp]} = {\mathbb G}_{T^{[*]}} [+]  
    (\cM_{\mathfrak D}^{[\perp]})_{1} [+] ( \cM_{\mathfrak  D}^{[\perp]})_{2},
  \end{equation}
where
  \begin{align}
  & {\mathbb G}_{T^{[*]}}  = \cM_{\mathfrak D}^{[\perp]} 
   [-]_{J} \operatorname{Ker} \begin{bmatrix} (\cO_{V,W})^{*} & - 
   (\cO_{U,W})^{*} \end{bmatrix},\notag \\
  & (\cM_{\mathfrak D}^{[ \perp]})_{1} = 
  \operatorname{Ker}\begin{bmatrix} (\cO_{V,W})^{*} & - 
  (\cO_{U,W})^{*} \end{bmatrix}
[-]_{J} \left( \sbm{ \operatorname{Ker} (\cO_{V,W})^{*} \\ 0} 
\oplus \sbm{ 0 \\ \operatorname{Ker} 
(\cO_{U,W})^{*} } \right),\notag\\ 
   & (\cM_{\mathfrak D}^{[\perp]})_{2} =
   \sbm{ \operatorname{Ker} (\cO_{V,W})^{*} \\ 0}  \oplus \sbm{ 0 \\ \operatorname{Ker} 
   (\cO_{U,W})^{*}}.    
\label{def012perp}
  \end{align}
More explicitly,
 \begin{align*}
 {\mathbb G}_{T^{[*]}} & = \left\{ \bg + T^{[*]} \bg \colon \bg
  \in \operatorname{Ran} J (\cC_{Z, \sbm{ X & -Y}})^{*} \right\}  \\
  & =
  \left\{ J(\cC_{Z, \sbm{ X & -Y}})^{*} x +  T^{[*]} J (\cC_{Z, \sbm{ X
  & -Y}})^{*} x \colon x \in {\mathbb C}^{n_{Z}}\right\}
  \end{align*}
where  $T^{[*]} = J T^{*} J\colon  \operatorname{Ran} J (\cC_{Z, \sbm{X & -Y}})^{*}
\to \operatorname{Ran} \cO_{\sbm{ V \\ U}, W}$ is the $J$-adjoint  of $T$, and
  \begin{equation}   \label{cMDperp1}
   (\cM_{\mathfrak D}^{[ \perp]})_{1} = \operatorname{Ran} 
   \begin{bmatrix} \cO_{V,W} \\ \cO_{U,W}(\cG_{U,W})^{-1} \cG_{V,W}  \end{bmatrix}.   
   \end{equation}
\end{lemma}

\begin{proof}  By the Nondegeneracy Assumption we can define subspaces
\eqref{def012} and \eqref{def012perp}, so that $\cM_{\mathfrak D}$
and $\cM_{\mathfrak D}^{[\perp]}$ decompose as in \eqref {Mdecom} and 
\eqref{Mperpdecom}, respectively.

\smallskip

Given an element $\bg \in P_{H^{2}_{p+m}(\Pi_-)} \cM_{\mathfrak D}$,
there is an $\boldf\in H^{2}_{p+m}(\Pi_{+})$ so that $-\bg + \boldf \in \cM_{\mathfrak D}$; 
furthermore, one can choose 
$$
\boldf\in H^{2}_{p+m}(\Pi_{+}) [-]_{J} (\cM_{\mathfrak D} \cap H^{2}_{p+m}(\Pi_{+}))=
P_{H^{2}_{p+m}(\Pi_{+})}\cM_{\mathfrak D}^{[\perp]}. 
$$
If $\boldf'$ is another such choice, 
    then $(-\bg + \boldf) - (-\bg + \boldf') = \boldf - \boldf'$ is in
    $\cM_{\mathfrak D} \cap H^{2}_{p+m}(\Pi_{+})$ as well as in 
    $H^{2}_{p+m}(\Pi_{+}) [-]_{J} (\cM_{\mathfrak D} \cap H^{2}_{p+m}(\Pi_{+}))$.  By 
    the Nondegeneracy Assumption, we conclude that $\boldf = \boldf'$.  
    Hence there is a well-defined map $\bg \mapsto \boldf$ defining a 
    linear operator $T$ from 
$$
P_{H^{2}_{p+m}(\Pi_-)} \cM_{\mathfrak D}=\operatorname{Ran}\cO_{\sbm{ V \\ U}, W}
\quad\mbox{into}\quad
P_{H^{2}_{p+m}(\Pi_{+})}\cM_{\mathfrak D}^{[\perp]}=
\operatorname{Ran} J (\cC_{Z,\sbm{X & -Y}})^{*}
$$
(see \eqref{ad8} and \eqref{ad9}).
    In this way we arrive at a well-defined operator $T$ so that 
    $\widehat{\mathbb G}_{T}$ as in \eqref{graphT} is equal to the subspace 
    (see \eqref{def012})
$$
\cM_{\mathfrak D} [-]_{J} \left(\cM_{\mathfrak D} \cap 
H^{2}_{p+m}(\Pi_{+})\right)
=\cM_{\mathfrak D} [-]_{J} \operatorname{Ker} \cC_{Z, \sbm{ X & -Y}}.
$$
To check that $T$ is also given by 
    \eqref{TGamma}, combine the fact that 
    $$
     -\cO_{\sbm{V \\ U}, W} x + T \cO_{\sbm{V \\ U}, W} x\in \cM_{\mathfrak D}
    $$
    together with the characterization \eqref{cMrep'} for 
    $\cM_{\mathfrak D}$ to 
    deduce that
    $$
    \cC_{Z, \sbm{ X & -Y}} \cdot T \cO_{\sbm{V \\ U}, W} x = 
    -\Gamma x
    $$
    for all $x$ to arrive at \eqref{TGamma}.  
    
\smallskip

    To get the formula \eqref{cMD1}, we first note that
\begin{equation}  \label{DecoupledCZXY}
  H^{2}_{p+m}(\Pi_{+}) [-]_{J}\sbm{\operatorname{Ker} \cC_{Z, X} \\  
  \operatorname{Ker} \cC_{Z, Y}}
= \sbm{ \operatorname{Ran} 
  (\cC_{Z,X})^{*} \\ \operatorname{Ran}  (\cC_{Z,Y})^{*} }.
\end{equation}
The space $\cM_{{\mathfrak D}, 1}$ is the intersection of this space 
with $\cM_{\mathfrak D} \cap H^{2}_{p+m}(\Pi_{+})$.
Therefore, it consists of elements of the form $\sbm{ (\cC_{Z,X})^{*} y_{1} \\ 
(\cC_{Z,Y})^{*} y_{2} }$ 
subject to condition
$$
 0 = \begin{bmatrix} \cC_{Z,X} & - \cC_{Z,Y} \end{bmatrix}
 \begin{bmatrix} (\cC_{Z,X})^{*} y_{1} \\ (\cC_{Z,X})^{*} y_{2} 
 \end{bmatrix} = \cC_{Z,X} (\cC_{Z,X})^{*} y_{1} - \cC_{Z,Y} 
 (\cC_{Z,Y})^{*} y_{2}.
$$
By the $\Pi_{+}$-admissibility requirement on the data set ${\mathfrak 
D}$, the gramian $\cG_{Z,X}: = \cC_{Z,X} (\cC_{Z,X})^{*}$ is 
invertible and hence we may solve this last equation for $y_{1}$:
$$
  y_{1} = \cG_{Z,X}^{-1} \cC_{Z,Y} (\cC_{Z,Y})^{*} y_{2}.
$$
With this substitution, the element $\sbm{ (\cC_{Z,X})^{*} y_{1} \\
(\cC_{Z,Y})^{*} y_{2} }$ 
of the $J$-orthogonal complement space \eqref{DecoupledCZXY} assumes 
the form
$$
  \begin{bmatrix} (\cC_{Z,X})^{*} \cG_{Z,X}^{-1} \cC_{Z,Y} 
      (\cC_{Z,Y})^{*} y_{2} \\
      (\cC_{Z,Y})^* y_{2} \end{bmatrix}
$$
and we have arrived at the formula  \eqref{cMD1} for $\cM_{{\mathfrak 
D},1}$.

\smallskip

For the dual case (2), similar arguments starting with the 
representation \eqref{cMperprep'} for $\cM^{[\perp]}_{\mathfrak D}$ show that 
there is an operator $T^{\times}$ from 
$\operatorname{Ran} J (\cC_{Z, \sbm{ X & -Y}})^{*}$ into 
$\cM_{\mathfrak D}^{[\perp]} [-]_{J} \left(\cM_{\mathfrak D}^{[\perp]}  
\cap H^{2}_{p+m}(\Pi_{-}) \right)$ so that
$$
  \cM_{\mathfrak D}^{[\perp]} [-]_{J} \left( \cM_{\mathfrak D}^{[\perp]} 
  \cap H^{2}_{p+m}(\Pi_{-})\right) = (I + T^{\times}) 
  \operatorname{Ran} J (\cC_{Z, \sbm{ X & -Y}})^{*}.
$$
From the characterization \eqref{cMperprep'} of the space 
$\cM_{\mathfrak D}^{[\perp]}$ we see that
the condition 
$$
 J (\cC_{Z, \sbm{ X & -Y}})^{*}y + T^{\times} J (\cC_{Z, \sbm{ X & 
-Y}})^{*}y  \in \cM_{\mathfrak D}^{[\perp]}
$$
requires that, for all $y \in {\mathbb C}^{n_{Z}}$,
$$
\begin{bmatrix} (\cO_{V,W})^{*} & - (\cO_{U,W})^{*} \end{bmatrix} 
   T^{\times} \begin{bmatrix} (\cC_{Z,X})^{*} \\ (\cC_{Z,Y})^{*} 
    \end{bmatrix} y = - \Gamma^{*}y.
$$
Cancelling off the vector $y$ and rewriting as an operator equation 
then gives:
\begin{align*}
 & \begin{bmatrix} (\cO_{V,W})^{*} & - (\cO_{U,W})^{*} \end{bmatrix} 
   T^{\times} \begin{bmatrix} (\cC_{Z,X})^{*} \\ (\cC_{Z,Y})^{*} 
    \end{bmatrix} \\
 & \quad \quad   =
\begin{bmatrix} (\cO_{V,W})^{*} & (\cO_{U,W})^{*} \end{bmatrix} J 
    T^{\times} J  \begin{bmatrix} (\cC_{Z,X})^{*} \\ (-\cC_{Z,Y})^{*} 
    \end{bmatrix} \\
  & \quad \quad =  (\cO_{\sbm{V \\ U}, W})^{*} J T^{\times} J (\cC_{Z, \sbm{ X & 
    -Y}})^{*} = -\Gamma^{*}.
\end{align*}
Taking adjoints of both sides of the identity \eqref{TGamma} 
satisfied by $T$, we see that 
$$
  (\cO_{\sbm{V \\ U}, W} )^{*} T^{*} ( \cC_{Z, \sbm{ X & -Y}})^{*} = 
  -\Gamma^{*}.
$$
Since $(\cO_{\sbm{V \\ U}, W} )^{*}$ is injective on 
the range space of $T^{\times}$ or $JT^{*}J$ and $(\cC_{Z, \sbm{ X & 
-Y}})^{*}$ maps onto the domain space of $T^{\times}$ or $T^{*}$, it 
follows that $T^{\times} = J T^{*} J = T^{[*]}$.
The remaining points in statement (2) of the Lemma follow in much the 
same way as the corresponding points in statement (1).
\end{proof}
\begin{lemma}  \label{L:MperpK}
{\rm (1)} With $\cK$ as in \eqref{defcK}, the subspace \eqref{cPdata} decomposes as
\begin{equation}   \label{MperpK}
    \cM_{\mathfrak D}^{[\perp \cK]} =  {\mathbb G}_{T^{[*]}} 
    [+] (\cM_{\mathfrak D}^{[\perp]})_{1} [+] \sbm{ 
    \operatorname{Ker} (\cO_{V,W})^{*} \\ 0 }.
\end{equation}
In particular, $\cM_{\mathfrak D}^{[\perp \cK]}$ is $J$-positive if and only 
if its subspace 
$$
    (\cM_{\mathfrak D}^{[\perp \cK]})_{0}:=  {\mathbb G}_{T^{[*]}} 
    [+] (\cM_{\mathfrak D}^{[\perp]})_{1}
$$
is $J$-positive.

\smallskip

{\rm (2)} Dually, define a space $\cK' \subset L^{2}_{p+m}(i{\mathbb R})$ 
by
\begin{equation}  \label{defcK'}
  \cK' = \sbm{ H^{2}_{p}(\Pi_{-}) \oplus \operatorname{Ran} 
  (\cC_{Z,X})^{*} \\ L^{2}_{m}(i {\mathbb R}) }.
\end{equation}
Then  $\cM_{\mathfrak D}^{[ \perp]} \subset \cK'$ and the space
$$
(\cM_{\mathfrak D}^{[ \perp]})^{[ \perp \cK']}: = \cK' [-]_{J} \cM_{\mathfrak D}^{[ \perp]} 
= \cK' \cap \cM_{\mathfrak D}
$$
is given by
$$
    (\cM^{[\perp]})^{[ \perp \cK']} = \widehat{\mathbb G}_{T} [+] 
\cM_{{\mathfrak D},1} [+] \sbm{ 0 \\ \operatorname{Ker} \cC_{Z,Y} }.
$$
 In particular,  $(\cM_{\mathfrak D}^{[\perp]})^{[ \perp \cK']}$ is 
 $J$-negative if and only if its subspace
$$ 
( (\cM_{\mathfrak D}^{[\perp]})^{[ \perp \cK']})_{0}: = {\mathbb G}_{T} [+] 
\cM_{{\mathfrak D}, 1}
$$
is $J$-negative.

 \end{lemma}
 
\begin{proof}
    By definition, $\cM_{\mathfrak D}^{[ \perp \cK]} = \cK \cap 
    \cM_{\mathfrak D}^{[ \perp ]}$, where $\cM_{\mathfrak D}^{[ \perp
    ]}$ is given by \eqref{Mperpdecom} and where, due to \eqref{defcK} and 
\eqref{cMrep}, $\cK =\sbm{ L^{2}(i {\mathbb R})     \\ \operatorname{Ran} \cO_{U,W} 
\oplus H^{2}_{m}(\Pi_{+})}$. Note that 
    $$
\cG_{T^{[*]}} \subset \cK, \quad (\cM_{\mathfrak D}^{[ 
\perp]})_{1} \subset H^{2}_{p+m}(\Pi_{+}) \subset \cK,
$$
while
$$
    (\cM_{\mathfrak D}^{[ \perp]})_{2} \cap \cK =
 \sbm{ \operatorname{Ker} (\cO_{V,W})^{*} \\ \operatorname{Ker} 
 (\cO_{U,W})^{*}} \cap \sbm{ L^{2}_{p}(i{\mathbb R}) \\ 
 \operatorname{Ran} \cO_{U,W} \oplus H^{2}_{m}(\Pi_{+}) } = 
\sbm{ \operatorname{ Ker } (\cO_{V,W})^{*} \\ 0 }.
$$
Putting the pieces together leads to the decomposition \eqref{MperpK}.
Since the $J$-orthogonal summand  $\sbm{ \operatorname{Ker} 
(\cO_{V,W}))^{*} \\ 0 }$ is clearly $J$-positive, it follows 
that $\cM_{\mathfrak D}^{[\perp \cK]}$ is $J$-positive if and only if 
$\widehat {\mathbb G}_{T^{[*]}} [+] (\cM_{\mathfrak D}^{[ \perp 
]})_{1}$ is $J$-positive.
Statement (2) follows in a similar way.
\end{proof}

\begin{lemma}  \label{L:posnegsubspaces}
{\rm (1)} The subspace ${\mathbb G}_{T^{[*]}}$ is $J$-positive if and only if 
$I + T T^{[*]}$ is $J$-positive on the subspace
$P_{H^{2}_{p+m}(\Pi_{+})}\cM_{\mathfrak D}^{[\perp]} = 
\operatorname{Ran} J (\cC_{Z, \sbm{ X & -Y }})^{*}$.

{\rm (2)} The subspace $(\cM_{\mathfrak D}^{[ \perp ]})_{1}$ is 
$J$-positive if and only if 
the subspace \\ $P_{H^{2}_{p+m}(\Pi_-)} \cM_{\mathfrak D}  =  
\operatorname{Ran} \cO_{\sbm{ V \\ U}, W}$
 is $J$-negative.

{\rm (3)} The subspace $\widehat {\mathbb G}_{T}$ is  $J$-negative  if and 
only if $I + T^{[*]} T$ is a $J$-negative operator on the subspace 
$P_{H^{2}_{p+m}(\Pi_-)} \cM_{\mathfrak D} = \operatorname{Ran}  
\cO_{\sbm{ V \\ U},W}$.

\smallskip

{\rm (4)} The subspace $\cM_{{\mathfrak D}, 1}$ is  $J$-negative  if and only 
if the subspace \\
$P_{H^{2}_{p+m}(\Pi_+)} \cM_{\mathfrak D}^{[ \perp ]}  = 
\operatorname{Ran} J (\cC_{Z, \sbm{ X & -Y}})^{*}$
is  $J$-positive.
\end{lemma}
     
\begin{proof}
To prove (1), note that ${\mathbb G}_{T^{[*]}}$ being a $J$-positive 
subspace means that
$$
 \left\langle \begin{bmatrix} I \\ T^{[*]} \end{bmatrix} x, \, 
 \begin{bmatrix} I \\ T^{[*]} \end{bmatrix} x\right\rangle_{J \oplus J} =
     \langle (I + T T^{[*]}) x, x \rangle_{J} \ge 0
$$
for all $x \in \operatorname{Ran} J (\cC_{Z, \sbm{ X & -Y}})^{*}$, i.e., that
$I + T T^{[*]}$ is a $J$-positive operator.

\smallskip

To prove (2), use \eqref{cMDperp1} to see that elements $\bg$ of 
$(\cM_{\mathfrak D}^{[\perp]})_{1}$ have the form 
$$
\bg =\sbm{ \cO_{V,W} \\ 
\cO_{U,W} (\cG_{U,W})^{-1} \cG_{V,W} } x\quad\mbox{for some} \quad x \in {\mathbb 
C}^{n_{W}}.
$$  
The associated $J$-gramian is then given by
\begin{align*}
 &   \begin{bmatrix} (\cO_{V,W})^{*} & \cG_{V,W} (\cG_{U,W})^{-1} 
	(\cO_{U,W})^{*} \end{bmatrix} \begin{bmatrix} I_{p} & 0 \\ 0 
	& -I_{m} \end{bmatrix} \begin{bmatrix} \cO_{V,W} \\ \cO_{U,W} 
	(\cG_{U,W})^{-1} \cG_{V,W} \end{bmatrix} \\
& = \cG_{V,W} - \cG_{V,W} (\cG_{U,W})^{-1} \cG_{V,W}.
\end{align*}
By a Schur-complement analysis, this defines a negative semidefinite operator 
(in fact by our Nondegeneracy Assumption, a negative definite operator) if and only if
\begin{align*}
    \begin{bmatrix} \cG_{V,W} & \cG_{V,W} \\ \cG_{V,W} & \cG_{U,W} 
    \end{bmatrix}  = \begin{bmatrix} \cG_{V,W}^{\frac{1}{2}} & 0 \\ 0 
    & I \end{bmatrix}\begin{bmatrix} I_{\operatorname{Ran} \cG_{V,W}} 
    & \cG_{V,W}^{\frac{1}{2}} \\ \cG_{V,W}^{\frac{1}{2}} & \cG_{U,W} 
    \end{bmatrix} \begin{bmatrix} \cG_{V,W}^{\frac{1}{2}} & 0 \\ 0 & 
    I \end{bmatrix} \prec 0,
\end{align*}
which in turn happens if and only if
$$ 
\begin{bmatrix} I_{\operatorname{Ran} \cG_{V,W}} 
    & \cG_{V,W}^{\frac{1}{2}} \\ \cG_{V,W}^{\frac{1}{2}} & \cG_{U,W} 
    \end{bmatrix} \prec 0.
$$
Yet another Schur-complement analysis converts this to the condition 
$$
\cG_{U,W} - \cG_{V,W} \prec 0
$$ 
which is equivalent to 
$\operatorname{Ran} \cO_{\sbm{V \\ U}, W}$ being a $J$-negative subspace.

The proofs of statements (3) and (4) are parallel to those  of (1) 
and (2) respectively.
\end{proof} 

\begin{lemma}  \label{L:GammaD-factored} The Pick matrix $\bGamma_{\mathfrak D}$ 
    \eqref{GammaData} can be factored as follows:
\begin{equation}  \label{GammaData-factored}	
\bGamma_{\mathfrak D}
 = \begin{bmatrix} -\cC_{Z, \sbm{ X & -Y}} & 0 \\ 0 & 
(\cO_{\sbm{V \\ U},W})^{*} J \end{bmatrix}
\begin{bmatrix} I & T \\ T^{[*]} & -I \end{bmatrix}
    \begin{bmatrix} -J (\cC_{Z, \sbm{ X & -Y}})^{*} & 0 \\ 0 & 
	\cO_{\sbm{ V \\ U}, W} \end{bmatrix}.
\end{equation}
\end{lemma}

\begin{proof}  Multiplying out the expression on the right-hand 
    side in \eqref{GammaData-factored}, we get
$$
\begin{bmatrix}  \cG^{J}_{Z, \sbm{ X & -Y}} & - \cC_{Z, \sbm{ X & -Y}} 
    T \cO_{\sbm{V \\ U}, W} \\
-( \cO_{\sbm{V \\ U}, W})^{*} J T^{[*]} J (\cC_{Z, \sbm{X & -Y}})^{*} 
 & - \cG^{J}_{\sbm{V \\ U}, W} \end{bmatrix},
$$
which is exactly $\sbm{\Gamma_{L} & \Gamma \\ \Gamma^{*} & \Gamma_{R}}=:
\bGamma_{\mathfrak D}$
as we can see from the identities \eqref{GammaLRid} and \eqref{TGamma}.
\end{proof}

\begin{lemma} \label{L:GammaDpos}
 The following conditions are equivalent:
    
\begin{enumerate}
    \item The matrix $\bGamma_{\mathfrak D}$ \eqref{GammaData} is 
    positive.
    
     \item The subspace $P_{H^{2}_{p+m}(\Pi_-)} \cM_{\mathfrak D} = 
    \operatorname{Ran} \cO_{\sbm{V \\ U}, W}$ is $J$-negative and the 
    subspace ${\mathbb G}_{T^{[*]}}$ is $J$-positive.
    
    \item The subspace $P_{H^{2}_{p+m}(\Pi_+)} \cM_{\mathfrak D}^{[ \perp]} 
    = \operatorname{Ran} J (\cC_{Z, \sbm{X & -Y}})^{*}$ is 
    $J$-positive and the subspace $\widehat {\mathbb G}_{T}$ is 
    $J$-negative.

\end{enumerate}
\end{lemma}

\begin{proof}
    From the factorization \eqref{GammaData-factored} we see that 
    $\bGamma_{\mathfrak D} \succ 0$ if and only if the Hermitian form 
    on the subspace $\sbm{\operatorname{Ran} J (\cC_{Z, \sbm{ X & 
    -Y}})^{*} \\ \operatorname{Ran} \cO_{\sbm{V \\ U}, W} }$
    induced by the operator $\sbm{I & T \\ T^{[*]} & - I}$ in the $J 
    \oplus J$-inner product is positive. On the one hand we may consider the factorization
$$
 \begin{bmatrix} I & T \\ T^{[*]} & -I \end{bmatrix} =
     \begin{bmatrix} I & 0 \\ T^{[*]} & I \end{bmatrix} 
	 \begin{bmatrix} I & 0 \\ 0 & -I - T^{[*]} T \end{bmatrix}
\begin{bmatrix} I & T \\ 0 & I \end{bmatrix}
$$
to deduce that $\sbm{ I & T \\ T^{[*]} & -I  }$ is $(J \oplus 
J)$-positive if and only if 
\begin{enumerate}
\item[(i)] the identity operator $I$ is $J$-positive on 
$\operatorname{Ran} J (\cC_{Z, \sbm{X & -Y}})^{*}$ (i.e., the subspace
$\operatorname{Ran} J (\cC_{Z, \sbm{X & -Y}})^{*}$ is $J$-positive), and 
\item[(ii)]  $-I - T^{[*]}T$ is a $J$-positive operator on 
$\operatorname{Ran} \cO_{\sbm{V \\ U}, W}$, i.e., $\widehat{\mathbb 
G}_{T}$ is a $J$-negative subspace.  
\end{enumerate}
Note that this analysis 
amounts to taking the $J$-symmetrized Schur complement of the matrix 
$\sbm{ I & T \\ T^{[*]} & -I}$ with respect to the (1,1)-entry. This 
establishes the equivalence of (1) and (3).

\smallskip

On the other hand we may take the $J$-symmetrized Schur complement of $\sbm{ I & T \\ T^{[*]} & -I}$
with respect to the (2,2)-entry, corresponding to the factorization
$$
 \begin{bmatrix} I & T \\ T^{[*]} & -I \end{bmatrix} = 
     \begin{bmatrix} I & -T \\ 0 & I \end{bmatrix}
\begin{bmatrix} I + T T^{[*]} & 0 \\ 0 & -I \end{bmatrix}
    \begin{bmatrix} I & 0 \\ -T^{[*]} & I \end{bmatrix}.
$$
In this way we see that $(J \oplus J)$-positivity of $\sbm{ I & T \\ T^{[*]} & -I}$
corresponds to 
\begin{enumerate}
\item[(i$^{\prime}$)] $I + T T^{[*]}$ is a $J$-positive operator (i.e., 
the subspace ${\mathbb G}_{T^{[*]}}$ is $J$-positive), and 
\item[(ii$^{\prime}$)] minus the identity operator $-I$ is $J$ positive on $\operatorname{Ran}  
\cO_{\sbm{ V \\ U}, W}$ (i.e., the subspace is $\operatorname{Ran}  
\cO_{\sbm{ V \\ U}, W}$ is $J$-negative).  
\end{enumerate}
This establishes the  equivalence of (1) and (2).
\end{proof}

To conclude the proof of Corollary \ref{C:BTOA-NP} for the general 
BiTangential case (at least with the Nondegeneracy Assumption in 
place), it remains only to assemble the various pieces. By Lemma 
\ref{L:MperpK} part (1), we see that $\cM_{\mathfrak D}^{[\perp \cK]}$ 
being $J$-positive is equivalent to 
\begin{equation}  \label{condition1}
 {\mathbb G}_{T^{[*]}} \text{ and } (\cM_{\mathfrak D}^{[\perp]})_{1} 
\text{ are $J$-positive subspaces.}
\end{equation}
By Lemma \ref{L:posnegsubspaces}, we see that $ (\cM_{\mathfrak 
D}^{[\perp]})_{1} $ being $J$-positive is equivalent to $\operatorname{Ran} \cO_{\sbm{ V \\ U}, W}$ 
being $J$-negative.  We therefore may amend \eqref{condition1} to
\begin{equation}   \label{condition2}
{\mathbb G}_{T^{[*]}} \text{ is $J$-positive  and }  
P_{H^{2 \perp}_{p+m}(\Pi_+)} \cM_{\mathfrak D}  = \operatorname{Ran} \cO_{\sbm{ V \\ U}, W}
\text{ is $J$-negative}
\end{equation}
which is exactly statement (2) in Lemma \ref{L:GammaDpos}.  Thus (1) 
$\Leftrightarrow$ (2) in Corollary \ref{C:BTOA-NP} follows from (1) 
$\Leftrightarrow$ (2) in Lemma \ref{L:GammaDpos}.

\smallskip

For the general BTOA case, the reproducing kernel space 
$\cH(K_{\Theta,J})$ again can be identified with a range space, namely
\begin{equation}   \label{HKThetaJid}
 \cH(K_{\Theta,J}) = \operatorname{Ran} (P^{J}_{H^{2}_{p+m}(\Pi_{+})} - 
 P^{J}_{\cM_{\mathfrak D}})
\end{equation}
with lifted indefinite inner product, 
where $ P^{J}_{H^{2}_{p+m}(\Pi_{+})}$ and $P^{J}_{\cM_{\mathfrak D}}$ are the 
$J$-orthogonal projections of $L^{2}_{p+m}(i{\mathbb R})$ onto $H^{2}_{p+m}(\Pi_{+})$ and $\cM_{\mathfrak D}$
respectively (see \cite[Theorem 3.3]{BHMesa}). Due to $J$-orthogonal decompositions
\begin{align*}
 &   H^{2}_{p+m}(\Pi_{+}) 
= \operatorname{Ran} J (\cC_{Z, \sbm{ X & -Y}})^{*} \, [+] 
 \cM_{{\mathfrak D}, 1} [+] \, \cM_{{\mathfrak D}, 2}, \\
& \cM = \widehat{\mathbb G}_{T} \, [+]  \cM_{{\mathfrak D}, 1} [+] \, 
\cM_{{\mathfrak D}, 2},
 \end{align*}
 we can simplify the difference of $J$-orthogonal projections to
$$
 P^{J}_{H^{2}_{p+m}} -  P^{J}_{\cM_{\mathfrak D}}
 = P^{J}_{\operatorname{Ran} J (\cC_{Z, \sbm{ X &-Y}})^{*}} - 
 P^{J}_{\widehat {\mathbb G}_{T}}.
 $$
 By a calculation as in the proof for Case 1, one can show that 
 \begin{equation}   \label{HKThetaJ}
 \cH(K_{\Theta,J}) = ( \operatorname{Ran} J (\cC_{Z, \sbm{X & 
 -Y}})^{*})_{J} \, [+] (\widehat {\mathbb G}_{T})_{-J}
 \end{equation}
 with the identity map a Kre\u{\i}n-space isomorphism, 
 where the subscripts on the right hand side indicating that one should use the $J$-inner 
 product for the first component but the $-J$-inner product for the 
 second component. We conclude that $\cH(K_{\Theta, J})$ is a HIlbert 
 space exactly when condition (3) in Lemma \ref{L:GammaDpos} holds. We now see that 
 (1) $\Leftrightarrow$ (3) in Corollary \ref{C:BTOA-NP} is an immediate consequence of (1) 
 $\Leftrightarrow$ (3) in Lemma \ref{L:GammaDpos}.
 \end{proof}
 
The above analysis actually establishes a bit more which we collect 
in the following Corollary.

\begin{corollary}  \label{C:BTOA-NP'}  The following conditions are 
    equivalent:
\begin{enumerate}
    \item The subspace $\cM_{\mathfrak D}^{[\perp \cK]}$ is $J$-positive.
    
    \item The subspace $(\cM_{\mathfrak D}^{[ \perp ]})^{[\perp \cK']}$ is 
    $J$-negative.
 \end{enumerate}
\end{corollary}

\begin{proof}
    We have seen in Lemma \ref{L:MperpK} part (2) that 
    $(\cM_{\mathfrak D}^{[\perp]})^{[ \perp \cK']}$ being $J$-negative is 
    equivalent to
\begin{equation}   \label{condition1'}
   \widehat{\mathfrak G}_{T} \text{ and } (\cM_{\mathfrak D}^{[\perp]})_{1} \text{ 
   are $J$-negative subspaces.}
\end{equation}
Lemma \ref{L:posnegsubspaces} (4) tells us that $ \cM_{{\mathfrak D},1}$ 
being $J$-negative is equivalent to \\
$\operatorname{Ran} J (\cC_{Z, \sbm{ X & -Y}})^{*}$  
being $J$-positive.  Thus condition \eqref{condition1'} can be amended to
\begin{equation}   \label{condition2'}
 \widehat{\mathbb G}_{T} \text{ is negative and }
 P_{H^2_{p+m}(\Pi_+)} \cM_{\mathfrak D}^{[ \perp]} = 
\operatorname{Ran} J (\cC_{Z, \sbm{ X & -Y}})^{*} \text{ is $J$-positive.}
\end{equation}
We next use the equivalence of (1) $\Leftrightarrow$ (3) in Theorem 
\ref{L:GammaDpos} to see that condition \eqref{condition2'} is also 
equivalent to $\bGamma_{\mathfrak D} \succ 0$.  We then use the 
equivalence (1) $\Leftrightarrow$ (2) in Theorem \ref{L:GammaDpos} 
to see that this last condition in turn is equivalent 
to $\cM_{\mathfrak D}^{[ \perp \cK]}$ being $J$-positive.
\end{proof}

\section{Interpolation problems in the generalized Schur class}  \label{S:negsquares}
Much of the previous analysis extends from the Schur class  $\cS^{p \times m}(\Pi_{+})$ 
to a larger class  $\cS_\kappa^{p \times m}(\Pi_{+})$ (generalized Schur class)
consisting of $\mathbb C^{p\times m}$-valued
functions that are meromorphic on $\Pi_+$ with total pole multiplicity equal $\kappa$ 
and such that their $L^\infty$ norm (that is, $\sup_{y\in\mathbb R}\|S(iy)\|$) does 
not exceed one. The values 
$S(iy)$ are understood in the sense of non-tangential boundary limits that exist for 
almost all 
$y\in\mathbb R$. The multiplicity of a pole $z_0$ for a matrix-valued function $S$ is 
defined as the sum of absolute values of all negative partial multiplicities 
appearing in the Smith form of $S$
at $z_0$ (see e.g.\ \cite[Theorem 3.1.1]{BGR}). Then the total pole multiplicity of $S$ 
is defined as the sum of multiplicities of all poles.  Let us 
introduce the notation 
$$
m_{P}(S) = \text{ sum of all pole multiplicities of $S$ over all poles in } \Pi_{+}.
$$
It follows by the maximum modulus principle that  $\cS^{p \times m}_0(\Pi_{+})$ 
is just the classical Schur class. Generalized Schur functions appeared first 
in \cite{T} in the interpolation context and 
were comprehensively studied by Kre\u{\i}n and Langer in \cite{kl,kl1}. Later
work on the classes $\cS_{\kappa}$ include \cite{dls1}, \cite{J},
\cite{CG1}, and \cite{ADRS}, as well as \cite{Gol}, \cite{nud}, \cite{BH}, 
\cite{BH86} and the book \cite{BGR} in the context of interpolation.

\smallskip

The class $\cS_{\kappa}(\Pi_{+})$ can alternatively be characterized 
by any of the following conditions:
\begin{enumerate}
    \item  ${\rm sq}_{-}(K_{S}) = \kappa$ where the kernel $K_{S}$ is 
    given by \eqref{dBRkerPi+}.
    
    \item ${\rm sq}_{-}({\mathbf K}_{S}) = \kappa$, where ${\mathbf 
    K}_{S}$ is the $2 \times 2$-block matrix kernel \eqref{ker22}.
    
    \item $S$ admits left and right 
(coprime) {\em Kre\u{\i}n-Langer factorizations}
$$
F(\lam)=S_R(\lam)\vartheta_R(\lam)^{-1}= \vartheta_L(\lam)^{-1}S_L(\lam),
$$
where $S_L, \, S_R\in\cS^{p\times m}(\Pi_+)$ and $\vartheta_L$ and $\vartheta_R$ are matrix-valued finite
Blaschke products of degree $\kappa$ (see \cite{kl1} for the scalar-valued case and 
\cite{dls1} for the Hilbert-space operator-valued case).
By a $\mathbb C^{n\times n}$-valued finite Blaschke product we mean the product of $\kappa$ Blaschke 
(or Blaschke-Potapov) factors 
$$
I_n-P+\frac{\lam-\alpha}{\lam+\overline{\alpha}}P
$$
where $\alpha\in\Pi_+$ and $P$ is an orthogonal projection in $\mathbb C^n$. 
\end{enumerate}
There is also an intrinsic characterization of matrix triples 
$(C,A,B)$ which can arise as the pole triple over the unit disk for a 
generalized Schur class function---see \cite{bolleiba} for details.

\smallskip

Let us take another look at the  BiTangential Nevanlinna-Pick problem 
\eqref{inter1'}--\eqref{inter3'}.
If the Pick matrix \eqref{Pick-simple} is not positive semidefinite, 
the problem has no solutions in the Schur class $\cS^{p\times n}(\Pi_+)$, 
by Theorem \ref{T:BT-NP}. However, there always exist generalized Schur 
functions that are analytic at all interpolation nodes $z_i,w_j$ and satisfy 
interpolation conditions \eqref{inter1'}--\eqref{inter3'}. One can show that 
there exist such functions with only one pole of a sufficiently high multiplicity 
at any preassigned point in $\Pi_+$. The question of interest is
to find the smallest integer $\kappa$, for which interpolation conditions  
\eqref{inter1'}--\eqref{inter3'} 
are met for some function $S\in \cS_\kappa^{p \times m}(\Pi_{+})$ and then 
to describe the set of all such functions. 

\smallskip

The same question makes sense in the more general setting of the {\bf BTOA-NP} 
interpolation problem:
{\em given a $\Pi_{+}$-admissible BTOA interpolation data set \eqref{data},
find the  smallest integer $\kappa$, for which interpolation conditions  
\eqref{BTOAint1}--\eqref{BTOAint3} are satisfied for some function 
$S\in \cS_\kappa^{p \times m}(\Pi_{+})$ which is analytic on 
$\sigma(Z)\cup\sigma(W)$, and  describe the set of all such functions}. 

\smallskip

The next theorem gives the answer to the question above in the so-called nondegenerate case.

\begin{theorem}  \label{T:BTOA-NPkap}  Suppose that ${\mathfrak D}  = ( X,Y,Z; U,V, W; \Gamma)$
 is a $\Pi_{+}$-admissible BTOA interpolation data set and let us 
 assume that the BTOA-Pick matrix 
$\bGamma_{\mathfrak D}$ defined by \eqref{GammaData} is invertible. 
Let $\kappa$ be the  smallest 
integer for which there is a function $S\in \cS_\kappa^{p \times m}(\Pi_{+})$ which is analytic on 
$\sigma(W) \cup \sigma(Z)$ and satisfies the interpolation conditions \eqref{BTOAint1}--\eqref{BTOAint3}.
Then $\kappa$ is given by any one of the following three equivalent 
formulas:
\begin{enumerate}
    \item $\kappa=\nu_-(\bGamma_{\mathfrak D})$, the number of negative 
    eigenvalues of $\bGamma_{\mathfrak D}$.
    \item $\kappa = \nu_{-}(\cM_{\mathfrak D}^{[\perp] \cK})$, the negative signature 
   of the Kre\u{\i}n-space $\cM_{\mathfrak D}^{[\perp] \cK}$ 
    in the $J$-inner product.
    
    \item $\kappa =  \nu_{-}(\cH(K_{\Theta, J}))$,  the negative signature  of the 
reproducing kernel Pontryagin space $\cH(K_{\Theta,J})$, where $\Theta$ is defined 
as in \eqref{sep1} and $K_{\Theta,J}$ as in \eqref{ad1}.
\end{enumerate}
Furthermore, the function $S$ belongs to the generalized Schur class 
$\cS_\kappa^{p \times m}(\Pi_{+})$ and satisfies the interpolation conditions 
\eqref{BTOAint1}--\eqref{BTOAint3}
if and only if it is of the form
\begin{equation}   \label{LFTparamkap}
S(\lambda) = (\Theta_{11}(\lam) G(\lam) + \Theta_{12}(\lam))
( \Theta_{21}(\lam) G(\lam) + \Theta_{22}(\lam))^{-1}
\end{equation}
for a Schur class function  $G\in \cS^{p \times m}(\Pi_{+})$ such that
\begin{equation}   \label{LFTparampar}
\det (\psi(\lam)(\Theta_{21}(\lam) G(\lam) + \Theta_{22}(\lam)))\neq 0,\quad \lam\in \Pi_{+}
\backslash(\sigma(Z)\cup\sigma(W))
\end{equation}
where $\psi$ is the $m \times m$-matrix function defined in 
\eqref{sep1}.
\end{theorem}

\subsection{The state-space approach}  \label{S:statekap}
The direct proof of the necessity of condition (1) in Theorem 
\ref{T:BTOA-NPkap} for the existence of class-$\cS_{\kappa}^{p \times 
m}(\Pi_{+})$ solution of the interpolation  conditions 
\eqref{BTOAint1}--\eqref{BTOAint3} relies on the
characterization of the class $\cS_\kappa^{p \times m}(\Pi_{+})$ in 
terms of the kernel \eqref{ker22} mentioned above: {\em a 
$\mathbb C^{p\times n}$-valued function meromorphic on $\Pi_+$ belongs to $\cS_\kappa^{p \times m}(\Pi_{+})$
if and only if the kernel ${\mathbf K}_{S}(\lam, \lam_{*}; \zeta, \zeta_{*})$ defined  as in \eqref{ker22}
has $\kappa$ negative squares on $\Omega_S^4$}:
\begin{equation}
{\rm sq}_-{\mathbf K}_{S}=\kappa,
\label{aug3}
\end{equation}
where $\Omega_S\subset\Pi_+$ is the domain of analyticity of $S$. The latter equality means that the 
block matrix $\left[{\mathbf K}_{S}(z_i, z_i; z_j, z_j)\right]_{i,j=1}^N$ has at most $\kappa$ 
negative eigenvalues for any 
choice of finitely many points $z_1.\ldots,z_N\in\Omega_S$, and it has exactly $\kappa$ negative eigenvalues 
for at least one such choice.

\smallskip
Now suppose that $S \in \cS_{\kappa}^{p \times m}(\Pi_{+})$ satisfies 
the interpolation conditions \eqref{BTOAint1}--\eqref{BTOAint3}.
The kernel ${\mathbf K}_{S}$ satisfying condition \eqref{aug3} still admits the Kolmogorov decomposition
\eqref{KS2}, but this time the state space $\cX$ is a Pontryagin space of negative index $\kappa$.
All computations following formula \eqref{KS2} go through with $\Pi_+$ replaced by $\Omega_S$ showing 
that the matrix  $\bGamma'_{\mathfrak D}$ defined in \eqref{aug4} is equal to the Pick matrix 
$\bGamma_{\mathfrak D}$ given in \eqref{GammaData}. Note that the operations bringing the 
kernel ${\mathbf K}_{S}$ to the matrix $\bGamma_{\mathfrak D}$ amount 
to a sophisticated conjugation of the kernel ${\mathbf K}_{S}$.  We 
conclude that  $\nu_-(\bGamma_{\mathfrak D})=\nu_-(\bGamma'_{\mathfrak D}) \le \kappa$.
Once one of the sufficiency arguments has been carried out (by whatever 
method) to show that $\nu_-(\bGamma_{\mathfrak 
D})=\nu_-(\bGamma'_{\mathfrak D}) < \kappa$ implies that there is a 
a function $S$ in a generalized Schur class $\cS^{p \times 
m}_{\kappa'}(\Pi_{+})$ with $\kappa' < \kappa$ satisfying the 
interpolation conditions, then $\nu_{-}(\bGamma_{\mathfrak D}) < 
\kappa$ leads to a contradiction to 
the minimality property of $\kappa$.   We conclude 
that $\nu_{-}(\bGamma_{\mathfrak D}) = \kappa$ is necessary for 
$\kappa$ to be the smallest integer so that there is a solution $S$ 
of class $\cS_{\kappa}^{p \times m}(\Pi_{+})$ of the interpolation 
conditions \eqref{BTOAint1}--\eqref{BTOAint3}. 

\smallskip

We now suppose that $\nu_-(\bGamma_{\mathfrak D})=\kappa$. The identity \eqref{ad1} relies on 
equality \eqref{bigLySyl} and on the assumption that $\bGamma_{\mathfrak D}$ 
is invertible. In particular, the matrix $\Theta(\lam)$ still is $J$-unitary for each 
$\lam \in i \mathbb R$, i.e., equalities \eqref{ThetaJunitary} hold for all $\lam\in i\mathbb R$.
By using the 
controllability/observability assumptions on $(Z,X)$ and $(U,W)$, it 
follows from the formula on the right hand side of \eqref{ad1} that the kernel 
$K_{\Theta,J}$ \eqref{ad1} has $\kappa$ negative squares on
$\Omega_\Theta = \Pi_{+} \setminus \sigma(W)$ (the points of 
analyticity for $\Theta$ in the right half plane $\Pi_{+}$):
$$
 {\rm sq}_{-} K_{\Theta,J} = \kappa.
$$

\smallskip

We shall have need of the {\em Potapov-Ginsburg transform} $U = \sbm{ 
U_{11} & U_{12} \\ U_{21} & U_{22}}$ of 
a given block $2 \times 2$-block matrix function $\Theta = \sbm{ 
\Theta_{11} & \Theta_{21} \\ \Theta_{21} & \theta_{22}} $(called  
the {\em Redheffer transform} in \cite{BGR}) defined by
$$
  U = \begin{bmatrix} U_{11} & U_{12} \\ U_{21} & 
  U_{22} \end{bmatrix} : = \begin{bmatrix} \Theta_{12} 
  \Theta_{22}^{-1} & \Theta_{11} - \Theta_{12} \Theta_{22}^{-1} 
  \Theta_{21} \\ \Theta_{22}^{-1} & - \Theta_{22}^{-1} \Theta_{21} 
  \end{bmatrix}.
$$
 This transform is the result of rearranging the inputs and outputs in the system of 
 equations
\begin{equation}  \label{chain}
\begin{bmatrix} \Theta_{11} & \Theta_{12} \\  \Theta_{21} & \Theta_{22} \end{bmatrix} 
\begin{bmatrix} x_{2} \\ y_{2} \end{bmatrix} = \begin{bmatrix} y_{1} \\ x_{1} 
  \end{bmatrix}
\end{equation}
to have the form
\begin{equation}  \label{scattering}
  \begin{bmatrix} U_{11} & U_{12} \\ U_{21} & U_{22} \end{bmatrix} 
      \begin{bmatrix} x_{1} \\ x_{2} \end{bmatrix} = \begin{bmatrix} 
	  y_{1} \\ y_{2} \end{bmatrix},
\end{equation}
and in circuit theory has the interpretation as the change of 
variable from the {\em chain formalism} \eqref{chain} to the {\em scattering 
formalism} \eqref{scattering}.  Based on this connection it is not hard to show that
$$
 {\rm sq}_{-} K_{U} = {\rm sq}_{-} K_{\Theta,J} = \kappa
$$
where the notation $K_{U}$  is as in \eqref{dBRkerPi+} and 
$K_{\Theta,J}$ as in \eqref{ad1}
(see \cite[Theorem 13.1.3]{BGR}).
We conclude that $U$ is in the generalized Schur class $\cS_{\kappa}^{(p+m) \times (m+p)}(\Pi_{+})$.  
By the Kre\u{\i}n-Langer factorization result for the generalized Schur 
class (see \cite{kl1}), it follows that $\kappa$ is also equal to the 
total pole multiplicity of $U$ over points in $\Pi_{+}$:
$$
m_{P}(U) = \kappa.
$$
\smallskip

We would like to show next that
\begin{equation}   \label{toshow1}
    m_{P}(U_{22}) = m_{P}(\Theta_{22}^{-1} \Theta_{21}) = \kappa.
\end{equation}
Verification of this formula will take several steps and 
follow the analysis in \cite[Chapter 13]{BGR}.
We first note that the calculations \eqref{cMrep}--\eqref{cMBLrep} 
go through unchanged so we still have the Beurling-Lax representation
\begin{equation}  \label{BLrep}
  \cM_{\mathfrak D} = \Theta \cdot H^{2}_{p+m}(\Pi_{+})
\end{equation}
where $\cM_{\mathfrak D}$ also has the representation \eqref{cMrep}.
The observability assumption on the output pair $(U,W)$ translates to 
an additional structural property on $\cM_{\mathfrak D}$:
\begin{itemize}
    \item $(U,W)$ {\em observable implies}
\begin{equation}  \label{obs-imp1}
    \cM_{\mathfrak D} \cap \begin{bmatrix} L^{2}(i {\mathbb R}) \\ 0 
    \end{bmatrix} = \cM_{\mathfrak D} \cap \begin{bmatrix} 
    H^{2}_{p}(\Pi_{+}) \\ 0 \end{bmatrix}.
\end{equation}
\end{itemize}
Making use of \eqref{BLrep}, condition \eqref{obs-imp1} translates to 
an explicit property of $\Theta$, namely:
$$
f \in H^{2}_{p}(\Pi_{+}),\, g \in H^{2}_{m}(\Pi_{+}),\, \Theta_{21}f 
+ \Theta_{22}g = 0 \Rightarrow \Theta_{11} f + \Theta_{12}g \in 
H^{2}_{p}(\Pi_{+}).
$$
Solving the first equation for $g$ gives $g = -\Theta_{22}^{-1} 
\Theta_{21} f$ and this last condition can be rewritten as
$$
f \in H^{2}_{p}(\Pi_{+}), \, \Theta_{22}^{-1} 
\Theta_{21} f \in H^{2}_{m}(\Pi_{+}) \Rightarrow
(\Theta_{11} - \Theta_{12} \Theta_{22}^{-1} \Theta_{21}) f \in 
H^{2}_{p}(\Pi_{+}),
$$
or, more succinctly,
$$
 f \in H^{2}_{p}(\Pi_{+}),\, U_{22} f \in H^{2}_{m}(\Pi_{+}) 
 \Rightarrow U_{12} f \in H^{2}_{m}(\Pi_{+}).
$$
This last condition translates to
\begin{equation}  \label{obs-imp2}
    m_{P}(U_{22}) = m_{P}\left( \sbm{ U_{12} \\ U_{22} } \right).
\end{equation}

Similarly, the controllability assumption on the input pair $(Z,X)$ 
translates to an additional structural property on $\cM_{\mathfrak 
D}$, namely:
\begin{itemize}
    \item $(Z,X)$ {\em controllable implies}
\begin{equation}   \label{control-imp1}
    P_{\sbm{ 0 \\ H^{2}_{m}(\Pi_{+})}} \left( \cM_{\mathfrak D} \cap 
    H^{2}_{p+m}(\Pi_{+}) \right) = \begin{bmatrix} 0 \\ H^{2}_{m}(\Pi_{+}) 
    \end{bmatrix}.
\end{equation}
\end{itemize}
In terms of $\Theta$, from the representation \eqref{BLrep} we see 
that this means that, given any $h \in H^{2}_{m}(\Pi_{+})$, we can 
find $f \in H^{2}_{p}(\Pi_{+})$ and $g \in H^{2}_{m}(\Pi_{+})$ so that
$$
\Theta_{11} f + \Theta_{12} g \in H^{2}_{p}(\Pi_{+}), \quad
\Theta_{21} f + \Theta_{22} g = h.
$$
We can solve the second equation for $g$
$$
  g = \Theta_{22}^{-1} h - \Theta_{22}^{-1} \Theta_{21} f \in 
  H^{2}_{m}(\Pi_{+})
$$
and rewrite the first expression in terms of  $f$ and $h$:
$$
 (\Theta_{11} - \Theta_{12} \Theta_{22}^{-1} \Theta_{21}) f + 
 \Theta_{12} \Theta_{22}^{-1} h \in H^{2}_{p}(\Pi_{+}).
$$
Putting the pieces together, we see that an equivalent form of 
condition \eqref{control-imp1} is: {\em for any $ h \in H^{2}_{m}(\Pi_{+})$, there exists 
an $f \in H^{2}_{p}(\Pi_{+})$ such that
$$
\Theta_{22}^{-1} h - \Theta_{22}^{-1} \Theta_{21} f \in 
H^{2}_{m}(\Pi_{+}), \quad  (\Theta_{11} -  \Theta_{12} \Theta_{22}^{-1} 
\Theta_{21}) f + \Theta_{12} \Theta_{22}^{-1} h \in 
H^{2}_{p}(\Pi_{+}).
$$}
More succinctly,
\begin{align*}
& h \in H^{2}_{m}(\Pi_{+}) \Rightarrow \exists f \in 
 H^{2}_{p}(\Pi_{+}) \text{ so that } \\
& U_{21} h + U_{22} f \in H^{2}_{m}(\Pi_{+}), \quad
U_{12} f + U_{11} h \in H^{2}_{p}(\Pi_{+}),
\end{align*}
or, in column form, for each $h \in H^{2}_{m}(\Pi_{+})$ there exists 
$f \in H^{2}_{p}(\Pi_{+})$ so that
$$
 \begin{bmatrix} U_{11} \\ U_{21} \end{bmatrix} h + \begin{bmatrix} 
     U_{12} \\ U_{22} \end{bmatrix} f \in H^{2}_{p+m}(\Pi_{+}).
$$
The meaning of this last condition is:
\begin{equation}   \label{control-imp2}
    m_{P}(U) = m_{P}\left(\sbm{ U_{12} \\ U_{22}}  \right).
    \end{equation}
Combining \eqref{obs-imp2} with \eqref{control-imp2} gives us 
\eqref{toshow1} as wanted.

\smallskip
Since $\Theta$ is not $J$-contractive in $\Pi_+$ anymore,
we cannot conclude that $\Theta_{22}^{-1}\Theta_{21}$ is contraction valued. However, 
due to equalities \eqref{ThetaJunitary}, the function $\Theta_{22}(\lam)^{-1}\Theta_{21}(\lam)$
is a contraction for each $\lam\in i\mathbb R$.
Therefore, $\Theta_{22}^{-1}\Theta_{21}$ belongs to the generalized Schur class 
$\mathcal S^{m\times p}_\kappa(\Pi_+)$. We next wish to argue that
\begin{equation}  \label{toshow2}
{\rm wno} \det \Theta_{22}+{\rm wno} \det \psi=\kappa,
\end{equation}
where $\psi$ is given by \eqref{sep1}.
From the representation \eqref{BLrep} and the form of $\cM_{\mathfrak 
D}$ in \eqref{BLrep} we see that
$$
\Theta_{22} \begin{bmatrix} \Theta_{22}^{-1} \Theta_{21} & I_{m}
\end{bmatrix} H^{2}_{p+m}(\Pi_{+}) = \begin{bmatrix} \Theta_{21} & 
\Theta_{22} \end{bmatrix} H^{2}_{p+m}(\Pi_{+}) = \psi^{-1} 
H^{2}_{m}(\Pi_{+}).
$$
We rewrite this equality as
\begin{equation}   \label{subspace-eq}
  \begin{bmatrix} \Theta_{22}^{-1} \Theta_{21} & I_{m} \end{bmatrix} 
      H^{2}_{p+m}(\Pi_{+}) = \Theta_{22}^{-1} \psi^{-1} 
      H^{2}_{m}(\Pi_{+}).
\end{equation}
In particular,
$$
  \Theta_{22}^{-1} \psi^{-1} H^{2}_{m}(\Pi_{+}) \supset 
  H^{2}_{m}(\Pi_{+})
$$
so the matrix function $\Theta_{22}^{-1} \psi^{-1}$ has no zeros (in 
the sense of its Smith-McMillan form) in $\Pi_{+}$.  As 
$\Theta_{22}^{-1}$ and $\psi^{-1}$ are invertible on the boundary $i 
{\mathbb R}$, we see that ${\rm wno} \det(\Theta_{22}^{-1} 
\psi^{-1})$ is well-defined and by the Argument Principle we have
\begin{align}
& -{\rm wno} \det \Theta_{22} - {\rm wno} \det(\psi) 
= {\rm wno} \det(\Theta_{22}^{-1} \psi^{-1}) \notag \\
& \quad \quad \quad \quad = 
 m_{Z}(\det(\Theta_{22}^{-1} \psi^{-1})) - m_{P}(\det(\Theta_{22}^{-1} 
 \psi^{-1})) \notag \\
 & \quad \quad \quad \quad = -m_{P}(\det(\Theta_{22}^{-1} \psi^{-1})) = 
 -m_{P}(\Theta_{22}^{-1} \psi^{-1}) \notag \\
 &\quad \quad \quad \quad  = -\dim P_{H^{2}_{m}(\Pi_{-})} \Theta_{22}^{-1} \psi^{-1} 
 H^{2}_{m}(\Pi_{+})
 \label{know1}
\end{align}
where $m_{Z}(S)$ is the total zero multiplicity of the rational matrix 
function $S$ over all zeros in $\Pi_{+}$.
On the other hand we have
\begin{equation}  \label{know2}
    \dim P_{H^{2}_{m}(\Pi_{-})} \begin{bmatrix} \Theta_{22}^{-1} \Theta_{21} & I_{m} \end{bmatrix} 
      H^{2}_{p+m}(\Pi_{+}) = m_{P}(\Theta_{22}^{-1} \Theta_{21}) = 
      \kappa
\end{equation}
where we make use of \eqref{toshow1} for the last step.  Combining 
\eqref{know1} and \eqref{know2} with \eqref{subspace-eq} finally 
brings us to \eqref{toshow2}.

\smallskip

In addition to the Beurling-Lax representation \eqref{cMBLrep} or 
\eqref{BLrep}, we also still have the Beurling-=Lax representation 
\eqref{cM-rep} for $\cM_{{\mathfrak D},-}$ with $\psi, \psi^{-1}$ 
given by \eqref{sep1} and \eqref{psi-inv}.
However, the condition \eqref{sol-rep} should be modified as follows:
\begin{itemize}
    \item {\em A meromorphic function $S \colon \Pi_{+} \to {\mathbb C}^{p \times m}$ 
has total pole multiplicity at most $\kappa$ over $\Pi_+$ and 
    satisfies the interpolation conditions \eqref{BTOAint1}--\eqref{BTOAint3} if and only if there 
    is an $m\times m$-matrix
valued function $\Psi$ analytic on $\Pi_+$ with $\det \Psi$ having no zeros on $\sigma(Z)\cup\sigma(W)$ 
and with $\kappa$ zeros in $\Pi_+$ such that 
 \begin{equation}  \label{sol-rep'}
  \begin{bmatrix} S \\ I_m \end{bmatrix}\psi^{-1}\Psi H^2_m(\Pi_+) 
      \subset \cM_{\mathfrak D}.
 \end{equation}}
 \end{itemize}
Now instead of \eqref{Sparam}, we have
  \begin{equation}  \label{Sparam'}
 \begin{bmatrix} S \\ I_{m} \end{bmatrix} \psi^{-1} \Psi =
\begin{bmatrix} \Theta_{11} & \Theta_{12} \\
     \Theta_{21} & \Theta_{22} \end{bmatrix} \begin{bmatrix} Q_{1}
     \\ Q_{2} \end{bmatrix}
\end{equation}
for some $(p+m) \times m$ matrix function $\sbm{Q_{1} \\ Q_{2}} \in H^2_{(p+m) \times m}(\Pi_{+})$.
Then we conclude from the $J$-unitarity of $\Theta$ on $i\mathbb R$ (exactly as in Section 2) that for 
almost all 
$\lam\in i\mathbb R$, the matrix $Q_2(\lam)$ is invertible whereas the matrix 
$G(\lam)=Q_1(\lam)Q_2(\lam)^{-1}$
is a contraction. The identity \eqref{bottom} arising from looking at 
the bottom component of \eqref{Sparam'} must be modified to read
$$
    \psi^{-1} \Psi = \Theta_{21} Q_{1} + \Theta_{22} Q_{2} = 
    \Theta_{22} (\Theta_{22}^{-1} \Theta_{21} G + I_{m}) Q_{2}
$$
leading to the modification of \eqref{wno'}:
\begin{align*}   
   & {\rm wno} \det \psi^{-1} + {\rm wno} \det \Psi = \notag \\
    & \quad \quad \quad \quad 
{\rm wno} \det \Theta_{22} + {\rm wno} \det (\Theta_{22}^{-1} \Theta_{21} G + 
    I_{m}) + \rm{ wno} \det Q_{2}.
\end{align*}
The identity \eqref{wno''} must be replaced by \eqref{toshow2}.  
Using that ${\rm wno} \det \Psi = \kappa$, with all these adjustments 
in place we still arrive at ${\rm wno} \det Q_{2} = 0$ and hence 
$Q_{2}$ has no zeros in $\Pi_+$ and $G$ extends inside $\Pi_+$ as a 
Schur-class function.
The representation \eqref{LFTparamkap} follows from \eqref{Sparam'} as well as the equality
$\Psi = \psi(\Theta_{21}G+\Theta_{22}) Q_{2}$. Since $\Psi$ has no zeros 
in $\sigma(Z) \cap \sigma(W)$ while $\psi(\Theta_{21}G+\Theta_{22})$ 
and $Q_2$ are analytic on all of $\Pi_+$, we see that $\psi(\Theta_{21}G+\Theta_{22})$ has no zeros in  
$\sigma(Z) \cap \sigma(W)$ as well.

\smallskip

Conversely, for any $G\in\mathcal S^{p\times m}(\Pi_+)$ such that $\psi(\Theta_{21}G+\Theta_{22})$ has 
no zeros on $\sigma(Z)\cup\sigma(W)$, we let
$$
\begin{bmatrix}S_1 \\ S_2\end{bmatrix}=\Theta\begin{bmatrix}G \\ I_m\end{bmatrix},\quad \Psi=\psi S_2, 
    \quad S=S_1S_2^{-1}, 
$$
so that 
$$
\begin{bmatrix}S \\ I_m\end{bmatrix}\psi^{-1}\Psi=\Theta\begin{bmatrix}G \\ I_m\end{bmatrix}.
$$
Since $\Theta$ is $J$-unitary on $i\mathbb R$ and $G$ is a Schur-class, 
it follows that $S(\lambda)$ is contractive for almost all $\lam\in i\mathbb R$.
Since $\det \Psi$ has no zeros on $\sigma(Z)\cup\sigma(W)$ and has $\kappa$ zeros in $\Pi_+$, due to 
the equalities
$$
{\rm wno}\det \Psi={\rm wno}\det \psi+{\rm wno}\det\Theta_{22}+
{\rm wno}\det(\Theta_{22}^{-1}\Theta_{21}G+I)=\kappa
$$
we see that $S$  satisfies the interpolation conditions \eqref{BTOAint1}--\eqref{BTOAint3} 
by the criterion \eqref{sol-rep'}
and has total pole multiplicity at most $\kappa$ in $\Pi_+$. However, since  
$\nu_-(\bGamma_{\mathfrak D})=\kappa$,
by the part of the sufficiency criterion already proved we know that $S$ must 
have at least $\kappa$ poles in $\Pi_+$.  Thus $S$ has exactly 
$\kappa$ poles in $\Pi_{+}$ and therefore is in the $\mathcal 
S_\kappa^{p\times m}(\Pi_{+})$-class.

\smallskip

\subsection{The Fundamental Matrix Inequality approach for the 
generalized Schur-class setting} 
The Fundamental Matrix Inequality method extends to the present setting as follows. 
As in the definite case,
we extend the interpolation data by an arbitrary finite set of 
additional full-matrix-value interpolation conditions to conclude that the kernel 
$\bGamma_{\mathfrak D}(z,\zeta)$ 
defined as in \eqref{sep8c} has at most $\kappa$ negative squares in 
$\Omega_S \setminus  \sigma(W)$. 
Since the constant block
(the matrix $\Gamma_{\mathfrak D}$) has $\kappa$ negative eigenvalues (counted with multiplicities), 
it follows that 
${\rm sp}_-\bGamma_{\mathfrak D}(z,\zeta)=\kappa$ which holds if and only if the Schur complement 
of $\Gamma_{\mathfrak D}$ in \eqref{sep8c} is a positive kernel on $\Omega_S 
\setminus \sigma(W)$: 
$$
\frac{I_{p} - S(z)S(\zeta)^*}{z+\overline{\zeta}}-
\begin{bmatrix}I_p & -S(z)\end{bmatrix}{\bf C}(z I-{\bf A})^{-1}\Gamma_{\mathfrak D}^{-1}
(\overline\zeta I-{\bf A}^*)^{-1}{\bf C}^*\begin{bmatrix}I_p \\ -S(\zeta)^*\end{bmatrix}\succeq 0.
$$
As in Section \ref{S:FMI}, the latter positivity condition can be written in the form \eqref{sep8d} 
(all we need is formula \eqref{ad1} which still holds true) and eventually, implies equality 
\eqref{sep8g} for some $G\in\mathcal S^{p\times m}(\Pi_+)$, which in turn, implies 
the representation \eqref{LFTparamkap}. However, establishing the necessity of the condition 
\eqref{LFTparampar} 
requires a good portion of extra work. Most of the known proofs are still based the Argument Principle 
(the winding number computations \cite{BGR} or the operator-valued 
version of Rouch\'e's theorem \cite{GS}).
For example, it can be shown that if $K$ is a $p\times m$ matrix-valued polynomial 
satisfying interpolation conditions 
\eqref{BTOAint1}--\eqref{BTOAint3} and if $\varphi$ is the inner function given 
(analogously to \eqref{sep1}) by
$$
\varphi(z)=I_p-X^*(zI+Z^*)^{-1}\widetilde{P}^{-1}X,
$$
where the positive definite matrix $\widetilde{P}$ is uniquely defined from the Lyapunov equation
$\widetilde{P}Z+Z^*\widetilde{P}=XX^*$, then the matrix function 
\begin{equation}
\Sigma:=\begin{bmatrix}\Sigma_{11} & \Sigma_{12}\\ \Sigma_{21} & \Sigma_{22}\end{bmatrix}=
\begin{bmatrix}\varphi^{-1} & -\varphi^{-1}K \\ 0 & \psi\end{bmatrix}
\begin{bmatrix}\Theta_{11} & \Theta_{12}\\ \Theta_{21} & \Theta_{22}\end{bmatrix}
\label{nov2}
\end{equation}
is analytic on $\Pi_+$. Let us observe that by the formulas \eqref{oct1}, \eqref{bigLySyl} and well known
properties of determinants,
\begin{align}
\det\Theta(\lam)&=\det\left(I- \bC (\lam I - \bA)^{-1} \bGamma_{\mathfrak D}^{-1} \bC J\right)\notag\\
&=\det\left(I- \bC J\bC (\lam I - \bA)^{-1} \bGamma_{\mathfrak D}^{-1}\right)\notag\\
&=\det(\bGamma_{\mathfrak D} (\lam I - \bA)+\bGamma_{\mathfrak D}\bA+\bA^*\bGamma_{\mathfrak D})\cdot
\det((\lam I - \bA)^{-1} \bGamma_{\mathfrak D}^{-1})\notag\\
&=\frac{\det(\lam I+\bA^*)}{\det(\lam I - \bA)}=\frac{\det(\lam I-Z)\det(\lam I+W^*)}
{\det(\lam I +Z^*)\det(\lam I-W)}.\notag
\end{align}
Similar computations show that 
$$
\det \psi(\lam)=\frac{\det(\lam I-W)}{\det(\lam I +W^*)},\quad 
\det \varphi(z)=\frac{\det(\lam I-Z)}{\det(\lam I +Z^*)}.
$$
Combining the three latter equalities with \eqref{nov2} gives $\det \Sigma(\lam)\equiv 1\neq 0$. 
Therefore, for 
$G\in\mathcal S^{p\times m}$, the total pole multiplicity of the function 
$$
\Upsilon=(\Sigma_{11}G+\Sigma_{12})(\Sigma_{21}G+\Sigma_{22})^{-1}
$$
is the same as the number of zeros of the denominator 
$$\Sigma_{21}G+\Sigma_{22}=
\psi(\Theta_{21}G+\Theta_{22}),
$$
that is $\kappa$, by the winding number argument.
On the other hand, since 
\begin{equation}  \label{Takagi-Sarason}
 S=K+\varphi \Upsilon\psi,
 \end{equation}
as can be seen from \eqref{LFTparamkap} and \eqref{nov2},
the total pole multiplicity of $S$ equals $\kappa$ if no poles of $\Upsilon$ 
occur at zeros of $\varphi$ and 
$\Psi$, that is, in $\sigma(Z)\cup\sigma(W)$.  We note that the form 
\eqref{Takagi-Sarason} where $K,\varphi,\psi$ are part of the data 
and $\Upsilon$ is a free meromorphic function with no poles on $i 
{\mathbb R}$ but $\kappa$ poles in $\Pi_{+}$ (including possibly at 
points of $\sigma(W) \cup \sigma(Z)$) corresponds to a 
variant of the interpolation problem 
\eqref{BTOAint1}, \eqref{BTOAint2}, \eqref{BTOAint3} sometimes 
called the {\em Takagi-Sarason problem} (see \cite[Chapter 19]{BGR}, \cite{Bolo04}).  It 
turns out that discarding the side-condition \eqref{LFTparampar} on 
the Schur-class free-parameter function $G$ leads to a 
parametrization of the set of all solutions of the Takagi-Sarason 
problem.

\subsection{Indefinite kernels and reproducing kernel Pontryagin 
spaces}  \label{S:RKHSkap}
From the formula \eqref{ad1} for $K_{\Theta,J}$, we see from 
the observability  assumption on $(\bC, \bA)$ (equivalently, the 
observability and controllability assumptions on $(\sbm{V \\ U}, W)$ 
and $(Z, \begin{bmatrix} X & -Y \end{bmatrix})$) that 
$$
 \nu_{-}(\bGamma_{\mathfrak D}) = {\rm sq}_{-}(K_{\Theta,J}).
$$
By the general theory of reproducing kernel Hilbert spaces sketched 
in Section \ref{S:RKHSkap}, it follows that $\cH(K_{\Theta,J})$ is a Pontryagin space with 
negative index $\nu_{-}(\cH(K_{\Theta,J})$ equal to the number of 
negative eigenvalues of $\bGamma_{\mathfrak D}$:
$$
  \nu_{-}(\cH(K_{\Theta,J})) = \nu_{-}(\bGamma_{\mathfrak D}).
$$
We conclude that the formula for $\kappa$ in statement (1) agrees with that in statement (2)
in Theorem \ref{T:BTOA-NPkap}.

\subsection{The Grassmannian/Kre\u{\i}n-space approach for the 
generalized Schur-class setting} \label{S:Grass-kappa}

The Grassmannian approach extends to to the present setting as follows. 
The suitable analog of Lemma \ref{L:maxneg} is the following:
\begin{lemma}  \label{L:maxnegkap}  
Suppose that $\cM$ is a closed subspace of a Kre\u{\i}n-space $\cK$ 
such that the $\cK$-relative orthogonal 
complement $\cM^{[\perp]}$ has negative signature equal $\kappa$. If $\cG$ is a negative
subspace of $\cM$, then $\cG$ has codimension at least $\kappa$ in any maximal negative 
subspace of $\cK$. Moreover, the codimension of such a $\cG$ in any maximal negative 
subspace of $\cK$ is equal to $\kappa$ if and only if $\cG$ is a maximal negative
subspace of $\cM$. 
\end{lemma}

Let us now assume that we are given a $\Pi_{+}$-admissible 
interpolation data set $\mathfrak D$ with $\bGamma_{\mathfrak D}$ 
invertible. Then $\cM_{\mathfrak D}$  given by \eqref{cMrep} is a regular subspace of the 
Kre\u{\i}n space $L^{2}_{p+m}(i {\mathbb R})$ with the $J (= 
\sbm{I_{p} & 0 \\ 0 & - I_{m}}$)-inner product.  

With Lemma \ref{L:maxnegkap} in hand, we argue that 
$\nu_{-}(\cM_{\mathfrak D}^{[\perp \cK]})  \ge \kappa$ is necessary for the existence of 
$\cS_{\kappa}^{p \times m}(\Pi_{+})$-functions $S$ analytic on 
$\sigma(Z) \cup \sigma(W)$ satisfying the interpolation conditions
\eqref{BTOAint1}, \eqref{BTOAint2}, \eqref{BTOAint3}.  

\smallskip

\noindent
\textbf{Proof of necessity for the generalized Schur-class setting.}
If $S \in 
\cS_{\kappa'}^{p \times m}(\Pi_{+})$ is a solution of the 
interpolation conditions with $\kappa' \le \kappa$, then as in Section 
\ref{S:statekap}, there is a $m \times m$-matrix 
function $\Psi$ with $\det \Psi$ having no zeros in $\sigma(Z) \cup 
\sigma(W)$ and having $\kappa$ zeros in $\Pi_{+}$ so that the 
subspace $\cG_{S}: = \sbm{ S  \\ I_{m} } \psi^{-1} \Psi H^{2}_{m}(\Pi_{+})$ 
satisfies the inclusion \eqref{sol-rep'}.  We note that then 
$\cG_{S}$ is a negative subspace of $\cK$ and the fact that $\Psi$ has 
$\kappa$ zeros means that $\cG_{S}$ has codimension $\kappa$ in a 
maximal negative subspace of $\cK: = \sbm{ L^{2}_{p}(i {\mathbb R}) \\ 
H^{2}_{m}(\Pi_{+})}$.  As $\cG_{S}$ is also a subspace of 
$\cM_{\mathfrak D}$, it follows  by Lemma \ref{L:maxnegkap} that 
the negative signature of $\cM_{\mathfrak D}^{[ \perp \cK]}$ must be 
at least $\kappa$.  Thus $\nu_{-}(\cM_{\mathfrak D}^{[ \perp \cK]}) 
\ge \kappa$ is necessary for the existence of a solution $S$ of the 
interpolation problem in the class $\cS_{\kappa}^{p \times 
m}(\Pi_{+})$.  As part of the sufficiency direction, we shall show 
that conversely, if $\kappa = \nu_{-}(\cM^{[ \perp \cK]})$, then we 
can always find solutions $S$ of the interpolation conditions in the 
class $\cS_{\kappa}^{p \times m}(\Pi_{+})$.  This establishes 
the formula in statement (2) of Theorem \ref{T:BTOA-NPkap} as the 
minimal $\kappa$   such that solutions of the interpolation 
conditions can be found in class $\cS_{\kappa}^{p \times m}(\Pi_{+})$.

\smallskip

\noindent
\textbf{Proof of sufficiency for the generalized Schur-class setting.}
Let us suppose that $\bGamma_{\mathfrak D}$ is invertible and hence 
that $\cM_{\mathfrak D}$ is a regular subspace of the Kre\u{\i}n 
space $L^{2}_{p+m}(i {\mathbb R})$ with the $J = \sbm{ I_{p} & 0 \\ 0 
& -I_{m}}$-inner product.  By the results of 
\cite{BH}, there is a $J$-phase function $\Theta$ so that the 
Beurling-Lax representation \eqref{cMBLrep} holds (we avoid using the 
formula \eqref{oct1} for $\Theta$ at this stage).
We now assume that $\cM_{\mathfrak D}^{[ \perp 
\cK]}$ has negative signature $\nu_{-}(\cM_{\mathfrak D}^{[ \perp 
\cK]})$ equal to $\kappa$. We wish to verify the linear-fractional 
parametrization \eqref{LFTparamkap}--\eqref{LFTparampar} for the set 
of all $\cS_{\kappa}^{p\times m}(\Pi_{+})$-class solutions of the 
interpolation conditions.

Suppose first that $S$ is any $\cS_{\kappa}^{p \times m}(\Pi_{+})$-class 
solution of the interpolation conditions. By the graph-space criterion
for such solutions, there is a $m \times m$-matrix valued function 
$\Psi$ analytic on $\Pi_{+}$ with $\det \Psi$ having $\kappa$ zeros 
but none in $\sigma(Z) \cup \sigma(W)$, so that \eqref{sol-rep'} holds.
But then 
$$
\cG_{S}: = \sbm{ S \\ I_{m} } \psi^{-1} \Psi 
H^{2}_{m}(\Pi_{+})
$$ 
is a shift-invariant negative subspace of $\cK$
contained in $\cM_{\mathfrak D}$ and having codimension $\kappa$ in a maximal 
negative subspace of $\cK$.  It now follows from Lemma 
\ref{L:maxnegkap} that $\cG_{S}$ is maximal negative as a subspace of 
$\cM_{\mathfrak D}$.  As $\cG_{S}$ is also shift-invariant and 
multiplication by $\Theta$ is a Kre\u{\i}n-space isomorphism from 
$H^{2}_{p+m}(\Pi_{+})$ onto $\cM_{\mathfrak D}$, it follows that
$\cG_{S}$ is the image under multiplication by $\Theta$ of a 
shift-invariant $J$-maximal negative subspace of 
$H^{2}_{p+m}(\Pi_{+})$, i.e.,
\begin{equation}  \label{Srep'}
  \cG_{S}: = \begin{bmatrix} S \\ I_{m} \end{bmatrix} \cdot \psi^{-1} \Psi 
\cdot  H^{2}_{m}(\Pi_{+}) = \Theta \cdot \begin{bmatrix} G \\ I_{m} 
\end{bmatrix} \cdot H^{2}_{m}(\Pi_{+})
\end{equation}
for a $\cS^{p \times m}(\Pi_{+})$-class function $G$.  From the fact 
that $\Psi$ has no zeros in $\sigma(Z) \cup \sigma(W)$ one can read 
off from \eqref{Srep'} that $\psi (\Theta_{21} G + \Theta_{22})$ has no zeros in 
$\sigma(Z) \cup \sigma(W)$ and from the representation \eqref{Srep'} 
the linear-fractional representation \eqref{LFTparamkap} follows as 
well. From the subspace identity \eqref{Srep'} one can also read off 
that there is a $m \times m$ matrix function $Q$ with $Q^{\pm 1} \in 
H^{\infty}_{m \times m}(\Pi_{+})$ such that
$$
S \psi^{-1} \Psi = (\Theta_{11} G + \Theta_{12}) Q\quad\mbox{and}\quad
\psi^{-1} \Psi = (\Theta_{21} G + \Theta_{22})Q.
$$
Solving the second equation for $Q$ then gives
$$
  Q = (\Theta_{22} G + \Theta_{22})^{-1} \psi^{-1} \Psi.
$$
Substituting this back into the first equation and then solving for 
$S$ leads to the linear-fractional representation \eqref{LFTparamkap} 
for $S$.

\smallskip

Let now $G$ be any Schur-class function satisfying the additional 
constraint \eqref{LFTparampar}.  Since multiplication by $\Theta$ is 
a Kre\u{\i}n-space isomorphism from $H^{2}_{p+m}(\Pi_{+})$ to 
$\cM_{\mathfrak D}$ and $\sbm{ G \\ I_{m}} H^{2}_{m}(\Pi_{+})$  is a 
maximal negative shift-invariant subspace of $\cM_{\mathfrak D}$, it follows that 
$\Theta \cdot \sbm{ G \\ I_{m}} H^{2}_{m}(\Pi_{+})$ is maximal 
negative as a subspace of $\cM_{\mathfrak D}$.
By Lemma \ref{L:maxnegkap}, it follows that $\Theta \cdot \sbm{ G \\ 
I_{m}} H^{2}_{m}(\Pi_{+})$ has codimension $\kappa = 
\nu_{-}(\cM_{\mathfrak \cD}^{[ \perp] \cK})$ in a maximal negative 
subspace of $\cK$.  As $\Theta \cdot \sbm{ G \\ 
I_{m}} H^{2}_{m}(\Pi_{+})$ is also shift-invariant, it follows that
there must be a contractive matrix function $S$ on the unit circle 
and a  bounded analytic $m \times m$-matrix function $\Psi$ on 
$\Pi_{+}$ such that $\Psi$ has exactly $\kappa$ zeros in $\Pi_{+}$
and $\Psi$ is bounded and invertible on $i {\mathbb R}$ so that
\begin{equation}  \label{Srep}
\begin{bmatrix} S \\ I_{m} \end{bmatrix} \cdot \psi^{-1} \Psi 
   \cdot H^{2}_{m}(\Pi_{+}) = \Theta \cdot \begin{bmatrix} G \\ I 
\end{bmatrix} \cdot  H^{2}_{m}(\Pi_{+}).
\end{equation}
In particular, $\psi^{-1} \Psi\cdot  H^{2}_{m}(\Pi_{+}) \subset 
(\Theta_{21} G + \Theta_{22}) \cdot H^{2}_{m}(\Pi_{+})$, so
there is a $Q \in H^{\infty}_{p \times m}(\Pi_{+})$ so that
$ \psi^{-1} \Psi = (\theta_{21} G + \Theta_{22}) Q$, i.e., so that
$$
  \Psi = \psi (\Theta_{21} G + \Theta_{22}) Q.
$$
As $\psi (\Theta_{21} G + \Theta_{22})$ has no zeros in $\sigma(Z) 
\cup \sigma(W)$ by assumption, it follows that none of the zeros of 
$\Psi$ are in $\sigma(Z) \cup \sigma(W)$.  By the criterion 
\eqref{sol-rep'} for $\cS_{\kappa'}^{p \times m}(\Pi_{+})$-class 
solutions of the interpolation conditions with $\kappa' \le \kappa$, 
we read off from \eqref{Srep} that $S$ so constructed is a 
$\cS_{\kappa'}^{p \times m}(\Pi_{+})$-class solution of the 
interpolation conditions for some $\kappa' \le \kappa$. However, from 
the proof of the necessity direction already discussed, it follows 
that necessarily $\kappa' \ge \kappa$.  Thus $S$ so constructed is a 
$\cS_{\kappa'}^{p \times m}(\Pi_{+})$-class solution of the 
interpolation conditions.  The subspace identity \eqref{Srep} leads to
the formula \eqref{LFTparamkap} for $S$ in terms of $G$ just as in 
the previous paragraph.

\begin{remark} \label{R:summary}
We conclude that the Grassmannian approach extends to the generalized 
Schur-class setting.  As in the classical Schur-class case,  
one can avoid the elaborate winding-number argument used in Section 
\ref{S:statekap} by using Kre\u{\i}n-space geometry (namely, the fact 
the a Kre\u{\i}n-space isomorphism maps maximal negative 
subspaces to maximal negative subspaces combined with Lemma 
\ref{L:maxnegkap}),  unlike the story for the Fundamental Matrix Inequality 
Potapov approach, which avoids the winding number argument in an 
elegant way for the definite case but appears to still require such 
an argument for the indefinite generalized Schur-class setting.
\end{remark}

\subsection{State-space versus Grassmannian/Kre\u{\i}n-space-geometry solution criteria in the generalized 
Schur-class setting}  \label{S:synthesiskap}
The work of the previous subsections shows that each of conditions 
(1) and (2) in Theorem \ref{T:BTOA-NPkap} is equivalent to the 
existence of $\cS_{\kappa}^{p \times m}(\Pi_{+})$-class solutions f 
the interpolation conditions \eqref{BTOAint1}--\eqref{BTOAint3}, 
and that condition (2) is equivalent to condition 
(1).  It follows that conditions (1), (2), (3) are all equivalent to 
each other.  Here we wish to see this latter fact directly in a more 
concrete from, analogously to what is done in Section 
\ref{S:synthesis} above for the classical Schur-class setting.

\smallskip

As in Section \ref{S:synthesis}, we impose an assumption a little 
stronger than the condition that $\bGamma_{\mathfrak D}$ be 
invertible, namely, the Nondegeneracy Assumption:  {\em 
$\cM_{\mathfrak D}$, $\cM_{\mathfrak D} \cap H^{2}_{p+m}(\Pi_{+})$, 
and $\cM_{\mathfrak D} \cap H^{2}_{p+m}(\Pi_{-})$ are all regular 
subspaces of $L^{2}_{p+m}(i {\mathbb R})$ (with the $J\, (= \sbm{ I_{p} 
& 0 \\ 0 & -I_{m} })$-inner product).}  Then Lemmas \ref{L:cMDperp} and 
\ref{L:decom} go through with no change.  Lemma \ref{L:MperpK} goes 
through, but with the {\em in particular} statement generalized to 
the following (here $\nu_{-}(\cL)$ refers to negative signature of 
the given subspace $\cL$ of $L^{2}_{p+m}(i {\mathbb R})$ with respect 
to the $J$-inner product):
\begin{itemize}
    \item {\em In particular, $\nu_{-}(\cM_{\mathfrak D}^{[ \perp 
    \cK]}) = \kappa$ if and only if}
  $$
  \nu_{-}(\cM_{\mathfrak D}^{[\perp \cK]})_{0}) = \kappa
  $$
{\em if and only if}
$$
  \nu_{-}({\mathbb G}_{T^{[*]}}) + \nu_{-}((\cM_{\mathfrak D}^{[ 
  \perp]})_{1}) = \kappa.
$$
    \end{itemize}
Lemma \ref{L:posnegsubspaces} has the more general form:
\begin{enumerate}
    \item $\nu_{-}({\mathbb G}_{T^{[*]}}) =\kappa$ {\em if and only if }
    $\nu_{-}(I + T T^{[*]}) = \kappa$  ({\em where $I + T T^{[*]}$ is 
    considered as an operator on } $\operatorname{Ran} J (\cC_{Z, 
    \sbm{ X & -Y}})^{*}$).
    
    \item $\nu_{-}((\cM_{\mathfrak D}^{[ \perp]})_{1}) = 
   \nu_{-}(\operatorname{Ran} \cO_{\sbm{V \\ U}, W})$.
   
   \item $\nu_{-}(\widehat {\mathbb G}_{T}) = \nu_{-}( I + T^{[*]} 
   T)$ ({\em where } $I + T^{[*]}T$ {\em is considered as an operator 
   on } $\operatorname{Ran} \cO_{\sbm{ V \\ U}, W}$).
   
   \item $\nu_{+}(\cM_{{\mathfrak D},1}) = \nu_{-}(\operatorname{Ran} 
   J (\cC_{Z, \sbm{ X & -Y}})^{*})$.
    \end{enumerate}
 Lemma \ref{L:GammaD-factored} is already in general form but its 
 corollary, namely Lemma \ref{L:GammaDpos}, can be given in a more 
 general form:
 \begin{itemize}
     \item {\em The following conditions are equivalent:}
\begin{enumerate}
    \item $\nu_{-}(\bGamma_{\mathfrak D}) = \kappa$.
    
     \item $\nu_{+}(\operatorname{Ran} \cO_{\sbm{V \\ U}, W}) + 
    \nu_{-}({\mathbb G}_{T^{[*]}}) = \kappa.$
    
    \item $\nu_{-}( \operatorname{Ran} J (\cC_{Z, \sbm{ X & 
    -Y}})^{*} )  + \nu_{+}(\widehat{\mathbb G}_{T}) = \kappa.$
 \end{enumerate}
 \end{itemize}
 Putting the pieces together, we have the following chain of 
 reasoning.  By the generalized version of Lemma \ref{L:MperpK}, we 
 have
 \begin{equation}   \label{1}
 \nu_{-}(\cM_{\mathfrak D})^{[ \perp \cK]} = \nu_{-}({\mathbb 
 G}_{T^{[*]}}) + \nu_{-}((\cM_{\mathfrak D}^{[ \perp]})_{1})
 \end{equation}
 where, by the generalized version of Lemma \ref{L:posnegsubspaces} 
 part (2),
 $$
 \nu_{-}((\cM_{\mathfrak D}^{[\perp]})_{1}) = 
 \nu_{-}(\operatorname{Ran} \cO_{\sbm{V \\U}, W}).
 $$
 Thus \eqref{1} becomes
 $$
 \nu_{-}(\cM_{\mathfrak D})^{[ \perp \cK]} = \nu_{-}({\mathbb G}_{T^{[*]}}) 
 + \nu_{-}(\operatorname{Ran} \cO_{\sbm{V \\U}, W}).
 $$
 By (1) $\Leftrightarrow$ (2) in the generalized Lemma 
 \ref{L:GammaD-factored}, we get
 $$
 \nu_{-}(\cM_{\mathfrak D}^{[ \perp \cK]}) = 
 \nu_{-}(\bGamma_{\mathfrak D})
 $$
 which gives us (1) $\Leftrightarrow$ (2) in Theorem 
 \ref{T:BTOA-NPkap}.
 
\smallskip

 To give a direct proof of (1) $\Leftrightarrow$ (3) in Theorem 
 \ref{T:BTOA-NPkap}, we note the concrete identification 
 \eqref{HKThetaJid} of the space $\cH(K_{\Theta,J})$ (with 
 $J$-inner product on $\operatorname{Ran} (P^{J}H^{2}_{p+m}(\Pi_{+}) 
 - P^{J}_{\cM_{\mathfrak D}})$ which again leads to the more compact 
 identification \eqref{HKThetaJ} from which we immediately see that
 $$
 \nu_{-}(\cH(K_{\Theta,J})) = \nu_{-}(\operatorname{Ran} J(\cC_{Z, 
 \sbm{ X & -Y}})^{*}) + \nu_{+}(\widehat {\mathbb G}_{T}).
 $$
 By (1) $\Leftrightarrow$ (3) in the generalized Lemma 
 \ref{L:GammaDpos}, this last expression is equal to 
 $\nu_{-}(\bGamma_{\mathfrak D})$, and we have our more concrete 
 direct proof of the equivalence of conditions (1) and (3) in Theorem 
 \ref{T:BTOA-NPkap}.

\end{document}